\documentclass[11pt]{article}

\usepackage[utf8]{inputenc}
\usepackage[margin=0.75in]{ geometry }
\usepackage{ kpfonts }

\usepackage{ comment }
\usepackage{ amsmath }
\usepackage{ amssymb }
\usepackage{ amsthm }
\usepackage{ thmtools }

\usepackage{ xparse }
\usepackage{ xstring }

\usepackage{ accents }
\usepackage{ mathtools }

\usepackage{ tikz }
\usepackage{ tikz-cd }
\usepackage{ float }
\usepackage[boxed]{ algorithm2e }

\usepackage{ enumitem }
\usepackage{ svg }
\usepackage{ graphicx }
\usepackage{ scalerel }
\usepackage{ relsize }
\usepackage{ multicol }
\usepackage{ multirow }
\usepackage{ xcolor }
\usepackage{ wrapfig } 

\usepackage{ authblk }
\usepackage[square, numbers]{ natbib }
\usepackage{fancyhdr}

\usepackage{ standalone }

\usepackage{ navigator }
\usepackage[bookmarksnumbered=true, colorlinks=true, pdftitle={Rewriting Graphically with Symmetric Traced Monoidal Categories}, pdfauthor={George Kaye}]{ hyperref }
\usepackage{ bookmark }
\usepackage{ doi }

\usepackage{ macros }

\svgpath{./svgs/.}

\ExplSyntaxOn
\NewDocumentCommand{\labelstyle}{m}
    {
        \grill_convert_name:n { #1 }
        {\l_Grill_label_tl}
    }

\tl_new:N \l_Grill_label_tl

\cs_new:Npn \grill_convert_name:n #1 
    {
        \tex_lowercase:D { \tl_set:Nn \l_Grill_label_tl { #1 } }
        \tl_replace_All:Nnn \l_Grill_label_tl { ~ } { - }
    }
\ExplSyntaxOff

\makeatletter
\newcommand*{\theorembookmark}{%
  \bookmark[
    dest=\@currentHref,
    rellevel=1,
    keeplevel,
    italic,
  ]{%
    \thmt@thmname\space\csname the\thmt@envname\endcsname
    \ifx\@currentlabelname\@empty
    \else
      \space(\@currentlabelname)%
    \fi
  }%
}
\makeatother

\declaretheorem[title=Theorem, postheadhook=\theorembookmark]{theorem}
\declaretheorem[title=Definition, postheadhook=\theorembookmark, sibling=theorem]{definition}
\declaretheorem[title=Lemma, postheadhook=\theorembookmark, sibling=theorem]{lemma}
\declaretheorem[title=Example, postheadhook=\theorembookmark, sibling=theorem]{example}
\declaretheorem[title=Remark, postheadhook=\theorembookmark, sibling=theorem]{remark}
\declaretheorem[title=Proposition, postheadhook=\theorembookmark, sibling=theorem]{proposition}

\title{\huge Rewriting Graphically with Symmetric~Traced~Monoidal~Categories}
\author{\Large George Kaye}
\affil{School of Computer Science, University of Birmingham, UK}
\date{March 19, 2021}

\begin{document}

\maketitle

\begin{abstract}
  \noindent We examine a variant of hypergraphs that we call \emph{interfaced linear hypergraphs}, with the aim of creating a sound and complete graphical language for symmetric traced monoidal categories (STMCs) suitable for graph rewriting. 
  In particular, we are interested in rewriting for categorical settings with a \emph{Cartesian} structure, such as digital circuits.
  These are incompatible with previous languages where the trace is constructed using a compact closed or Frobenius structure, as combining these with Cartesian product can lead to degenerate diagrams.
  Instead we must consider an approach where the trace is constructed as an atomic operation.
  Interfaced linear hypergraphs are defined as regular hypergraphs in which each vertex is the source and target of exactly one edge each, equipped with an additional interface edge. The morphisms of a freely generated STMC are interpreted as interfaced linear hypergraphs, up to isomorphism (soundness). 
  Moreover, any linear hypergraph is the representation of a unique STMC morphism, up to the equational theory of the category (completeness).
  This establishes interfaced linear hypergraphs as a suitable combinatorial language for STMCs. 
  We then show how we can apply the theory of adhesive categories to our graphical language, meaning that a broad range of equational properties of STMCs can be specified as a graph rewriting system. 
  The graphical language of digital circuits is presented as a case study.
\end{abstract}

\section{Introduction}

Constructors, architects, and engineers have always enjoyed using blueprints, diagrams, floorplans and other kinds of graphical representations of their designs.
These are often more than simply illustrations aiding the understanding of a formal specification, they are the specification itself.
By contrast, in mathematics, diagrams have not been traditionally considered first-class citizens, although they are often used to help the reader \emph{visualise} a construction or a proof.
However, the development of new \emph{formal} diagrammatic languages for a variety of systems such as quantum communication and computation~\cite{coecke2018picturing}, computational linguistics~\cite{coecke2010mathematical} and signal-flow graphs~\cite{bonchi2014categorical,bonchi2015abstraction}, proved that diagrams can be used not just to aid understanding of proofs, but also to formulate proofs.
This formulation has multifaceted advantages, from enabling the use of graph-theoretical techniques to aid reasoning~\cite{ghica2017diagrammatic} to making the teaching of algebraic concepts to younger students less intimidating~\cite{ghica18mt}.

These graphical languages build on a mathematical infrastructure of (usually symmetric and strict) monoidal categories~\cite{joyal1991geometry}, and in particular compact closed categories~\cite{kelly1980coherence}.
Systems modelled by morphisms in a compact closed category have a general notion of \emph{interface port}, so that any two ports can be connected, provided the types match.
This allows compact closed categories to describe systems with a flexible and refined notion of \emph{causality}, such as quantum systems~\cite{kissinger2017causal} or games~\cite{castellan2016causality}.
In contrast, systems such as digital circuits have a stricter notion of causality, enforcing that connections may only happen between ports with the same type but opposite input-output polarities.
This requires a different kind of categorical setting, that of a \emph{symmetric traced monoidal category}~\cite{hasegawa2009traced}, or STMC.
These categories come equipped with an explicit construct (the trace) to model causal feedback loops.

\emph{String diagrams}~\cite{selinger2010survey} are becoming the established mathematical language of diagrammatic reasoning, whereby equal terms are usually interpreted as isomorphic (or isotopic) diagrams. 
While this is enough for reasoning about \textit{structural} properties, properties which have computational content require a \textit{rewriting} of the diagram. 
To make this possible, diagrams must be represented as combinatorial objects, such as graphs or hypergraphs, which have enough structure. 
The framework of \textit{adhesive categories} is of particular interest to us~\cite{lack2004adhesive}, as it implies that graph rewriting is always well-defined. 

Our main motivation is to fully formalise prior work on diagrammatic reasoning for digital circuits~\cite{ghica2016categorical,ghica2017diagrammatic}, for which we need a string diagram language of STMCs along with adhesive categorical infrastructure for rewriting.  
It might seem that this is a solved problem, as combinatorial languages for graph rewriting have already been studied as \textit{open graphs}~\cite{dixon2013open,kissinger2012pictures} and \textit{hypergraphs}~\cite{bonchi2016rewriting,zanasi2017rewriting,bonchi2018rewriting}, which satisfy soundness and completeness. 
However, the completeness theorem raises for us insurmountable technical problems, which we set to overcome in this paper. 
In \textit{loc.{}cit.{}} STMCs are constructed by embedding them into the more expressive setting of a SMC equipped with a \textit{Frobenius structure}, which induces a compact closed structure into which an STMC can be embedded. 
To reason about digital circuits we require the framework of \textit{dataflow categories}, which are STMCs in which the monoidal tensor is a Cartesian product~\cite{cazanescu1994feedback,hasegawa1997recursion}. 
It is the interaction between the diagonal morphism of the Cartesian product and the Frobenius structure which is problematic. 

In general it is well known that in compact closed categories finite products automatically become \textit{biproducts}~\cite{houston2008finite}. 
This is enough to compromise the construction as a setting for modelling digital circuits, which do not physically satisfy the equational properties of a biproduct. 
But the problem runs deeper, as the Frobenius structure itself is not compatible with Cartesian product, as seen in the diagram below:
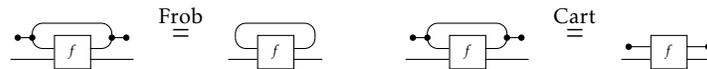
\begin{figure}[ht]
    \centering
    \begin{tikzpicture}[yscale=-1,x=1em,y=1.25em]
    
    \filldraw (0.5,0) circle (1pt);
    \draw (0.5,0) -- (1,0);
    \filldraw (1.0,0) circle (1pt);
    \draw [rounded corners] (1.0,0) -- (1.0,-0.4) -- (2,-0.4) -- (3.75,-0.4) -- (3.75,0);
    \draw [rounded corners] (1.0,0) -- (1.0,0.3) -- (1.75,0.3);
    \draw (0.25,0.6) -- (1.75,0.6);
    \node[draw, minimum height = 0.5em, minimum width = 1.25em, anchor = west] at (1.75,0.4){\scalebox{0.5}{$f$}};
    \filldraw (3.75,0) circle (1pt);
    \draw [rounded corners](3,0.3) -- (3.75,0.3) -- (3.75,0);
    \draw (3.75,0) -- (4.25,0);
    \filldraw (4.25,0) circle (1pt);
    \draw (3,0.6) -- (4.5,0.6);

\end{tikzpicture}
    \,\,
    \raisebox{1em}{$\stackrel{\text{Frob}}{=}$}
    \,\,
    \begin{tikzpicture}[yscale=-1,x=1em,y=1.25em]
    
    \draw [rounded corners] (1.0,0) -- (1.0,0.3) -- (1.75,0.3);
    \draw [rounded corners] (1.0,0) -- (1.0,-0.4) -- (2,-0.4) -- (3.75,-0.4) -- (3.75,0);
    \draw [rounded corners] (3.75,0) -- (3.75, 0.3) -- (3, 0.3);
    \draw (0.75,0.6) -- (1.75,0.6);
    \draw (3, 0.6) -- (4,0.6);
    \node[draw, minimum height = 0.6em, minimum width = 1.25em, anchor = west] at (1.75,0.4){\scalebox{0.5}{$f$}};

\end{tikzpicture}
    \,\qquad\,
    \begin{tikzpicture}[yscale=-1,x=1em,y=1.25em]
    
    \filldraw (0.5,0) circle (1pt);
    \draw (0.5,0) -- (1,0);
    \filldraw (1.0,0) circle (1pt);
    \draw [rounded corners] (1.0,0) -- (1.0,-0.4) -- (2,-0.4) -- (3.75,-0.4) -- (3.75,0);
    \draw [rounded corners] (1.0,0) -- (1.0,0.3) -- (1.75,0.3);
    \draw (0.25,0.6) -- (1.75,0.6);
    \node[draw, minimum height = 0.5em, minimum width = 1.25em, anchor = west] at (1.75,0.4){\scalebox{0.5}{$f$}};
    \filldraw (3.75,0) circle (1pt);
    \draw [rounded corners](3,0.3) -- (3.75,0.3) -- (3.75,0);
    \draw (3.75,0) -- (4.25,0);
    \filldraw (4.25,0) circle (1pt);
    \draw (3,0.6) -- (4.5,0.6);

\end{tikzpicture}
    \,\,
    \raisebox{1em}{$\stackrel{\text{Cart}}{=}$}
    \,\,
    \begin{tikzpicture}[yscale=-1,x=1em,y=1.25em]
    
    \filldraw (1,0.25) circle (1pt);
    \draw [rounded corners] (1.0,0.25) -- (1.75,0.25);
    \draw (0.75,0.6) -- (1.75,0.6);
    \node[draw, minimum height = 0.5em, minimum width = 1.25em, anchor = west] at (1.75,0.4){\scalebox{0.5}{$f$}};
    \filldraw (3.75,0.25) circle (1pt);
    \draw [rounded corners](3,0.25) -- (3.75,0.25);
    \draw (3,0.6) -- (4,0.6);

\end{tikzpicture}
\end{figure}

On the left, the Frobenius structure equates the splitting and joining of the wires with a feedback loop, implementing a trace structure. 
On the right, the splitting of the wires copies the co-unit of the Frobenius co-monoid, resulting in a degenerate circuit. 
To solve this problem we need to define the trace structure directly, and prove soundness and definability for these direct definitions. 

Besides the major problem above, there are some small technical issues with hypergraphs that we solve by reintroducing the concept of homeomorphism similar to that used in framed point graphs~\cite{kissinger2012pictures}. 
This allows us to represent the trace of the identity, which is not well-formed in vanilla hypergraphs as it is a closed loop of wires.
It also means we can identify a matching of a subgraph $F$ in a graph $\trace{x}{F}$ by using a monomorphism, which is essential for performing double pushout (DPO) graph rewriting.

The primary contributions of this paper are therefore as follows: we refine the definition of hypergraphs in \cite{bonchi2016rewriting} in order to define a \emph{sound} and \emph{complete} graphical language for STMCs that does not become degenerate in the presence of Cartesian structure.
We show that this language can be used with the framework of adhesive categories, so any additional axioms can be expressed as graph rewrite rules without any ambiguity.
This allows us to make the proofs in~\cite{ghica2017diagrammatic} rigorous.

\subsection{Structure of the report}

The structure of the report is as follows. 
In \S\ref{sec:monoidal-categories} we recap the required background on monoidal categories, and in particular Cartesian and symmetric traced monoidal categories, the primary focus of our work.
In \S\ref{sec:hypergraphs} we introduce a standard definition of hypergraphs, and then refine this to obtain \emph{linear hypergraphs}, which are motivated by our study of string diagrams. 
\S\ref{sec:operations} details several hypergraph constructs and operations that will be of use to us. 
We then use these ingredients in \S\ref{sec:soundness} to show that we can represent morphisms in a free PROP (a category of PROducts and Permutations, where objects are natural numbers) as hypergraphs.
We take the opposite perspective in \S\ref{sec:completeness}, to show that we can also recover categorical terms from hypergraphs, enabling us to conclude both soundness and completeness. 
In \S\ref{sec:graph-rewriting} we study graph rewriting, a useful application of our graphical language, and in \S\ref{sec:case-study} follow with a case study into the axioms related to digital circuits. 
Finally in \S\ref{sec:generalisation} we generalise our approach to consider terms from any STMC, not just PROPs. 
The finer details of some of the more bureaucratic proofs can be found in the appendices.

\subsection{Notation}
Let $|X|$ be the cardinality of a set $X$.
We write $[n] \subset \nat$ as the subset of the natural numbers containing $0, 1, \cdots , n-1$. 
For two sets $X$ and $Y$, let $X + Y$ be their disjoint union and $X - Y$ be the relative complement of $Y$ in $X$. 
Let $(X, \torder[X])$ be a totally ordered set, where usually we will just write the carrier $X$. 
We take the convention that the order is defined by the order the elements are written, i.e. in $\{x,y,z\}$, $x < y < z$ and for $\{x,y\} \cup \{z,w\}, x < y < z < w$.
We use $\proj{i}$ as a `projection' function to denote the $i$th element of a totally ordered set. 
For two functions $\morph{f}{X}{Y}$ and $\morph{g}{U}{V}$, we denote as $\morph{f + g}{X + U}{Y + V}$ their disjoint union that acts as $f$ on elements of $X$ and as $g$ on elements of $U$. 

\section{Monoidal categories}\label{sec:monoidal-categories}

We begin by recapping the concepts of monoidal categories.
A category $\mcc$ is a collection of objects $A,B,C,...$ with morphisms $f,g,h,...$ between them. 
A morphism $f$ between objects $A$ and $B$ is denoted $\morph{f}{A}{B}$. 
The morphisms between each pair of objects $A \to B$ form a \emph{hom-set}, denoted $\mcc(A,B)$. 
Each object is equipped with an identity morphism $\morph{\id_A}{A}{A}$. 
Morphisms can be composed sequentially: if we have morphisms $\morph{f}{A}{B}$ and $\morph{g}{B}{C}$ we also have the morphism $\morph{f \seq g}{A}{C}$. 
Composition is associative ($f \seq (g \seq h) = (f \seq g) \seq h$) and unital ($\id \seq f = f = f \seq \id$).
In the language of string diagrams, we represent morphisms as boxes, and composition by horizontal juxtaposition.
The identity is drawn as an empty wire.
Equal morphisms in the category correspond to isomorphic diagrams -- `only connectivity matters'.

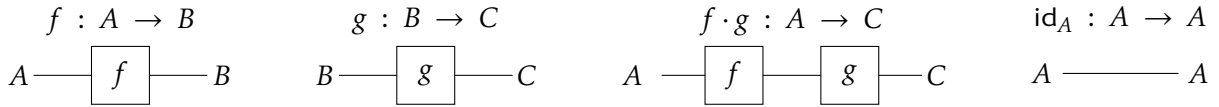
\begin{center}
  \begin{tikzpicture}[yscale=-1,x=1em,y=1.25em]

    \node at (-1,0) {$A$};
    \node at (6,0) {$B$};

    \node at (2.5,-1.5) {$\morph{f}{A}{B}$};
    \draw  (-0.5,0) -- (1.5,0);
    \node[draw,minimum width=2em,minimum height=2em] at (2.5,0) {$f$};
    \draw (3.5,0) -- (5.5,0);
\end{tikzpicture}
  \quad\quad
  \begin{tikzpicture}[yscale=-1,x=1em,y=1.25em]
        
    \node at (-1,0) {$B$};
    \node at (6,0) {$C$};

    \node at (2.5,-1.5) {$\morph{g}{B}{C}$};
    \draw (-0.5,0) -- (1.5,0);
    \node[draw,minimum width=2em,minimum height=2em] at (2.5,0) {$g$};
    \draw (3.5,0) -- (5.5,0);
\end{tikzpicture}
  \quad\quad
  
\begin{tikzpicture}[yscale=-1,x=1em,y=1.25em]
        
    \node at (-1,0) {$A$};
    \node at (9.5,0) {$C$};

    \node at (4.5,-1.5) {$\morph{f \seq g}{A}{C}$};
    \draw (0,0) -- (1.5,0);
    \node[draw,minimum width=2em,minimum height=2em] at (2.5,0) {$f$};
    \draw (3.5,0) -- (5.5,0);
    \node[draw,minimum width=2em,minimum height=2em] at (6.5,0) {$g$};
    \draw (7.5,0) -- (9,0);

\end{tikzpicture}

  \quad\quad
  \raisebox{0.4em}{
\begin{tikzpicture}[yscale=-1,x=1em,y=1.25em]
    \node [anchor=east] at (0.5,0) {$A$};
    \node [anchor=west] at (4.5,0) {$A$};
    \node at (2.5,-1.5) {$\morph{\id_A}{A}{A}$};
    \draw (0.5,0) -- (4.5,0);
\end{tikzpicture}
}
\end{center}

\noindent A \textit{monoidal} category \cite{joyal1991geometry} introduces a new binary operation known as the \textit{monoidal tensor}, denoted $- \tensor -$. 
The unit object of the monoid is denoted $I$.
Much like sequential composition, the tensor is associative ($(f \tensor g) \tensor h = f \tensor (g \tensor h)$) and unital with respect to the identity of the unit object ($\id[I] \tensor f = f = f \tensor \id[I]$).
When we write categorical terms, $\tensor$ binds tighter than $\seq$, so $f \tensor g \seq h \tensor k$ should be read as $(f \tensor g) \seq (h \tensor k)$.
Graphically, the tensor is drawn as vertical juxtaposition and the unit object is drawn as `empty space'.

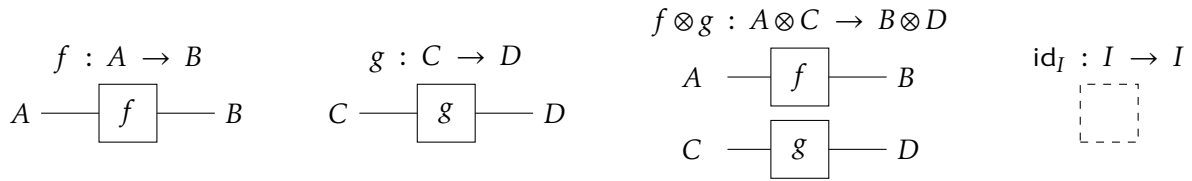
\begin{center}
  
\begin{tikzpicture}[yscale=-1,x=1em,y=1.25em]        
    \node [anchor=east] at (-0.5,0) {$A$};
    \node [anchor=west] at (5.5,0) {$B$};
    \node at (2.5,-1.5) {$\morph{f}{A}{B}$};
    \draw (-0.5,0) -- (1.5,0);
    \node[draw,minimum width=2em,minimum height=2em] at (2.5,0) {$f$};
    \draw (3.5,0) -- (5.5,0);
\end{tikzpicture}

  \quad\quad
  
\begin{tikzpicture}[yscale=-1,x=1em,y=1.25em]
    \node [anchor=east] at (-0.5,0) {$C$};
    \node [anchor=west] at (5.5,0) {$D$};
    \node at (2.5,-1.5) {$\morph{g}{C}{D}$};
    \draw (-0.5,0) -- (1.5,0);
    \node[draw,minimum width=2em,minimum height=2em] at (2.5,0) {$g$};
    \draw (3.5,0) -- (5.5,0);
\end{tikzpicture}

  \quad\quad
  \raisebox{-1.25em}{
\begin{tikzpicture}[yscale=-1,x=1em,y=1.25em]

    \node at (2.5,-1.5) {$\morph{f \tensor g}{A \tensor C}{B \tensor D}$};
    \node [anchor=east] at (-0.5,0) {$A$};
    \node [anchor=east] at (-0.5,2) {$C$};
    \node [anchor=west] at (5.5,0) {$B$};
    \node [anchor=west] at (5.5,2) {$D$};

    \draw (0,0) -- (1.5,0);
    \node[draw,minimum width=2em,minimum height=2em] at (2.5,0) {$f$};
    \draw (3.5,0) -- (5.5,0);

    \draw (0,2) -- (1.5,2);
    \node[draw,minimum width=2em,minimum height=2em] at (2.5,2) {$g$};
    \draw (3.5,2) -- (5.5,2);
\end{tikzpicture}
}
  \quad\quad
  
\begin{tikzpicture}[yscale=-1,x=1em,y=1.25em]

    \node at (2.5,-1.5) {$\morph{\id_I}{I}{I}$};

    \node[draw,dashed,minimum width=2em,minimum height=2em] at (2.5,0) {};
\end{tikzpicture}

\end{center}

\noindent The addition of tensor means that there are multiple ways in which we can compose morphisms in sequence or in parallel that lead to equal terms.
This is known as \emph{functoriality}, and can be expressed as the following axiom: $(f \seq g) \tensor (h \seq k) = f \tensor h \seq g \tensor k$.
Functoriality means that using the one dimensional algebraic notation can obfuscate the true nature of the inherently two dimensional structure.
This is especially important computationally, as numerous extra operations must be performed to manipulate a term appropriately.
Fortunately, the graphical notation eliminates this overhead, as both terms correspond to the same diagram:

\begin{center}
  
\begin{tikzpicture}[yscale=-1,x=1em,y=1.25em]

    \node [anchor=east] at (0,0) {$A$};
    \node [anchor=east] at (0,2) {$D$};
    \node [anchor=west] at (9,0) {$C$};
    \node [anchor=west] at (9,2) {$F$};

    \node at (4.5,-1.5) {$\morph{(f \tensor h) \seq (g \tensor k)}{A \tensor D}{C \tensor F}$};
    
    \draw (0,0) -- (1.5,0);
    \node[draw,minimum width=2em,minimum height=2em] at (2.5,0) {$f$};
    \draw (3.5,0) -- (5.5,0);
    \node[draw,minimum width=2em,minimum height=2em] at (6.5,0) {$g$};
    \draw (7.5,0) -- (9,0);

    \draw (0,2) -- (1.5,2);
    \node[draw,minimum width=2em,minimum height=2em] at (2.5,2) {$h$};
    \draw (3.5,2) -- (5.5,2);
    \node[draw,minimum width=2em,minimum height=2em] at (6.5,2) {$k$};
    \draw (7.5,2) -- (9,2);

\end{tikzpicture}

\end{center}

\noindent To acquire a framework suitable for modelling systems, we need a way of crossing over the wires in our diagrams.
This is achieved by equipping each pair of objects $A,B$ in our category with a \emph{symmetry} $\morph{\swap{A}{B}}{A \tensor B}{B \tensor A}$.
A category in \textit{symmetric} monoidal category (or SMC), and braidings are called \textit{symmetries}.
The symmetry satisfies the axioms of \emph{naturality} $f \seq g \seq \swap{B}{D} = \swap{A}{C} \seq f \tensor g$, hexagon $\swap{A}{B} \tensor \id[C] \seq \id[A] \tensor \swap{B}{C} = \swap{A}{B \tensor C}$ and self-inverse $\swap{A}{B} \seq \swap{B}{A} = \id[A] \tensor \id[B]$, illustrated below.

\begin{center}
  
\begin{tikzpicture}[yscale=-1,x=1em,y=1.25em]
        
    \node [anchor=east] at (-2.5,0) {$A$};
    \node [anchor=east] at (-2.5,2) {$C$};
    \node [anchor=west] at (4,0) {$D$};
    \node [anchor=west] at (4,2) {$B$};

    \draw (-1.5,0) -- (-2.5,0);
    \draw (-1.5,2) -- (-2.5,2);

    \node[draw, minimum width = 1.5em, minimum height = 1.5em, anchor=east] at (0,0) {$f$};
    \node[draw, minimum width = 1.5em, minimum height = 1.5em, anchor=east] at (0,2) {$g$};

    \draw[rounded corners] (0,0) -- (1,0) -- (2,1) -- (3,2) -- (4,2);
    \draw[rounded corners] (0,2) -- (1,2) -- (3,0) -- (4,0);

    \node at (6,1) {$=$};

    \node [anchor=east] at (8,0) {$A$};
    \node [anchor=east] at (8,2) {$C$};
    \node [anchor=west] at (14.5,0) {$D$};
    \node [anchor=west] at (14.5,2) {$B$};

    \draw[rounded corners] (8,0) -- (9,0) -- (11,2) -- (12,2);
    \draw[rounded corners] (8,2) -- (9,2) -- (11,0) -- (12,0);

    \node[draw, minimum width = 1.5em, minimum height = 1.5em, anchor=west] at (12,0) {$g$};
    \node[draw, minimum width = 1.5em, minimum height = 1.5em, anchor=west] at (12,2) {$f$};

    \draw (13.5,0) -- (14.5,0);
    \draw (13.5,2) -- (14.5,2);

\end{tikzpicture}

  \quad
  \raisebox{0.1em}{
\begin{tikzpicture}[yscale=-1,x=1em,y=1.25em]
        
    \node [anchor=east] at (0,0) {$A$};
    \node [anchor=east] at (0,1) {$B$};
    \node [anchor=east] at (0,2) {$C$};
    \node [anchor=west] at (5,0) {$B$};
    \node [anchor=west] at (5,1) {$C$};
    \node [anchor=west] at (5,2) {$A$};

    \draw[rounded corners] (0,0) -- (1,0) -- (2,1) -- (3,1) -- (4,2) -- (5,2);
    \draw[rounded corners] (0,1) -- (1,1) -- (2,0) -- (5,0);
    \draw[rounded corners] (0,2) -- (3,2) -- (4,1) -- (5,1);

    \node at (7,1) {$=$};

    \node [anchor=east] at (9,0) {$A$};
    \node [anchor=east] at (9,1) {$B$};
    \node [anchor=east] at (9,2) {$C$};
    \node [anchor=west] at (13,0) {$B$};
    \node [anchor=west] at (13,1) {$C$};
    \node [anchor=west] at (13,2) {$A$};

    \draw[rounded corners] (9,0) -- (10,0) -- (12,2) -- (13,2);
    \draw[rounded corners] (9,1) -- (10,1) -- (12,0) -- (13,0);
    \draw[rounded corners] (9,2) -- (10,2) -- (12,1) -- (13,1);

\end{tikzpicture}
}
  \quad
  \raisebox{0.5em}{
\begin{tikzpicture}[yscale=-1,x=1em,y=1.25em]
        
    \node [anchor=east] at (0,0) {$A$};
    \node [anchor=east] at (0,1) {$B$};
    \node [anchor=west] at (5,0) {$B$};
    \node [anchor=west] at (5,1) {$A$};

    \draw[rounded corners] (0,0) -- (1,0) -- (2,1) -- (3,1) -- (4,0) -- (5,0);
    \draw[rounded corners] (0,1) -- (1,1) -- (2,0) -- (3,0) -- (4,1) -- (5,1);

    \node at (7,0.5) {$=$};

    \node [anchor=east] at (9,0) {$A$};
    \node [anchor=east] at (9,1) {$B$};
    \node [anchor=west] at (13,0) {$B$};
    \node [anchor=west] at (13,1) {$A$};

    \draw (9,0) -- (13,0);
    \draw (9,1) -- (13,1);

\end{tikzpicture}
}
\end{center}

\noindent We are particularly interested in \textit{free} monoidal categories, where morphisms, or `terms', are generated over a \textit{monoidal signature} $\Sigma = (\Sigma_O,\Sigma_M)$: a set of object variables and morphism variables (\emph{generators}), equipped with functions $\morph{\mf{dom},\mf{cod}}{\Sigma_M}{\Sigma_O^\mon}$, where $\Sigma_O^\mon$ is a list of object variables, denoting the domain and codomain of each generator. 
Effectively, generators are the building blocks from which we can form categorical terms, by composing generators in sequence or parallel with each other, identity morphisms and symmetries.
For example, the free monoidal category generated over the signature $\Sigma = \{\morph{f}{X}{B \tensor C}, \morph{g}{B \tensor A}{X}\}$ contains the following term: \begin{center}
  $f \tensor \id[A] \seq \id[B] \tensor \swap{C}{A} \seq g \tensor \id[C]$

  \vspace{0.5em}

  \begin{tikzpicture}[yscale=-1,x=1em,y=1.25em]

    \node [anchor=east] at (0,0.5) {$X$};
    \node [anchor=east] at (0,1.5) {$A$};
    \draw (0,0.5) -- (1,0.5);
    \draw [rounded corners] (0,1.5) -- (3,1.5) -- (4,0.75) -- (5,0.75);

    \node[draw, minimum height = 1.6em, minimum width = 1.5em, anchor = west] at (1,0.5){$f$};

    \draw (2.5,0.25) -- (5,0.25);
    \draw [rounded corners] (2.5,0.75) -- (3,0.75) -- (4,1.5) -- (7.5,1.5);
    \draw (6.5,0.5) -- (7.5,0.5);

    \node[draw, minimum height = 1.6em, minimum width = 1.5em, anchor = west] at (5,0.5){$g$};
    \node [anchor=west] at (7.5,0.5) {$X$};
    \node [anchor=west] at (7.5,1.5) {$C$};
\end{tikzpicture}
\end{center}

\begin{lemma}[Staging]\label{lem:staging}
  Any morphism $f \in \termtype$ can be written as in the form $f = f_0 \seq f_1 \seq \,\, \cdots \,\, \seq f_n$, where $f_i$ is a tensor containing only one non-identity morphism, $f_i = \id[p] \tensor k \tensor \id[q]$.
\end{lemma}
\begin{proof}
  By functoriality and unitality.
\end{proof}

\noindent A useful class of symmetric monoidal categories are called \emph{PROPs} (PROduct and Permutation categories), categories with natural numbers as objects and addition as tensor product.
These are especially natural with regards to graphical notation as an object $n \in \nat$ can be drawn as $n$ wires.

\begin{lemma}[Composite symmetry]
  Any symmetry $\swap{m}{n}$ in a free PROP can be expressed as a combination of multiple symmetries $\swap{1}{1}$ and identities.
\end{lemma}
\begin{proof}
  By the hexagon axiom.
\end{proof}

\subsection{Symmetric traced monoidal categories}

So far, the wires in our string diagrams have only travelled in one direction across the page: from left to right. 
However to model some systems we may want to `bend' these wires, such as to model feedback. 
A common way of doing this is to use a \emph{compact closed category} \cite{kelly1980coherence}, in which every object $A$ has a \emph{dual} $A^\mon$, drawn as a wire travelling from right to left.
Each object is also equipped with additional structural morphisms known as the $\morph{cup}{A^\mon \tensor A}{I}$ and the $\morph{cap}{I}{A \tensor A^\mon}$ for `bending' wires.
However, this setting is not suitable for all applications.
In a compact closed category there is a flexible notion of causality, where morphisms do not have so much a notion of input and output but rather a bidirectional \emph{interface port}.
Instead, we may wish to enforce a strict notion of causality, where only outputs of morphisms can connect to inputs.
To do this, we must look at a flavour of monoidal categories known as \textit{symmetric traced monoidal categories} (or STMCs for short), which were introduced by \citet{joyal1996traced} and refined by \citet{hasegawa2009traced}. 

An STMC is an SMC with an extra family of operations known as \textit{trace operators}. For a morphism $\morph{f}{X \tensor A}{X \tensor B}$, we can \textit{trace} it to form the morphism $\morph{\text{Tr}_{A,B}^{X}(f)}{A}{B}$. 
We will often drop the subscript for clarity when there is no ambiguity. 
A trace is represented graphically by `bending around' one of the output wires to join up with one of the input wires. 
This enables wires to travel in the opposite direction for a period, but all wires must still be oriented left-to-right when interacting with morphisms, as shown below:

\begin{center}
    \begin{tikzpicture}[yscale=-1,x=1em,y=1.25em]

    \draw (0.5,0.5) -- (2.25,0.5); 
    \draw (0.5,1.5) -- (2.25,1.5);

    \draw (4.25,0.5) -- (6,0.5);
    \draw (4.25,1.5) -- (6,1.5);
    
    \node[draw, minimum height = 2.5em, minimum width = 2em, anchor = west] at (2.25,1){$f$};
    \node[anchor = east] at (0,0.5){$X$};
    \node[anchor = east] at (0,1.5){$A$};
    \node[anchor = west] at (6,0.5){$X$};
    \node[anchor = west] at (6,1.5){$B$};

\end{tikzpicture}
    \hspace{1em}
    \raisebox{10pt}{$\xrightarrow{\vtrace{X}{A}{B}{f}}$}
    \hspace{1em}
    \begin{tikzpicture}[yscale=-1,x=1em,y=1.25em]

    \draw (-0.75,1.5) -- (1.25,1.5);
    \draw (1.25,1.5) -- (2,1.5);
    \draw (4.0,0.5) -- (5.0, 0.5);
    \draw [rounded corners, ->] (5.0,0.5) -- (6.0,0.5) -- (6.0,-0.5) -- (3,-0.5);
    \draw [rounded corners] (3,-0.5) -- (0,-0.5) -- (0,0.5) -- (1.25,0.5);
    \draw (1.25,0.5) -- (2,0.5);
    \draw (4.0,1.5) -- (5.25,1.5);
    \draw (5.25, 1.5) -- (6.75, 1.5);

    \node[anchor = east] at (-0.75,1.5){$A$};
    \node[anchor = west] at (7,1.5){$B$};
    \node[draw, minimum height = 2.5em, minimum width = 2em, anchor = west] at (2,1){$f$};
\end{tikzpicture}
\end{center}

\noindent There are several (equivalent) formulations of the axioms of STMCs, but here we present the four detailed by Hasegawa in \cite{hasegawa2009traced}.

\vspace{1em}

\noindent \textbf{Tightening}

\begin{center}
    $\vtrace{X}{A}{D}{\id_X \tensor g \seq f \seq \id[X] \tensor h} = g \seq \vtrace{X}{B}{C}{f} \seq h$

    \vspace{0.5em}

\begin{tikzpicture}[yscale=-1,x=1em,y=1.25em]
        
    \node at (0,-2.5) {$\null$};

    \node [anchor=east] at (0,0) {$A$};
    \node [anchor=west] at (12,0) {$D$};

    \draw (0,0) -- (2,0);
    \node[draw, minimum height = 1em, minimum width = 2em, anchor = west] at (2,0){$g$};
    \draw (4,0) -- (5,0);
    \node[draw, minimum height = 2.5em, minimum width = 2em, anchor = west] at (5,-0.5){$f$};
    \draw (7,0) -- (8,0);
    \node[draw, minimum height = 1em, minimum width = 2em, anchor = west] at (8,0){$h$};
    \draw (10,0) -- (12,0);
    \draw [->, rounded corners] (7,-1) -- (11,-1) -- (11,-2) -- (5.9,-2);
    \draw [rounded corners] (5.9,-2) -- (1,-2) -- (1,-1) -- (5,-1);

    \node at (15,-1) {$=$};

    \node [anchor=east] at (18,0) {$A$};
    \node [anchor=west] at (30,0) {$D$};

    \draw (18,0) -- (19,0);
    \node[draw, minimum height = 1em, minimum width = 2em, anchor = west] at (19,0){$g$};
    \draw (21,0) -- (23,0);
    \node[draw, minimum height = 2.5em, minimum width = 2em, anchor = west] at (23,-0.5){$f$};
    \draw (25,0) -- (27,0);
    \node[draw, minimum height = 1em, minimum width = 2em, anchor = west] at (27,0){$h$};
    \draw (29,0) -- (30,0);
    \draw [->, rounded corners] (25,-1) -- (26,-1) -- (26,-2) -- (23.9,-2);
    \draw [rounded corners] (23.9,-2) -- (22,-2) -- (22,-1) -- (23,-1);

\end{tikzpicture}

  \end{center}

\noindent \textbf{Yanking}

\begin{center}
    $\vtrace{X}{X}{X}{\swap{X}{X}} = \id[X]$

    \vspace{0.5em}

\begin{tikzpicture}[yscale=-1,x=1em,y=1.25em]

    \node at (0,-1.5) {$\null$};

    \node [anchor=east] at (-1,1) {$X$};
    \node [anchor=west] at (6,1) {$X$};

    \draw (-1,1) -- (1,1);
    \draw [rounded corners] (1,1) -- (2,1) -- (3,0) -- (4,0);
    \draw [->, rounded corners] (4,0) -- (5,0) -- (5,-1) -- (2.4,-1); 
    \draw [rounded corners] (2.4,-1) -- (0,-1) -- (0,0) -- (1,0);
    \draw [rounded corners] (1,0) -- (2,0) -- (3,1) -- (4,1);
    \draw (4,1) -- (6,1);

    \node at (9,0) {$=$};

    \node [anchor=east] at (12,0) {$X$};
    \node [anchor=west] at (16,0) {$X$};
    \draw (12,0) -- (16,0);

\end{tikzpicture}

  \end{center}

\noindent \textbf{Superposing}

  \begin{center}
    $\vtrace{X}{A \tensor C}{B \tensor C}{f \tensor \id[C]} = \vtrace{X}{A}{B}{f} \tensor \id[C]$

    \vspace{0.5em}

\begin{tikzpicture}[yscale=-1,x=1em,y=1.25em]

    \node [anchor=east] at (0,-0.25) {$C$};
    \node [anchor=east] at (0,-1) {$A$};
    \node [anchor=west] at (6,-0.25) {$C$};
    \node [anchor=west] at (6,-1) {$B$};

    \draw (0,-0.25) -- (6,-0.25);
    \draw (0,-1) -- (2,-1);
    \node[draw, minimum height = 1.5em, minimum width = 2em, anchor = west] at (2,-1.25){$f$};
    \draw (4,-1) -- (6,-1);
    \draw [->, rounded corners] (4,-1.5) -- (5,-1.5) -- (5,-2.5) -- (3,-2.5);
    \draw [rounded corners] (3,-2.5) -- (1,-2.5) -- (1,-1.5) -- (2,-1.5);
    \node[rounded corners, fill=gray, minimum width = 5em, minimum height = 3.75em, opacity=0.5] at (3,-1.33) {};

    \node at (9,-1) {$=$};

    \node [anchor=east] at (12,-0.25) {$C$};
    \node [anchor=east] at (12,-1) {$A$};
    \node [anchor=west] at (18,-0.25) {$C$};
    \node [anchor=west] at (18,-1) {$B$};

    \draw (12,-0.25) -- (18,-0.25);
    \draw (12,-1) -- (14,-1);
    \node[draw, minimum height = 1.5em, minimum width = 2em, anchor = west] at (14,-1.25){$f$};
    \draw (16,-1) -- (18,-1);
    \draw [->, rounded corners] (16,-1.5) -- (17,-1.5) -- (17,-2.5) -- (15,-2.5);
    \draw [rounded corners] (15,-2.5) -- (13,-2.5) -- (13,-1.5) -- (14,-1.5);
    \node[rounded corners, fill=gray, minimum width = 5em, minimum height = 3em, opacity=0.5] at (15,-1.63) {};

\end{tikzpicture}

  \end{center}

\noindent \textbf{Exchange}

  \begin{center}
    $\vtrace{Y}{A}{B}{\vtrace{X}{Y \tensor A}{Y \tensor B}{f}} = \vtrace{X}{A}{B}{\vtrace{Y}{X \tensor A}{X \tensor B}{\swap{Y}{X} \tensor \id[A] \seq f \seq \swap{X}{Y} \tensor \id[B]}}$

    \vspace{0.5em}

\begin{tikzpicture}[yscale=-1,x=1em,y=1.25em]

    \node at (5,-0.5) {$\null$};

    \node [anchor=east] at (1.5,2) {$A$};
    \node [anchor=west] at (8.5,2) {$B$};

    \draw (1.5,2) -- (4,2);
    \node[draw, minimum height = 2em, minimum width = 2em, anchor = west] at (4,1.5){$f$};
    \draw [->, rounded corners] (6,1) -- (7,1) -- (7,0) -- (5,0);
    \draw [rounded corners] (5,0) -- (3,0) -- (3,1) -- (4,1);
    \draw [->, rounded corners] (6,1.5) -- (7.5,1.5) -- (7.5,-0.5) -- (5,-0.5);
    \draw [rounded corners] (5,-0.5) -- (2.5,-0.5) -- (2.5,1.5) -- (4,1.5);
    \draw (6,2) -- (8.5,2);

    \node at (11,1) {$=$};

    \node [anchor=east] at (13.5,2) {$A$};
    \node [anchor=west] at (22.5,2) {$B$};
    
    \draw (13.5,2) -- (17,2);
    \node[draw, minimum height = 2em, minimum width = 2em, anchor = west] at (17,1.5){$f$};
    \draw [rounded corners, ->] (19,1) -- (20,1) -- (20.5,1.5) -- (22,1.5) -- (22,-0.5) -- (18,-0.5);
    \draw [rounded corners, ->] (19,1.5) -- (20,1.5) -- (20.5,1) -- (21.5,1) -- (21.5,0) -- (18,0);
    \draw [rounded corners] (18,-0.5) -- (14, -0.5) -- (14, 1.5) -- (15.5, 1.5) -- (16,1) -- (17,1);
    \draw [rounded corners] (18,0) -- (14.5,0) -- (14.5,1) -- (15.5,1) -- (16,1.5) -- (17,1.5);
    \draw (19,2) -- (22.5,2);

\end{tikzpicture}

  \end{center}

\noindent As with regular symmetric monoidal categories, we can generate \emph{free} STMCs over a given signature with the addition of the trace operator.
For example, the free STMC defined over $\Sigma$ contains the following term:

\begin{center}
  $\trace{X}{\fork \tensor \id[A] \seq \id[B] \tensor \swap{D}{A} \seq \join \tensor \id[D]}$

  \vspace{0.5em}

  \begin{tikzpicture}[yscale=-1,x=1em,y=1.25em]

    \node [anchor=east] at (0,1.5) {$A$};
    \draw [rounded corners] (0,1.5) -- (3,1.5) -- (4,0.75) -- (5,0.75);

    \node[draw, minimum height = 1.6em, minimum width = 1.5em, anchor = west] at (1,0.5){$f$};

    \draw (2.5,0.25) -- (5,0.25);
    \draw [rounded corners] (2.5,0.75) -- (3,0.75) -- (4,1.5) -- (7.5,1.5);
    \draw [rounded corners] (6.5,0.5) -- (7.5,0.5) -- (7.5,-0.5) -- (0,-0.5) -- (0,0.5) -- (1,0.5);

    \node[draw, minimum height = 1.6em, minimum width = 1.5em, anchor = west] at (5,0.5){$g$};
    \node [anchor=west] at (7.5,1.5) {$C$};
\end{tikzpicture}
\end{center}

\noindent We can also derive one other important lemma that holds in any free STMC.

\begin{lemma}[Global trace]
  For any morphism $f \in \ntermtype$, we can represent it as $\trace{x}{\hat{f}}$, where $\hat{f}$ is a morphism containing no trace.
\end{lemma}
\begin{proof}
  By superposing and tightening.
\end{proof}

\subsection{Monoidal theories}

On their own, the axioms of symmetric traced monoidal categories are not particularly interesting.
To model systems we need to impose additional structure on our categories, which can be done with the introduction of new axioms.
A \textit{monoidal theory} is a monoidal signature equipped with a set of equations: pairs of terms with equal domain and codomain, e.g. for generators $\morph{f,g}{A}{A}$, a suitable equation could be $f \seq g = g \seq f$.
Well-known monoidal theories include those of commutative monoids, Frobenius monoids and non-commutative monoids, which contain various combinations of generators for forking and splitting wires (see \cite[Example 2.1]{bonchi2016rewriting} for details).
Monoidal theories can also be used to model the operational semantics of compositional systems: the generators are the building blocks of that system and the axioms represent the operational semantics that we can use to reduce complex systems into simpler ones.
In \S\ref{sec:case-study} we will examine a theory at the centre of our research, that of \emph{digital circuits}.
For now, we will present an example that motivate our work.

\subsubsection{Example: Cartesian categories}\label{sec:cartesian}

A \emph{Cartesian category} is a symmetric monoidal category where each object is equipped with a \emph{diagonal} morphism $\morph{\ccopy{A}}{A}{A \tensor A}$, and where the unit object is \emph{terminal}: for every object $A$, there is a unique morphism $\morph{\cdel{A}}{A}{I}$.
\begin{center}
  \begin{tikzpicture}[yscale=-1,x=1em,y=1.25em]
        
    \node at (1.4,-2) {$\morph{\ccopy{A}}{A}{A \tensor A}$};

    \node [anchor=east] at (-1,0) {$A$};
    \draw (-1,0) -- (0,0);
    \node[draw, minimum height = 2.2em, minimum width = 2.2em, anchor = west] at (0,0){$\ccopy{A}$};
    \draw (2.2,-0.5) -- (3.2,-0.5);
    \draw (2.2,0.5) -- (3.2,0.5);
    \node [anchor=west] at (3.2,-0.5) {$A$};
    \node [anchor=west] at (3.2,0.5) {$A$};

\end{tikzpicture}
  \qquad
  \begin{tikzpicture}[yscale=-1,x=1em,y=1.25em]
        
    \node at (0.3,-2) {$\morph{\cdel{A}}{A}{I}$};

    \node [anchor=east] at (-1,0) {$A$};
    \draw (-1,0) -- (0,0);
    \node[draw, minimum height = 2.2em, minimum width = 2.2em, anchor = west] at (0,0){$\cdel{A}$};
\end{tikzpicture}
\end{center}
In essence, a Cartesian category is a monoidal category in which the tensor product is the Cartesian product.
In the Cartesian monoidal theory, the families of diagonals and terminal morphisms are the generators; the accompanying axioms can be seen in Table~\ref{table:cartesian}.

\begin{table}
  \centering
  \begin{tabular}{ll}
  Naturality axioms & \\
  \hline \\
  $f \seq \ccopy{B} = \ccopy{A} \seq f \tensor f$ & $A \to B \tensor B$ \\
  $f \seq \cdel{B} = \cdel{A}$ & $A \to 0$ \\ \\
  Commutative comonoid axioms & \\
  \hline \\
  $\ccopy{A} \seq \ccopy{A} \tensor A = \ccopy{A} \seq (A \tensor \ccopy{A})$ & $A \to A \tensor A \tensor A$ \\
  $\ccopy{A} \seq (\cdel{A} \tensor A) = A$ & $A \to A$ \\
  $\ccopy{A} \seq (A \tensor \cdel{A})$ & $A \to A$ \\
  $\ccopy{A} \seq \swap{A}{A} = \ccopy{A}$ & $A \to A \tensor A$ \\ \\
  Coherence axioms & \\
  \hline \\
  $\ccopy{0} = 0$ & $0 \to 0$ \\
  $\ccopy{A \tensor B} \seq (A \tensor \swap{A}{B} \tensor B) = \ccopy{A} \tensor \ccopy{B}$ & $A \tensor B \to A \tensor B \tensor A \tensor B$ \\
  $\cdel{A} = 0$ & $0 \to 0$ \\
  $\cdel{A \tensor B} = \cdel{A} \tensor \cdel{B}$ & $A \tensor B \to 0$
  \end{tabular}
  \caption{The axioms for a Cartesian monoidal category~\cite{selinger2010survey}}
  \label{table:cartesian}
\end{table}

Cartesian categories that are also traced are known as \emph{dataflow categories}~\cite[\S6.4]{selinger2010survey}.
The interaction of the trace with the Cartesian product is especially interesting, as it admits a \emph{fixpoint operator}, as noticed by Hasegawa~\cite{hasegawa1997recursion} and Martin Hyland independently.
Equivalent observations had also been made before the introduction of traced monoidal categories, such as by Bloom and Ésik~\cite{bloom1993iteration} and Ştefǎnescu~\cite{stefanescu2000network}.

\begin{theorem}[\texorpdfstring{Trace-fixpoint correspondence~\cite{hasegawa1997recursion}}{Trace-fixpoint correspondence}]
  A Cartesian category $\mcc$ is traced if and only if it has a family of functions 
  $\morph{\fix{-}}{\mcc(A \tensor X,X)}{\mcc(A,X)}$

  \vspace{0.5em}

  \begin{center}
    \begin{tikzpicture}[yscale=-1,x=1em,y=1.25em]
    
    \node [anchor=east] at (-11,-0.5) {$A$};
    \draw (-11,-0.5) -- (-10,-0.5);
    \node[draw, minimum height = 2em, minimum width = 2em, anchor = west] at (-10,-0.5){$f^\dagger$};
    \draw (-8,-0.5) -- (-7,-0.5);
    \node [anchor=west] at (-7,-0.5) {$X$};

    \node at (-4,-0.5) {$=$};

    \node [anchor=east] at (-1,0) {$A$};
    \draw (-1,0) -- (1,0);
    \node[draw, minimum height = 2em, minimum width = 2em, anchor = west] at (1,-0.5){$f$};
    \draw (3,-0.5) -- (4,-0.5);
    \node[draw, minimum height = 1em, minimum width = 1.75em, anchor = west] at (4,-0.5){$\ccopy{}$};
    \draw [rounded corners, ->] (5.75,-0.75) -- (6.5,-0.75) -- (6.5, -2) -- (3.5, -2);
    \draw [rounded corners] (3.5,-2) -- (0,-2) -- (0,-1) -- (1,-1);
    \draw [rounded corners] (5.75, -0.25) -- (7,-0.25);
    \node [anchor=west] at (7,-0.25) {$X$};

\end{tikzpicture}
  \end{center}
  such that the following axioms are satisfied:
\end{theorem}

\noindent\textbf{Naturality}
\begin{center}
  $\fix{\id[X] \tensor g \seq f} = g \seq \fix{f}$

  \vspace{0.5em}

  \begin{tikzpicture}[yscale=-1,x=1em,y=1.25em]
        
    \draw (-2.5,0) -- (-1.25,0);
    \draw (0.25,0) -- (1,0);
    \node[draw, minimum height = 2em, minimum width = 2em, anchor = west] at (1,-0.5){$f$};
    \node[draw, minimum height = 1.5em, minimum width = 1.5em, anchor = west] at (-1.25, 0){\scalebox{0.75}{$g$}};
    \draw (3,-0.5) -- (4,-0.5);
    \node[draw, minimum height = 1em, minimum width = 1.75em, anchor = west] at (4,-0.5){$\ccopy{}$};
    \draw [rounded corners, ->] (5.75,-0.75) -- (6.5, -0.75) -- (6.5,-2) -- (2, -2);
    \draw [rounded corners] (2,-2) -- (-2,-2) -- (-2,-1) -- (1,-1);
    \draw [rounded corners] (5.75,-0.25) -- (7,-0.25);

    \node [anchor=east] at (-2.5,0) {$B$};
    \node [anchor=west] at (7,-0.25) {$X$};

    \node at (10,-0.5) {$=$};

    \node [anchor=east] at (13,0) {$B$};
    \draw (13,0) -- (13.75,0);
    \node[draw, minimum height = 1.5em, minimum width = 1.5em, anchor = west] at (13.75, 0){\scalebox{0.75}{$g$}};
    \draw (15.25, 0) -- (17,0);
    \node[draw, minimum height = 2em, minimum width = 2em, anchor = west] at (17,-0.5){$f$};
    \draw (19,-0.5) -- (20,-0.5);
    \node[draw, minimum height = 1em, minimum width = 1.75em, anchor = west] at (20,-0.5){$\ccopy{}$};
    \draw [rounded corners, ->] (21.75,-0.75) -- (22.5, -0.75) -- (22.5,-2) -- (19, -2);
    \draw [rounded corners] (19,-2) -- (16,-2) -- (16,-1) -- (17,-1);
    \draw (21.75,-0.25) -- (23,-0.25);
    \node [anchor=west] at (23,-0.25) {$X$};

\end{tikzpicture}
\end{center}

\noindent\textbf{Dinaturality}
\begin{center}
  $\fix{\ccopy{X \tensor A} \seq g \tensor \cdel{X} \tensor \id[A] \seq f} = \ccopy{A} \seq \fix{\ccopy{X \tensor A} \seq g \tensor \cdel{X} \seq f} \tensor \id[A] \seq f$

  \vspace{0.5em}

  \raisebox{0.5em}{\begin{tikzpicture}[yscale=-1,x=1em,y=1.25em]
        
    \node [anchor=east] at (0,0) {$A$};
    \draw (0,0) -- (2,0);
    \node[draw, minimum height = 2em, minimum width = 1.75em, anchor = west] at (2,-0.4){$\ccopy{}$};
    \draw [rounded corners] (3.75,-1) -- (4.2,-1) -- (4.5,-1.75) -- (5.25,-1.75);
    \draw [rounded corners] (3.75,-0.6) -- (4.4,-0.6) -- (4.7,-1.25) -- (5.25,-1.25);
    \draw (3.75,-0.2) -- (5.25,-0.2);
    \draw [rounded corners] (3.75,0.2) -- (4.25, 0.2) -- (5,0.6) -- (7,0.6) -- (7.5,0) -- (8.5,0);
    \node[draw, minimum height = 1.5em, minimum width = 1.5em, anchor = west] at (5.25,-1.5){\scalebox{0.75}{$g$}};
    \node[draw, minimum height = 0.25em, minimum width = 0.25em, anchor = west] at (5.25,-0.2){\scalebox{0.5}{$\diamond$}};
    \draw [rounded corners] (6.75, -1.5) -- (7.25,-1.5) -- (8,-0.6) -- (8.5,-0.6);
    \node[draw, minimum height = 1.5em, minimum width = 1.5em, anchor = west] at (8.5,-0.3){\scalebox{0.75}{$f$}};
    \draw (10,-0.3) -- (11,-0.3);
    \node[draw, minimum height = 1em, minimum width = 1.75em, anchor = west] at (11,-0.3){$\ccopy{}$};
    \draw [rounded corners, ->] (12.75,-0.6) -- (13.5,-0.6) -- (13.5, -3) -- (7,-3);
    \draw [rounded corners] (7,-3) -- (1,-3) -- (1,-0.75) -- (2,-0.75);
    \draw (12.75, 0) -- (14.5,0);
    \node [anchor=west] at (14.5,0) {$X$};

\end{tikzpicture}}
  \quad\quad \raisebox{2.5em}{$=$} \quad\quad
  \begin{tikzpicture}[yscale=-1,x=1em,y=1.25em]
        
    \node [anchor=east] at (-3,1) {$A$};
    \draw (-3,1) -- (-2.25,1);
    \node[draw, minimum height = 1.25em, minimum width = 1.75em, anchor = east] at (-0.5,1){$\ccopy{}$};
    \draw [rounded corners] (-0.5,0.75) -- (0.5,0.75) -- (1,0) -- (2,0);
    \draw [rounded corners] (-0.5,1.25) -- (13,1.25) -- (13.5,0.5) -- (14.5,0.5);
    \node[draw, minimum height = 2em, minimum width = 1.75em, anchor = west] at (2,-0.4){$\ccopy{}$};
    \draw [rounded corners] (3.75,-1) -- (4.2,-1) -- (4.5,-1.75) -- (5.25,-1.75);
    \draw [rounded corners] (3.75,-0.6) -- (4.4,-0.6) -- (4.7,-1.25) -- (5.25,-1.25);
    \draw (3.75,-0.2) -- (5.25,-0.2);
    \draw [rounded corners] (3.75,0.2) -- (4.25, 0.2) -- (5,0.6) -- (7,0.6) -- (7.5,0) -- (8.5,0);
    \node[draw, minimum height = 1.5em, minimum width = 1.5em, anchor = west] at (5.25,-1.5){\scalebox{0.75}{$g$}};
    \node[draw, minimum height = 0.25em, minimum width = 0.25em, anchor = west] at (5.25,-0.2){\scalebox{0.5}{$\diamond$}};
    \draw [rounded corners] (6.75, -1.5) -- (7.25,-1.5) -- (8,-0.6) -- (8.5,-0.6);
    \node[draw, minimum height = 1.5em, minimum width = 1.5em, anchor = west] at (8.5,-0.3){\scalebox{0.75}{$f$}};
    \draw (10,-0.3) -- (11,-0.3);
    \node[draw, minimum height = 1em, minimum width = 1.75em, anchor = west] at (11,-0.3){$\ccopy{}$};
    \draw [rounded corners, ->] (12.75,-0.6) -- (13.5,-0.6) -- (13.5, -3) -- (7,-3);
    \draw [rounded corners] (7,-3) -- (1,-3) -- (1,-0.75) -- (2,-0.75);
    \draw (12.75, 0) -- (14.5,0);
    \node[draw, minimum height = 1.5em, minimum width = 1.5em, anchor = west] at (14.5,0.25){\scalebox{0.75}{$f$}};
    \node [anchor=west] at (17,0.25) {$X$};
    \draw (16,0.25) -- (17,0.25);

\end{tikzpicture}
\end{center}

\noindent\textbf{Diagonal}
\begin{center}
  $\fix{\ccopy{X} \tensor \id[A] \seq f} = \fix{\fix{f}}$

  \vspace{0.5em}

  \raisebox{0.5em}{\begin{tikzpicture}[yscale=-1,x=1em,y=1.25em]
        
    \node [anchor=east] at (0.5,0) {$A$};
    \draw (0.5,0) -- (4.5,0);
    \node[draw, minimum height = 1em, minimum width = 1.75em, anchor = west] at (2,-1){$\ccopy{}$};
    \draw (3.75,-1.25) -- (4.5,-1.25);
    \draw (3.75,-0.75) -- (4.5,-0.75);
    \node[draw, minimum height = 2em, minimum width = 1.75em, anchor = west] at (4.5,-0.6){$f$};
    \draw (6.25,-0.6) -- (7,-0.6);
    \node[draw, minimum height = 1em, minimum width = 1.75em, anchor = west] at (7,-0.6){$\ccopy{}$};
    \draw [rounded corners,->] (8.75, -1) -- (9.5, -1) -- (9.5,-2) -- (5,-2);
    \draw [rounded corners] (5,-2) -- (1,-2) -- (1,-1) -- (2,-1);
    \draw (8.75,-0.2) -- (9.5,-0.2);
    \node [anchor=west] at (9.5,-0.2) {$X$};
\end{tikzpicture}}
  \quad\quad \raisebox{2.5em}{$=$} \quad\quad
  \begin{tikzpicture}[yscale=-1,x=1em,y=1.25em]
        
    \node [anchor=east] at (0.5,0) {$A$};
    \draw (0.5,0) -- (4.5,0);
    \draw (3.75,-0.75) -- (4.5,-0.75);
    \node[draw, minimum height = 2em, minimum width = 1.75em, anchor = west] at (4.5,-0.6){$f$};
    \draw (6.25,-0.6) -- (7,-0.6);
    \node[draw, minimum height = 1em, minimum width = 1.75em, anchor = west] at (7,-0.6){$\ccopy{}$};
    \draw [rounded corners,->] (8.75, -1) -- (9.5, -1) -- (9.5,-2) -- (6.5,-2);
    \draw [rounded corners] (6.5,-2) -- (3,-2) -- (3,-1.25) -- (4.5,-1.25);
    \draw (8.75,-0.2) -- (10,-0.2);
    \node[draw, minimum height = 1em, minimum width = 1.75em, anchor = west] at (10,-0.2){$\ccopy{}$};
    \draw [rounded corners, ->] (11.75,-0.6) -- (12.5,-0.6) -- (12.5,-2.5) -- (6.5, -2.5);
    \draw [rounded corners] (6.5,-2.5) -- (2, -2.5) -- (2,-0.75) -- (4.5,-0.75);
    \draw (11.75, 0.2) -- (12.5,0.2);
    \node [anchor=west] at (12.5,0.2) {$X$};
\end{tikzpicture}
\end{center}

\noindent We can use the fixpoint operator to model \emph{feedback} in our systems.
In particular, we can derive the slightly simpler \emph{fixed-point equation} from the dinaturality axiom~\cite{hasegawa2009traced} that allows us to `unfold' the fixpoint.

\begin{center}
  $f^\dagger = \ccopy{A} \seq f^\dagger \tensor \id[A] \seq f$

  \vspace{0.5em}

  \raisebox{0.5em}{\begin{tikzpicture}[yscale=-1,x=1em,y=1.25em]

    \node [anchor=east] at (-1,0) {$A$};
    \draw (-1,0) -- (1,0);
    \node[draw, minimum height = 2em, minimum width = 2em, anchor = west] at (1,-0.5){$f$};
    \draw (3,-0.5) -- (4,-0.5);
    \node[draw, minimum height = 1em, minimum width = 1.75em, anchor = west] at (4,-0.5){$\ccopy{}$};
    \draw [rounded corners, ->] (5.75,-0.75) -- (6.5,-0.75) -- (6.5, -2) -- (3.5, -2);
    \draw [rounded corners] (3.5,-2) -- (0,-2) -- (0,-1) -- (1,-1);
    \draw [rounded corners] (5.75, -0.25) -- (7,-0.25);
    \node [anchor=west] at (7,-0.25) {$X$};

\end{tikzpicture}} 
  \quad\quad
  \raisebox{1.5em}{$=$}
  \quad\quad
  \begin{tikzpicture}[yscale=-1,x=1em,y=1.25em]
        
    \node [anchor=east] at (10.5,0) {$A$};
    \draw (10.5, 0) -- (11.5,0);
    \node[draw, minimum height = 1.75em, minimum width = 1.75em, anchor = west] at (11.5,0){$\ccopy{}$};
    \draw [rounded corners] (13.25,-0.5) -- (14.5, -0.5);
    \draw [rounded corners] (13.25,0.5) -- (20.5,0.5);
    \node[draw, minimum height = 2em, minimum width = 2em, anchor = west] at (14.5,-1){$f$};
    \draw (16.5,-1) -- (17.5,-1);
    \node[draw, minimum height = 1.75em, minimum width = 1.75em, anchor = west] at (17.5,-1){$\ccopy{}$};
    \draw [->, rounded corners] (19.25,-1.25) -- (20,-1.25) -- (20, -2.5) -- (17, -2.5);
    \draw [rounded corners] (17, -2.5) -- (13.5,-2.5) -- (13.5, -1.5) -- (14.5,-1.5);
    \draw [rounded corners] (19.25,-0.5) -- (20.5,-0.5);
    \node[draw, minimum height = 2em, minimum width = 2em, anchor = west] at (20.5,0){$f$};
    \draw (22.5,0) -- (24,0);
    \node [anchor=west] at (24,0) {$X$};

\end{tikzpicture}
\end{center}

\subsubsection{Graphical reasoning with monoidal theories}

The reason that graphical languages are so useful when dealing with monoidal categories is that the axioms are absorbed into the notation, and the tedious bureaucracy is eliminated.
Unfortunately, once we start adding extra structure this starts to fall apart.
For example, take the example of the naturality of the Cartesian diagonal.

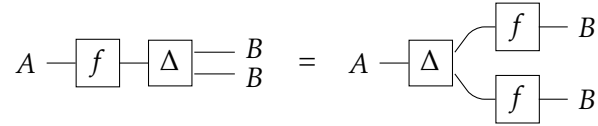
\begin{center}
  \begin{tikzpicture}[yscale=-1,x=1em,y=1.25em]

    \node [anchor=east] at (-0.5,0) {$A$};
    \node [anchor=west] at (6,-0.4) {$B$};
    \node [anchor=west] at (6,0.4) {$B$};
    
    \draw (-0.5,0) -- (0.5,0);
    \node[draw, minimum height = 1.5em, minimum width = 1.5em, anchor = west] at (0.5,0){$f$};
    \draw (2,0) -- (3,0);
    \node[draw, minimum height = 1.5em, minimum width = 1.5em, anchor = west] at (3,0){$\ccopy{}$};
    \draw [rounded corners] (4.6,-0.3) -- (6,-0.3);
    \draw [rounded corners] (4.6,0.3) -- (6,0.3);

    \node at (8.5,0) {$=$};
    \node at (9.4,-3) {$f \seq \ccopy{B} = \ccopy{A} \seq f \tensor f$};

    \node [anchor=east] at (11,0) {$A$};
    \node [anchor=west] at (17.5,-1) {$B$};
    \node [anchor=west] at (17.5,1) {$B$};
    
    \draw (11,0) -- (12,0);
    \node[draw, minimum height = 1.5em, minimum width = 1.5em, anchor = west] at (12,0){$\ccopy{}$};
    \draw [rounded corners] (13.6,-0.3) -- (14.25, -1) -- (15,-1);
    \draw [rounded corners] (13.6,0.3) -- (14.25,1) -- (15,1);

    \node[draw, minimum height = 1.5em, minimum width = 1.5em, anchor = west] at (15,-1){$f$};
    \node[draw, minimum height = 1.5em, minimum width = 1.5em, anchor = west] at (15,1){$f$};

    \draw (16.5,-1) -- (17.5,-1);
    \draw (16.5,1) -- (17.5,1);

\end{tikzpicture}
\end{center}

\noindent Clearly, this axiom cannot be absorbed by the graphical notation: even the number of boxes differs!
To tackle these axioms, we must consider diagrams not just up to isomorphism, but up to \emph{rewriting}.
To do this, we must move away from the topological string diagrams and towards a more combinatorial diagram, where vertices and edges are explicitly defined.
With these diagrams we can perform \emph{graph rewriting}, of which numerous formalisms and frameworks exist \cite{ehrig1973graph,ehrig1991parallelism,lack2004adhesive}.

As we have already observed in the introduction, this is not a new endeavour: previously this has been studied with string graphs \cite{dixon2013open,kissinger2012pictures} and hypergraphs \cite{bonchi2016rewriting,zanasi2017rewriting,bonchi2018rewriting,bonchi2020string}.
However, these are rooted in compact closed categories, which are incompatible with the Cartesian product.
This is because the Cartesian product automatically becomes a \emph{biproduct} in a compact closed setting~\cite{houston2008finite}, which is not always suitable (e.g. in the category of digital circuits detailed in \S\ref{sec:case-study}).
Therefore, we will need to define a slightly different combinatorial structure.

\section{Hypergraphs}\label{sec:hypergraphs}

We begin by recalling a standard notion of hypergraphs in which edges have ordered sources and targets, as in \cite{bonchi2016rewriting}.
Let $\atoms$ be a countably infinite set of \emph{atoms} (or \emph{names}, in the sense of~\cite{pitts2013nominal}). 

\begin{definition}[Hypergraph]\label{def:hypergraph}
  A hypergraph is a tuple $H~=~(V,E,\esources,\etargets)$ where
  \begin{itemize}
    \item $V \subset \atoms$ is a finite set of vertices.
    \item $E$ is a set containing, for each $k,l\in\nat$, finite sets $E[k,l]$ of hyperedges with $k$ sources and $l$ targets.
    \item $\esources,\etargets$ are families of functions denoting sources and targets of edges, i.e. for each $E[k,l]$:
    \begin{itemize} 
      \item for each $i < k$, there exists the $i$th source map $\morph{\esources[i]}{E[k,l]}{V}$ 
      \item for each $i < l$, there exists the $j$th target map $\morph{\etargets[i]}{E_{k,l}}{V}$ 
  \end{itemize}
\end{itemize}
\end{definition}

\noindent We call the \emph{in-degree}, written $\inputs{v}$ (resp. \emph{out-degree}, written $\outputs{v}$) of a vertex the number of edges it is the target (resp. source) of.
We call a hypergraph \emph{discrete} if it has no edges.
To reduce our use of space, for a hypergraph $H = (V,E,\esources,\etargets)$ we will often use $V_H,E_H$ etc. to access members of the tuple.

A \emph{hypergraph signature} is a set of labels $\hypsig$ equipped with functions $\morph{\mf{dom},\mf{cod}}{\hypsig}{\nat}$.
A \emph{labelled hypergraph} over signature $\hypsig$ is a hypergraph $H = (V,E,\esources,\etargets)$ and a labelling function $\morph{\labels}{\edges}{\hypsig}$, such that for any $e \in \edges$, if $\labels(e) = l$ then $\dom{l} = |\esources(e)|$ and $\cod{l} = |\etargets(e)|$.

\begin{example}\label{ex:hypergraph}

  Below there is an informal drawing of a hypergraph over the hypergraph signature \[\hypsig~=~\{\morph{\fork}{1}{2},~\morph{\join}{2}{1}\}.\]
  Vertices are drawn as black dots. 
  Edges are drawn as boxes, with ordered sources and targets connected on the left and right respectively.

  \begin{center}
    \begin{minipage}{0.6\textwidth}
      \begin{gather*}
        V = \{v_0,v_1,v_2,v_3,v_4\} \qquad E[1,2] = \{e_0,e_2\} \qquad E[2,1] = \{e_1\} \\
        \esources[0](e_0) = v_0 \quad \etargets[0](e_0) = v_2 \quad \etargets[1](e_0) = v_3 \\
        \esources[0](e_1) = v_0 \quad \esources[1](e_1) = v_1 \quad \etargets[0](e_1) = v_3 \\
        \esources[0](e_2) = v_3 \quad \etargets[0](e_2) = v_4 \quad \etargets[1](e_2) = v_1 \\
        \labels = \{e_0 \mapsto {\fork}, e_1 \mapsto {\join}, e_2 \mapsto {\fork}\}
      \end{gather*}
    \end{minipage}
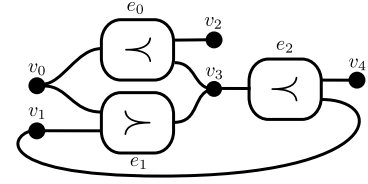
    \begin{minipage}{0.35\textwidth}
      \begin{center}
        \includesvg[scale=0.75]{example-simple-annotated}
      \end{center}
    \end{minipage}
  \end{center}
\end{example}

\noindent\textbf{Category.}
A labelled hypergraph homomorphism $\morph{h}{F}{G}$ consists of functions $\morph{\vmap{h}}{V_F}{V_G}$ and, for each $k,l \in \nat$, $\morph{\emap{h}}{E_F[k,l]}{E_G[k,l]}$ such that sources, targets and labels are preserved. 

\begin{center}

\begin{tikzcd}[row sep = large]
    E_F[k,l] \arrow{r}{\esources_F[i]} \arrow{d}{\emap{h}} & {V_F} \arrow{d}{\vmapm{h}} \\
    E_G[k,l] \arrow{r}{\esources_G[i]} & {V_G} 
\end{tikzcd}

\begin{tikzcd}[row sep = large]
    E_F[k,l] \arrow{r}{\etargets_F[i]} \arrow{d}{\emap{h}} & {V_F} \arrow{d}{\vmapm{h}} \\
    E_G[k,l] \arrow{r}{\etargets_G[i]} & {V_G} 
\end{tikzcd}

\begin{tikzcd}[row sep = large]
    E_F \arrow{r}{\labels_F} \arrow{d}{\emap{h}} & {\ehypsig} \arrow{d}{\id} \\
    E_G \arrow{r}{\labels_G} & {\ehypsig} 
\end{tikzcd}

\end{center}

\noindent If $\vmap{h}$ and $\emap{h}$ are bijective then $F$ and $G$ are isomorphic $F \equiv G$. 
It is immediate that $\equiv$ is an equivalence relation, and we quotient hypergraphs by it.

Hypergraph homomorphisms are the morphisms in the category of hypergraphs $\hyp$, a functor category~\cite{bonchi2016rewriting}.
Hypergraph signatures $\hypsig$ can be seen as hypergraphs, with a vertex $v$ and edges for each label $m \to n$ in the signature, with $v$ appearing $m$ (resp. $n$) times in its sources (resp. targets). 
Thus labelled hypergraphs are defined as a slice category. 

\begin{definition}[\texorpdfstring{Category of hypergraphs~\cite{bonchi2016rewriting}}{Category of hypergraphs}]\label{def:hyp}
  Let $\hyp$ be the functor category $[\textbf{X},\textbf{Set}]$, where $\textbf{X}$ has as objects pairs of natural numbers $(m,n)$ and an extra object $\star$. For each object $x = (m,n)$, there are $m+n$ arrows from $x$ to $\star$. Let $\labhyp = \hyp / \hypsig$ be the slice category over a hypergraph signature $\hypsig$.
\end{definition}

\noindent We call a hypergraph homomorphism an \emph{embedding} if its components are injective.

\begin{lemma}\label{lem:hyp-monos}
  A morphism in $\labhyp$ is a monomorphism if and only if is an embedding.
\end{lemma}
\begin{proof}
  For a morphism $\morph{m}{F}{G}$ to be mono, for any two morphisms $\morph{p,q}{H}{F}$ (for any other hypergraph $H$), if $m \circ p = m \circ q$ then $p = q$. First we show that if $m$ is an embedding it must be mono. If we consider each equivalence map of $m$ separately, this means that we must show that $\vmap{m} \circ \vmap{p} = \vmap{m} \circ \vmap{q}$ implies $\vmap{p} = \vmap{q}$ (and the same for $\emap{m}$). But we have assumed that $\vmapt{m}$ is injective, so the antecedent reduces to $\vmapt{g} = \vmapt{h}$. So $m$ is mono.

    Conversely, if $m$ is not an embedding, it cannot be a monomorphism. A morphism that is not an embedding maps multiple vertices or edges into one. Therefore for a morphism $\morph{m}{F}{G}$ that maps vertices $v_1$ and $v_2$ in $F$ to $v$ in $G$, there exist two morphisms $\morph{p,q}{G}{F}$ (in $p$, $v \mapsto v_1$ and in $q$, $v \mapsto v_2$), and similar for morphisms that map multiple edges to one. Therefore there exist $p,q$ such that $p \neq q$, so $m$ is not mono.
\end{proof}

\subsection{Linear hypergraphs}\label{sec:linear-hypergraphs}

In hypergraphs, vertices can connect to an arbitrary number of edges.
However, to make wires in string diagrams split or join, an additional Frobenius structure must be imposed.
This structure works particularly well in the framework of compact closed categories, but this is a structure which we aim to avoid. 
Therefore we must restrict hypergraphs so that the in-degree and out-degree of each vertex is at most one: vertices with in-degree $0$ represent the inputs of the term and vertices with out-degree $0$ represent the outputs of the term.
We call a hypergraph \emph{linear} if this condition is satisfied.

While we identified the input and output vertices of the hypergraph above, they are not ordered, and thus we do not have a true `interface'.
One option is to identify the interfaces by means of certain (ordered) cospans, as in \cite{bonchi2016rewriting}, but we take an alternative approach in which we build our interfaces directly into our hypergraphs by means of an additional interface edge $\interface$.
We write the set of edges and this interface as $E + 1$. 

We could simply add this edge to our existing definition of hypergraphs.
However, we wish to define a \emph{sound and complete} graphical language: we want \emph{every} diagram to correspond to a term in our category.
Therefore we take this opportunity to reformulate our definition of hypergraphs, yielding \emph{interfaced linear hypergraphs}.
In Section~\ref{sec:graph-rewriting} we shall see how our definition can be related the more traditional definition.

\begin{definition}[Interfaced linear hypergraph]\label{def:interfaced-linear-hypergraphs}
    An interfaced linear hypergraph is a tuple $H = (\edges,\vs,\vt,\vconnsr)$ where
    \begin{itemize}
        \item A finite set $\edges \subset \atoms$ of edges
        \item $\vs,\vt$ are finite sets of containing, for each edge $e \in \edges + 1$, finite totally ordered sets of source and target vertices $\vs[e], \vt[e] \subset \atoms$, such that for any $V_1,V_2 \in \vs \cup \vt$,  $V_1 \cap V_2 = \emptyset$.
        \item $\morph{\vconnsr}{\bigcup_{e\in\edges}\vt[e]}{\bigcup_{e\in\edges}\vs[e]}$ is a connections bijection between targets and sources.
    \end{itemize}
\end{definition}

\noindent We split vertices into sets of sources $S$ and targets $T$, with a connections bijection $\vconnsr$ between them.
The ordering of sources and targets of edges is determined by the order of the sets.
Splitting the vertices in this way allows us to enforce that each vertex is the source and target of only one edge while still retaining the order of sources and targets for each edge.
It also simplifies the operations defined below: for example, when we compose hypergraphs we `coalesce' the outputs of $F$ and inputs of $G$ together, as explained in Section~\ref{sec:sequential}.
With one set of vertices, we could delete the outputs of $F$ and the inputs of $G$, and then define `fresh' vertices as the bridge between the two hypergraphs.
However, we would have to redefine the orders on the vertices such that the ordering on the non-interface edges was preserved.
Keeping sources and targets separate eliminates this problem.
To simplify notation when talking about members of the source and target sets, we will use lower case variables to denote a single vertex, i.e. $s \in \vs$ means $s \in \bigcup_{e \in E}\vs[e]$.
$\vt[\interface]$ is the set of \emph{inputs}, and $\vs[\interface]$ is the set of of \emph{outputs}.

For an interfaced linear hypergraph $H$ with $m$ inputs and $n$ outputs, we write it as $\morph{H}{m}{n}$, where $m \to n$ is the \emph{type} of the hypergraph.
As with simple hypergraphs, we can define \emph{labelled linear hypergraphs} over a signature $\hypsig$ with labelling function $\labels$, where $\labels(e) = \phi$ is valid only if $|\vs[e]| = \dom{\phi}$ and $|\vt[e]| = \cod{\phi}$.

\begin{example}
    A linear hypergraph over $\hypsig$ can be drawn in two ways, illustrated below. 
    In a more formal notation (left), edges are stacked with their ordered source (resp. target) vertex sets on the left (resp. right) of the diagram.
    Connections are represented by the arrow on the far right.
    We represent the inputs (resp. outputs) of the term as incident to a grey edge labelled $\einput$ (resp. $\eoutput$).

    A more intuitive representation (right) is similar to how we drew hypergraphs earlier, where we draw connected target and source vertices as a single black dot. 
    The orders on the vertices dictate the position of each vertex's connection to an edge. 
    The more formal representation can be unambiguously recovered from the more intuitive one. 
    
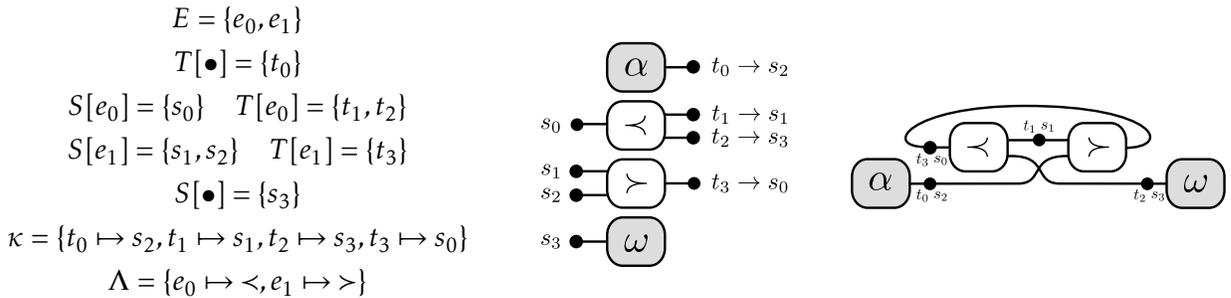
\vspace{-2em}

    \begin{center}
    \begin{minipage}{0.4\textwidth}
      \begin{gather*}
        E = \{e_0, e_1\} \\ 
        \vt[\interface] = \{t_0\} \\ 
        \vs[e_0] = \{s_0\} \quad \vt[e_0] = \{t_1,t_2\} \\ 
        \vs[e_1] = \{s_1,s_2\} \quad \vt[e_1] = \{t_3\} \\ 
        \vs[\interface] = \{s_3\} \\
        \vconnsr = \{t_0 \mapsto s_2, t_1 \mapsto s_1, t_2 \mapsto s_3, t_3 \mapsto s_0\} \\
        \labels = \{e_0 \mapsto {\fork}, e_1 \mapsto {\join}\}
      \end{gather*}
    \end{minipage}
    \quad
    \raisebox{-0.8em}{\begin{minipage}{0.5\textwidth}
      \raisebox{-0em}{\includesvg[scale=0.6]{example-formal}}
      \qquad
      \raisebox{2em}{\includesvg[scale=0.6]{example-annotated}}
    \end{minipage}}
  \end{center}
\end{example}

\vspace{1em}

\noindent\textbf{Category.} A (labelled) linear hypergraph homomorphism $\morph{h}{F}{G}$ consists of functions 
\[\morph{\vmaps{h}}{\bigcup_{e \in E_F} \vs_F[e]}{\bigcup_{e \in E_G} \vs_G[e]} \qquad \morph{\vmapt{h}}{\bigcup_{e \in E_F} \vt_F[e]}{\bigcup_{e \in E_G} \vt_G[e]} \qquad \morph{\emap{h}}{E_F}{E_G}\] between sources, targets and edges, such that the first four diagrams below commute.
If $\vmapt{h}$, $\vmaps{h}$ and $\emap{h}$ are bijective and the latter two diagrams below also commute, then $F$ and $G$ are \emph{isomorphic} written $F \equiv G$.
It is immediate that $\equiv$ is an equivalence relation, and we quotient labelled interfaced linear hypergraphs by it.

\begin{center}
  $\underbrace{
      \overbrace{

\begin{tikzcd}[row sep = large]
    E_F \arrow{r}{\vs[-][F]} \arrow{d}{\emap{h}} & {\vs}_F \arrow{d}{\vmapss{h}} \\
    E_G \arrow{r}{\vs[-][G]} & {\vs_G} 
\end{tikzcd}

 \,

\begin{tikzcd}[row sep = large]
    E_F \arrow{r}{\vt[-][F]} \arrow{d}{\emap{h}} & {\vt_F} \arrow{d}{\vmapts{h}} \\
    E_G \arrow{r}{\vt[-][G]} & {\vt_G} 
\end{tikzcd}

 \,
          \raisebox{-0.1em}{
    
\begin{tikzcd}[row sep = large]
    {\bigcup\vt_F} \arrow{r}{\vconnsr_F} \arrow{d}{\vmapt{h}} & {\bigcup\vs_F} \arrow{d}{\vmaps{h}} \\
    {\bigcup\vt_G} \arrow{r}{\vconnsr_G} & {\bigcup\vs_G} 
\end{tikzcd}

} \,

\begin{tikzcd}[row sep = large]
    E_F \arrow{r}{\labels_F} \arrow{d}{\emap{h}} & {\ehypsig} \arrow{d}{\id} \\
    E_G \arrow{r}{\labels_G} & {\ehypsig} 
\end{tikzcd}

 \,
      }^{\text{homomorphism}}

\begin{tikzcd}[row sep = large]
    1\arrow{r}{\vt[-][F]} \arrow{d}{\id} & {\vt_F} \arrow{d}{\vmapts{h}} \\
    1\arrow{r}{\vt[-][G]} & {\vt_G} 
\end{tikzcd}

 \,

\begin{tikzcd}[row sep = large]
    1 \arrow{r}{\vs[-][F]} \arrow{d}{\id} & {\vs_F} \arrow{d}{\vmapss{h}} \\
    1 \arrow{r}{\vs[-][G]} & {\vs_G} 
\end{tikzcd}

  }_{\text{equivalence}}$
\end{center}

\noindent Labelled interfaced linear hypergraphs form a category $\lilhyp$ with objects the labelled interfaced linear hypergraphs over signature $\hypsig$ and morphisms the labelled interfaced linear hypergraph homomorphisms.

\begin{lemma}\label{lem:lilhyp-monos}
  A morphism in $\lilhyp$ is mono if and only if it is an embedding.
\end{lemma}
\begin{proof}
  As with simple hypergraphs (Lemma \ref{lem:hyp-monos}).
\end{proof}

\noindent We will now use the term `hypergraph' to mean `interfaced linear hypergraph' unless specified.

\section{Operations and constructs}\label{sec:operations}

We can create hypergraphs compositionally using the operations of an STMC: composition, monoidal tensor, symmetry and trace.
In this section we will detail their definitions, in addition to some other important components of our hypergraph framework.

When performing operations, it is imperative that our hypergraphs do not become degenerate.
We call a hypergraph \emph{well-formed} if for any $V_1,V_2 \in \vs \cup \vt$, $V_1 \cap V_2 = \emptyset$, $\vconnsr$ is bijective, and the labelling condition is satisfied.
Some of the more bureaucratic proofs in this section have been omitted: to find them the interested reader can turn to Appendix~\ref{sec:appendix}.

\subsection{Equivariance}\label{sec:equivariance}

When performing operations on hypergraphs, the vertices and edges of the hypergraphs involved must be disjoint so that we do not create degenerate hypergraphs. 
However, this is not always the case, such as when composing a hypergraph with itself. 
Fortunately, since the sets of vertices and edges are subsets of the countably infinite set of atoms $\atoms$, we can simply \emph{rename} the problematic edges or vertices~\cite{pitts2013nominal}. 

\begin{definition}[Action]\label{def:action}
  For any permutation $\morph{\rename}{\atoms}{\atoms}$, an action $\rename \permaction -$ acts as follows:

  \begin{description}
      \item[Element] For any elements $x \in \atoms$, $y \not\in \atoms$, $\rename \permaction x = \rename(x)$ and $\rename \permaction y = y.$
      \item[Set] For any set $X$, $\rename \permaction X = \{\rename \permaction x \,|\, x \in X\}.$ 
      \item[Totally ordered sets] As with regular sets, preserving the order i.e. if $x < y$ then $\rename \permaction x < \rename \permaction y$  
      \item[Function] For any function $\morph{f}{X}{Y}$, $(\rename \permaction f)(v) = \rename \permaction f(\renameinv \permaction v).$
  \end{description}

\end{definition}

\begin{definition}[Renaming]\label{def:renaming}
  For any labelled interfaced linear hypergraph \[F = \lilhyper\] and for any permutation $\morph{\rename}{\atoms}{\atoms}$ we can apply $\rename$ to $F$ to rename it: 
  \[\rename \permaction F = (\rename \permaction E, \rename \permaction \vs, \rename \permaction \vt, \rename \permaction \vconnsr, \rename \permaction \labels).\] 
\end{definition}

\begin{proposition}[Equivariance of hypergraphs]\label{prop:equivariance}
  For any labelled interfaced linear hypergraph $\morph{F}{m}{n}$ and permutation $\morph{\rename}{\atoms}{\atoms}$, $\rename \permaction F \equiv F$.
\end{proposition}
\begin{proof}
  \[\vmapt{h}(v) = \renameinv \permaction v \qquad \vmaps{h}(v) = \renameinv \permaction v \qquad \emap{h}(e) = \renameinv \permaction e\]
\end{proof}

\noindent Therefore the definition of $F$ is equivariant under name permutations, so we are justified in renaming vertices and edges `on the fly'. 
Since we use graphs up to isomorphism, this will implicitly also quotient by equivariance.

\subsection{Composition}\label{sec:sequential}

To compose hypergraphs sequentially, we `redirect' any vertices that connected to the output of the first hypergraph to those originally connected to the input of the second hypergraph.
Graphically, we juxtapose the hypergraphs horizontally:

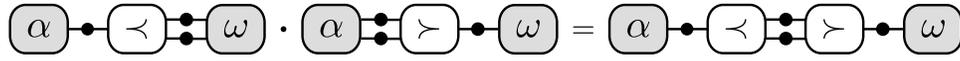
\begin{center}
  \includesvg[scale=0.6]{operations/sequential/sequential}
\end{center}

\begin{definition}[Composition]\label{def:composition}
  For any two labelled interfaced linear hypergraphs $\morph{F}{m}{n}$ and $\morph{G}{n}{p}$ over signature $\Sigma$:
  \[\morph{F}{m}{n} = \lilhyper[F] \qquad \morph{G}{n}{p} = \lilhyper[G]\]
  \noindent we define their composition as follows.
  \[\morph{H}{m}{p} = F \seq G = \lilhyper[H]\]
\end{definition}

\noindent The new set of edges is simply the disjoint union of the edges in $F$ and $G$, and we do the same for the labelling function.
To obtain the new sets of vertices, we delete the outputs of $F$ and the inputs of $G$.
\begin{gather*}
  E_H = E_F + E_G \qquad \labels_H = \labels_F + \labels_G \\
  \vs[\interface][H] = \vs[\interface][G] \qquad 
  \vs[e \in E_F][H] = \vs[e][F] \qquad
  \vs[e \in E_G][H] = \vs[e][G] \\
  \vt[\interface][H] = \vt[\interface][F] \qquad
  \vt[e \in E_F][H] = \vt[e][F] \qquad 
  \vt[e \in E_G][H] = \vt[e][G]
\end{gather*}
The connections function maps a vertex connected to the $i$th output of $F$ to the vertex connected to the $i$th input of $G$.
\[\vconnsr_H(v) = \begin{cases}
  \vconnsr_G(\proj{i}(\vt[\interface][G])) & \text{if}\ \vconnsr_F(v) = \proj{i}(\vs[\interface][F]) \\
  \vconnsr_F + \vconnsr_G & \text{otherwise}
\end{cases}\]
This can be drawn formally as follows:

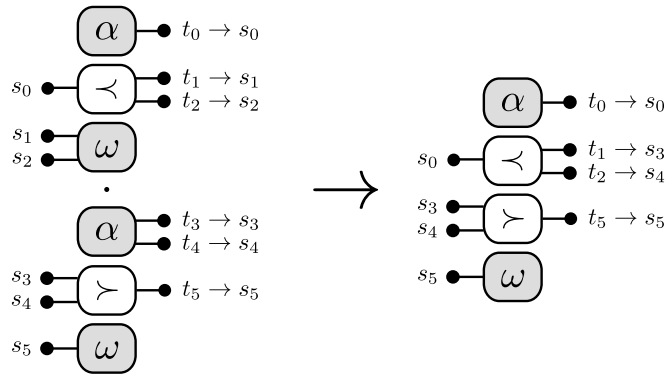
\begin{center}
  \includesvg[scale=0.6]{operations/sequential/sequential-formal}
\end{center}

\begin{proposition}[Well-formedness of composition]\label{prop:well-formed-composition}
  For two labelled interfaced linear hypergraphs $\morph{F}{m}{n}$ and $\morph{G}{n}{p}$, $F \seq G$ is a well-formed labelled interfaced linear hypergraph.
\end{proposition}
\begin{proof}
  We have only removed vertices so the sources and targets must still be disjoint.
  The only change in the connections function means that the vertices that originally connected to the output vertices of $F$ (which have been deleted) now connect to the sources originally connected to the input vertices of $G$ (which have also been deleted), so $\vconnsr$ is bijective.
  The incidence of vertices on regular edges is also unaffected, so the labelling condition is satisfied.
\end{proof}

\noindent The unit of composition is the identity hypergraph, a hypergraph where all vertices are the sources and the targets of the interface.
Below are examples for $n = 1$ and $n=2$.

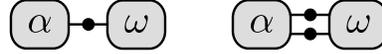
\begin{center}
  \includesvg[scale=0.6]{operations/sequential/identity} \qquad   \includesvg[scale=0.6]{operations/sequential/identity-2}
\end{center}

\begin{definition}[Identity hypergraph]\label{def:identity-hypergraph}
  An identity hypergraph $\morph{\id[n]}{n}{n}$ over signature $\Sigma$ is defined as
  \[\id[n] = (\{\vs[\interface]\}, \{\vt[\interface]\}, \emptyset, \vconnsr, \emptyset)\] 
  where $\vs[\interface], \vt[\interface] \subset \atoms$ are finite disjoint totally ordered sets, $|\vs[\interface]| = |\vt[\interface]| = n$, and $\vconnsr(\proj{i}(\vt[\interface])) = \proj{i}(\vs[\interface])$.
\end{definition}

\subsection{Monoidal tensor}\label{sec:monoidal-tensor}

We can also compose hypergraphs in parallel, which is known as their \emph{monoidal tensor}.
We simply combine their input and outputs, and leave everything else untouched.
Graphically, we can represent this by juxtaposing them vertically.

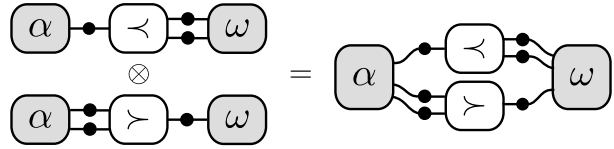
\begin{center}
  \includesvg[scale=0.6]{operations/parallel/parallel}
\end{center}

\begin{definition}[Monoidal tensor]\label{def:monoidal-tensor}
  For any two labelled interfaced linear hypergraphs $\morph{F}{m}{n}$ and $\morph{G}{p}{q}$ over a signature $\Sigma$:
  \[F = \lilhyper[F] \qquad G = \lilhyper[G]\]
  we define their monoidal tensor as follows.
  \[\morph{H}{m+p}{n+q} = F \tensor G = \lilhyp[H]\]
\end{definition}

\noindent Once again, the edges and labels are the union of those in $F$ and $G$.
\[E_H = E_F + E_G \qquad \labels_H = \labels_F + \labels_G \]
We do not need to delete any vertices, only combine the interfaces.
\[\vs[e \in E_F][H] = \vs[e][F] \quad \vs[e \in E_G][H] = \vs[e][G] \quad \vs[\interface][H] = \vs[\interface][F] + \vs[\interface][G]\]
\[\vt[e \in E_F][H] = \vt[e][F] \quad \vt[e \in E_G][H] = \vt[e][G] \quad \vt[\interface][H] = \vt[\interface][F] + \vt[\interface][G]\]
Subsequently the connections function is just the union of those in $F$ and $G$.
\[\vconnsr_H = \vconnsr_F + \vconnsr_G\]
This can be drawn formally as follows:

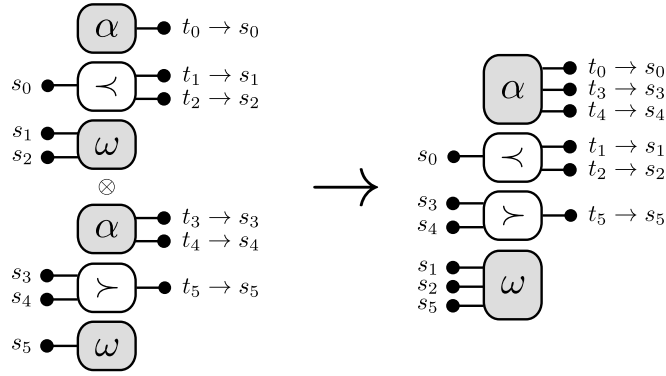
\begin{center}
  \includesvg[scale=0.6]{operations/parallel/parallel-formal}
\end{center}

\begin{proposition}[Well-formedness of tensor]\label{prop:well-formed-tensor}
  For any two labelled interfaced linear hypergraphs $\morph{F}{m}{n}$ and $\morph{G}{p}{q}$, $F \tensor G$ is a well-formed labelled interfaced linear hypergraph.
\end{proposition}
\begin{proof}
  We have not added any new vertices, so the sources and targets are still disjoint.
  The connections of each hypergraph are unaffected, so $\vconnsr$ is a bijection.
  Likewise, we have not interfered with the sources and targets of regular edges, so the labelling condition is satisfied.
\end{proof}

\noindent The unit of monoidal tensor is the empty hypergraph (an identity hypergraph on $0$).
This is simply a hypergraph with no edges or vertices.
Graphically this is represented as two empty interfaces.

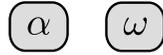
\begin{center}
  \includesvg[scale=0.6]{operations/parallel/empty}
\end{center}

\begin{definition}[Empty hypergraph]\label{def:empty-hypergraph}
  The empty hypergraph $\morph{\id[0]}{0}{0}$ over signature $\Sigma$ is defined as 
  \[\id[0] = (\emptyset,\emptyset,\emptyset,\emptyset,\emptyset)\]
\end{definition}

\noindent $- \tensor -$ is a bifunctor, so there may be multiple orders in which we can perform sequential composition or monoidal tensor that still result in the same hypergraph.

\begin{proposition}[Bifunctoriality I]\label{prop:bifunctoriality-1}
  For any $m,n \in \nat$, $\id[m] \tensor \id[n] \equiv \id[{m+n}]$
\end{proposition}

\begin{center}
  \includesvg[scale=0.6]{axioms/bifunctoriality-1}
\end{center}

\begin{proposition}[Bifunctoriality II]\label{prop:bifunctoriality-2}
  For any labelled interfaced linear hypergraphs $\morph{F}{m}{n}$, $\morph{G}{r}{s}$, $\morph{H}{n}{p}$, $\morph{K}{s}{t}$, $F \tensor G \seq H \tensor K \equiv (F \seq H) \tensor (G \seq K)$.
\end{proposition}

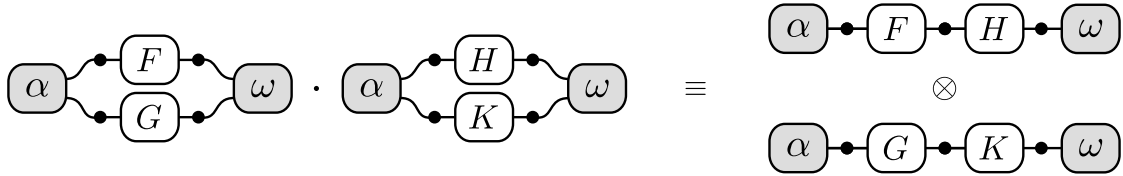
\begin{center}
  \includesvg[scale=0.6]{axioms/bifunctoriality-2}
\end{center}

\subsection{Symmetry}\label{sec:symmetry}

To swap the orders of vertices in the interfaces, we require a new construct, named the \emph{swap} hypergraph. 
This hypergraph swaps over two wires.

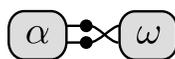
\begin{center}
  \includesvg[scale=0.6]{operations/symmetry/symmetry}
\end{center}

\begin{definition}[Swap hypergraph]\label{def:swap-hypergraph}
  The swap hypergraph for two wires $\swap{1}{1}$ is defined as
  \[\swap{1}{1} = (\emptyset, \{\vs[\interface]\}, \{\vt[\interface]\}, \vconnsr, \emptyset)\]
  \noindent where $\vs[\interface] = \{\mf{c},\mf{d}\}$, $\vt[\interface] = \{\mf{a}, \mf{b}\}$, $\mf{a},\mf{b},\mf{c},\mf{d} \in \atoms$, $\vconnsr(\mf{a}) = \mf{d}$, $\vconnsr(\mf{b}) = \mf{c}$. 
\end{definition}

\noindent By composing multiple copies of the swap hypergraph in sequence and parallel we can build up constructs in which we swap many wires.

\begin{definition}[Composite swap]\label{def:composite-swap}
  For any $m,n \in \nat$, we can define a composite swap hypergraph as follows. \[\morph{\swap{m}{n}}{m~+~n}{n~+~m}\]
\end{definition}

\begin{center}
  \[\swap{0}{n} = \id[n]\]
  \includesvg[scale=0.6]{operations/symmetry/symmetry-0-n}
  \[\swap{m}{0} = \id[m]\]
  \includesvg[scale=0.6]{operations/symmetry/symmetry-m-0}
  \[\swap{1}{1} = \swap{1}{1}\]
  \includesvg[scale=0.6]{operations/symmetry/symmetry}
  \[\swap{m+1}{1} = m \tensor \swap{1}{1} \seq \swap{m}{1} \tensor 1\]
  \includesvg[scale=0.6]{operations/symmetry/symmetry-m+1-1}
  \[\swap{1}{n+1} = \swap{1}{n} \tensor 1 \seq n \tensor \swap{1}{1}\]
  \includesvg[scale=0.6]{operations/symmetry/symmetry-1-n+1}
  \[\swap{m+1}{n+1} = 1 \tensor \swap{1}{n} \tensor 1 \seq \swap{m}{n} \tensor \swap{1}{1} \seq n \tensor \swap{m}{1} \tensor 1\]
  \includesvg[scale=0.6]{operations/symmetry/symmetry-m+1-n+1}
\end{center}

\noindent For some proofs, it may be preferential to represent composite swaps in a non-inductive way, and instead think in terms of swapping the sets of the input and output vertices.

\begin{lemma}[Alternate swap]\label{lem:alternate-swap}
  For any $m,n\in\nat$, aa composite swap hypergraph $\swap{m}{n}$ can be written in the form \[\swap{m}{n} = (\emptyset, \{\vs[\interface]\}, \{\vt[\interface]\}, \vconnsr, \emptyset)\] where $\mf{A},\mf{B},\mf{C},\mf{D} \subset \atoms$ are disjoint sets such that $|\mf{A}| = |\mf{D}| = m$, $|\mf{B}| = |\mf{C}| = n$, $\vs[\interface] = \mf{C} + \mf{D}$, $\vs[\interface] = \mf{A} + \mf{B}$, and \[\vconnsr(\proj{i}(\mf{A})) = \proj{i}(\mf{D}) \qquad \vconnsr(\proj{i}(\mf{B})) = \proj{i}(\mf{C})\]
\end{lemma}

\noindent Composite swap hypergraphs are \emph{natural}: we can `push through' hypergraphs composed on either side.

\begin{proposition}[Naturality of swap]\label{prop:naturality-swap}
  For $m,n,p,q \in \nat$ and labelled interfaced linear hypergraphs $\morph{F}{m}{n}$ and $\morph{G}{p}{q}$, \[F \tensor G \seq \swap{n}{q} \equiv \swap{m}{p} \seq G \tensor F\]
  
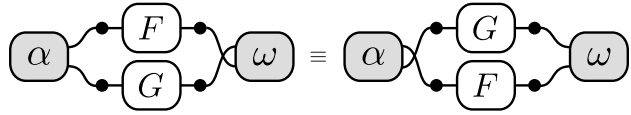
  \begin{center}
    \includesvg[scale=0.6]{axioms/naturality-swap}
  \end{center}

\end{proposition}

\subsection{Homeomorphism}\label{sec:homeomorphism}

The operations so far have been fairly straightforward.
However, a subtlety arises when we consider the trace.
A naive approach to performing $\trace{x}{F}$ would be to take the first $x$ inputs and outputs and join them together.
Now consider the trace of the identity: one might assume that $\trace{x}{\id[x]} = \id[0]$ as it is simply a closed loop and does not `affect' the term per se, but this is is not always the case \cite[\S6.1]{hasegawa2012models}.
So we cannot discard these loops, but we cannot represent closed loops in vanilla hypergraphs as vertices can only connect to edges.

This issue arises because we have `absorbed' the identity morphisms, so to solve this problem we introduce the notion of \emph{homeomorphism} to create \emph{identity edges} $1 \to 1$, drawn as grey diamonds.
We write the set of identity edges in a hypergraph as $E[\id]$, and as such our hypergraph tuple becomes $H = (E, E[\id], \vs,\vt,\vconnsr,\labels)$.
The sources and targets of these identity edges must be preserved by homomorphism.
In general, we can introduce or remove identity edges at will by performing an \emph{expansion} or \emph{smoothing} respectively.

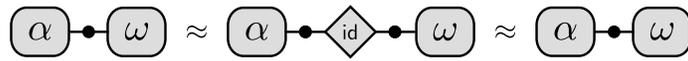
\begin{center}
  \includesvg[scale=0.6]{operations/homeomorphism/expansion-smoothing}
\end{center}

\noindent The only exception is when the source and target vertex of the identity edge are connected.
Performing a smoothing here would create an invalid hypergraph.

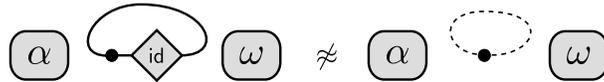
\begin{center}
  \includesvg[scale=0.6]{operations/homeomorphism/smoothing-invalid}
\end{center}

\begin{definition}[Expansion]\label{def:expansion}
  For a labelled interfaced linear hypergraph $F = (E_F, E_F[\id], \vs_F, \vt_F, \vconnsr_F, \labels_F)$, and $s \in \vs_F, t \in \vt_F$ such that $\vconnsr_F(t) = s$, we can perform an expansion on $(t, s)$ to yield hypergraph \[G = (E_F, E_F[{\id}] + \{e_{\id}\}, \vs_F + \vs[e_{\id}], \vt_F + \vt[e_{\id}], \vconnsr_H, \labels_F)\] with $e_{\id} \text{ fresh in } \atoms$, $\vs[e_{\id}] = \{s'\}$, $\vt[e_{\id}] = \{t'\}$, $s',t' \text{ fresh in } \atoms$ and \[\vconnsr_H(v) = \begin{cases}
    s' & \text{if}\ v = t \\
    s  & \text{if}\ v = t' \\
    \vconnsr_F(v) & \text{otherwise}
  \end{cases}\]
\end{definition}

\begin{proposition}[Well-formedness of expansion]\label{prop:well-formed-expansion}
  For any labelled interfaced linear hypergraph $F$ containing target vertex $v$ and source vertex $s$, where $\vconnsr(t) = s$, the result of performing an expansion on $(t,s)$ is a well formed labelled interfaced linear hypergraph.
\end{proposition}
\begin{proof}
  All introduced vertices are fresh in $\atoms$, so the sources and targets are disjoint.
  One target vertex $v$ connects to the fresh source vertex, and the fresh target vertex connects to its original connection $\vconnsr(v)$, so $\vconnsr$ is bijective.
  The regular edges are unaffected, so the labelling condition is satisfied.
\end{proof}

\begin{definition}[Smoothing]\label{def:smoothing}
  For a labelled interfaced linear hypergraph $F = (E_F, E_F[\id] + \{e_{\id}\}, \vs_F, \vt_F, \vconnsr_F, \labels_F)$, where $\vs[e_{\id}] = \{s\}$ and $\vt[e_{\id}] = \{t\}$, $\vconnsr_F(t) \neq s$, we can perform a smoothing on $e_{\id}$ to yield hypergraph 
  \[G = (E_F, E_F[{\id}], \vs_F - \vs[e_{\id}][F], \vt_F - \vt[e_{\id}][F], \vconnsr_H, \labels_F)\]
  \[\vconnsr_H(v) = \begin{cases}
    \vconnsr_F(t) & \text{if}\ \vconnsr(v) = t \\
    \vconnsr_F(v) &  \text{otherwise}\
  \end{cases}\] 
\end{definition}

\begin{proposition}[Well-formedness of smoothing]\label{prop:well-formed-smoothing}
  For any labelled interfaced linear hypergraph containing an identity edge $e_{\id}$ with source $s$ and target $t$, $\vconnsr(t) \neq s$, the result of performing a smoothing on $e_{\id}$ is a well formed labelled interfaced linear hypergraph.
\end{proposition}
\begin{proof}
  We only remove vertices, so the sources and targets are disjoint.
  The change in the connections is that the target vertex that original connected to the source of the identity edge now redirects to the source vertice connected to by the target of the identity edge, so $\vconnsr$ is bijective.
  The regular edges are unaffected, so the labelling condition is satisfied.
\end{proof}

\noindent We call an interfaced linear hypergraph \emph{minimal} if no smoothings can be performed.
We quotient hypergraphs by homeomorphism and always draw the minimal version.

\subsection{Trace}\label{sec:trace}

Now equipped with homeomorphism, we can define a suitable trace operation.
To trace a hypergraph with one wire, we create a new identity edge with $\proj{0}(\vs[\interface][F])$ and $\proj{0}(\vt[\interface][F])$ as its source and target respectively.

\begin{center}
  \includesvg[scale=0.6]{operations/trace/full-trace}
\end{center}

\noindent The use of the identity edge ensures that we can represent the trace of the identity as a valid hypergraph.

\begin{center}
  \includesvg[scale=0.6]{operations/trace/full-trace-loop}
\end{center}

\noindent Tracing multiple wires is performed inductively. 
The trace of no wires is equal to the original hypergraph.

\begin{center}
  \includesvg[scale=0.6]{operations/trace/trace-0}
\end{center}

\noindent To trace multiple wires, we simply trace one at a time.

\begin{center}
  \includesvg[scale=0.6]{operations/trace/trace-multiple}
\end{center}

\begin{definition}[Trace]\label{def:trace}
  For a labelled interfaced linear hypergraph \[\morph{F}{x + m}{x + n} = \lilhypere[F]\] we can recursively define its trace of $x$ wires $\morph{\trace{x}{F}}{m}{n}$ as \[\trace{0}{F} = F\]\[\trace{x+1}{F} = \trace{1}{\trace{x}{F}} \text{ for } x > 0\] with the base case defined as follows. \[H = \trace{1}{F} = \lilhypere[H]\] 
\end{definition}

\noindent To perform a trace, we must introduce one identity edge to join the first input and output together, and ensure this does not create a closed loop of wires.
Otherwise, the edges and labels remain the same.
\[E_H = E_F \qquad E_H[\id] = E_F[\id] + \{e_{\id}\} \quad e_{\id} \text{ fresh in } \atoms \qquad \labels_H = \labels_F\]
We take the first input and output vertex, and set them to be the target and source of the identity edge respectively.
\[\vs[\interface][H] = \vs[\interface][F] - \proj{0}(\vs[\interface][F]) \qquad \vs[e_{\id}] = \{\proj{0}(\vs[\interface][F])\} \qquad \vs[e \in E_F][H] = \vs[e][F] \]
\[\vt[\interface][H] = \vt[\interface][F] - \proj{0}(\vt[\interface][F]) \qquad \vt[e_{\id}] = \{\proj{0}(\vt[\interface][F])\} \qquad \vt[e \in E_F][H] = \vt[e][F]\]
Since we have not deleted any vertices, the connections bijection remains the same.
\[\vconnsr_H = \vconnsr_F\]

\noindent After this operation, we can smooth the term as much as possible to remove any redundant identity edges, while still preserving the loops that do not connect to any regular edges.

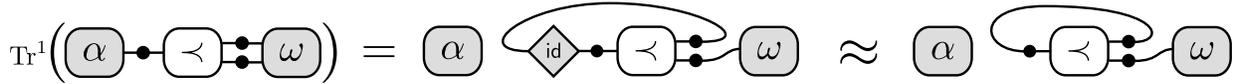
\begin{center}
  \includesvg[scale=0.6]{operations/trace/full-trace-homeo}
\end{center}

\noindent The whole procedure can be drawn formally as follows:

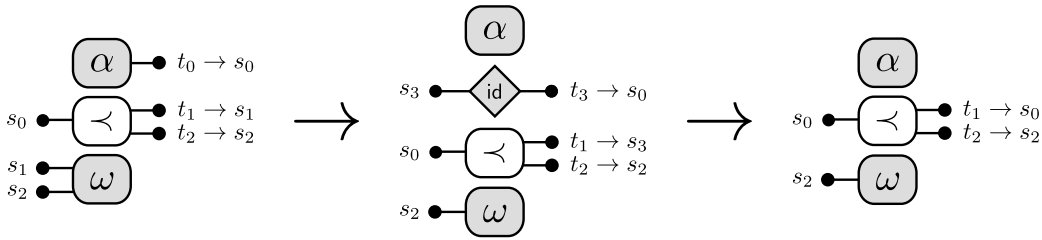
\begin{center}
  \includesvg[scale=0.6]{operations/trace/trace-formal}
\end{center}

\begin{proposition}[Well-formedness of trace]\label{prop:well-formed-trace}
  For any labelled interfaced linear hypergraphs $\morph{F}{x + m}{x + n}$, $\trace{x}{F}$ is a well-formed labelled interfaced linear hypergraph.
\end{proposition}
\begin{proof}
  We have only moved a source and target from one edge to the identity edge, so the sets are still disjoint.
  The connections bijection is unchanged, so $\vconnsr$ is bijective.
  The regular edges are unaffected, so the labelling condition is satisfied.
\end{proof}

\noindent With trace defined, this means that we have all the operations of a STMC defined in terms of interfaced linear hypergraphs.
Now we must show that hypergraphs equipped with these operations form a \emph{sound and complete} graphical language.

\section{Soundness}\label{sec:soundness}

We propose hypergraphs as a graphical language for STMCs. 
In particular we will focus on traced PROPs~\cite{lack2004composing}, categories with natural numbers as objects and addition as tensor product.
First we consider \emph{soundness}. 

We fix a traced PROP $\termtype$ of morphisms freely generated over a signature $\hypsig$, and assemble labelled interfaced linear hypergraphs into the traced PROP $\lhypterm$, in which the morphisms $m \to n$ are hypergraphs of type $m \to n$, with composition, tensor, symmetry and trace defined as above.

\begin{definition}[Interpretation functor]
  We define the interpretation functor from terms to labelled interfaced linear hypergraphs as the identity-on-objects traced monoidal functor $\morph{\hypfun{-}_\hypsig}{\termtype}{\lhypterm}$.
\end{definition}

\noindent We omit the subscript if unambiguous.
$\hypfun{-}_\Sigma$ is defined recursively over the syntax of the term.
For a generator $\morph{\phi}{m}{n}$, we interpret it as an edge with $m$ sources and $n$ targets.
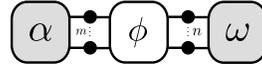
\begin{center}
  \includesvg[scale=0.6]{generator}
\end{center}
\noindent Formally, for a generator $\morph{\phi}{m}{n}$, this is defined as
\[\hypfun{\phi}_\Sigma = (\{\mf{e} \text{ fresh in } \atoms \}, \vs, \vt, \vconnsr, \labels)\]
\noindent where
\begin{gather*}
\vs[\interface] = \{s_i \text{ fresh in } \atoms \,|\, i < m\} \qquad \vs[\mf{e}] = \{s_{i+m} \text{ fresh in } \atoms \,|\, i < n\} \qquad \labels(\mf{e}) = \phi \\
\vt[\interface] = \{t_i \text{ fresh in } \atoms \,|\, i < m\} \qquad \vt[\mf{e}] = \{s_{i+m} \text{ fresh in } \atoms \,|\, i < n\} \qquad \vconnsr(t_i) = s_i\\
\end{gather*}

\noindent This can be drawn formally as follows.

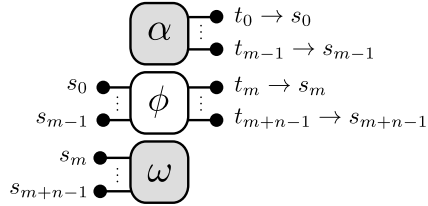
\begin{center}
  \includesvg[scale=0.6]{generator-formal}
\end{center}

\noindent Identity and symmetry morphisms translate into their hypergraph versions (Definitions~\ref{def:identity-hypergraph} and \ref{def:swap-hypergraph})
\[\hypfun{\id[n]} = \id[n] \qquad \hypfun{\swap{m}{n}} = \swap{m}{n}\]

\noindent To generate hypergraphs of larger terms, we can combine the morphism, identity and swap hypergraphs using composition and monoidal tensor, or by using the trace operator.
\[\hypfun{f \seq g} = \hypfun{f} \seq \hypfun{g} \qquad \hypfun{f} \tensor \hypfun{g} \qquad \hypfun{\trace{x}{f}} = \trace{x}{\hypfun{f}}\]

\begin{proposition}[Well-formedness]
  For any term $f$ in a traced PROP $\termtype$, $\hypfun{f}_\Sigma$ is a well-formed labelled interfaced linear hypergraph.
\end{proposition}
\begin{proof}
  Generator, identity and swap hypergraphs are well-formed, and all other operations involved create well-formed interfaced linear hypergraphs (Propositions \ref{prop:well-formed-composition}, \ref{prop:well-formed-tensor}, \ref{prop:well-formed-trace}).
\end{proof}

\noindent To show soundness, we must examine that the axioms of STMCs are satisfied in the language of labelled interfaced linear hypergraphs. as illustrated in Figure~\ref{fig:hypergraph-stmc}.

\begin{figure}
  \centering
  \includesvg[scale=0.55]{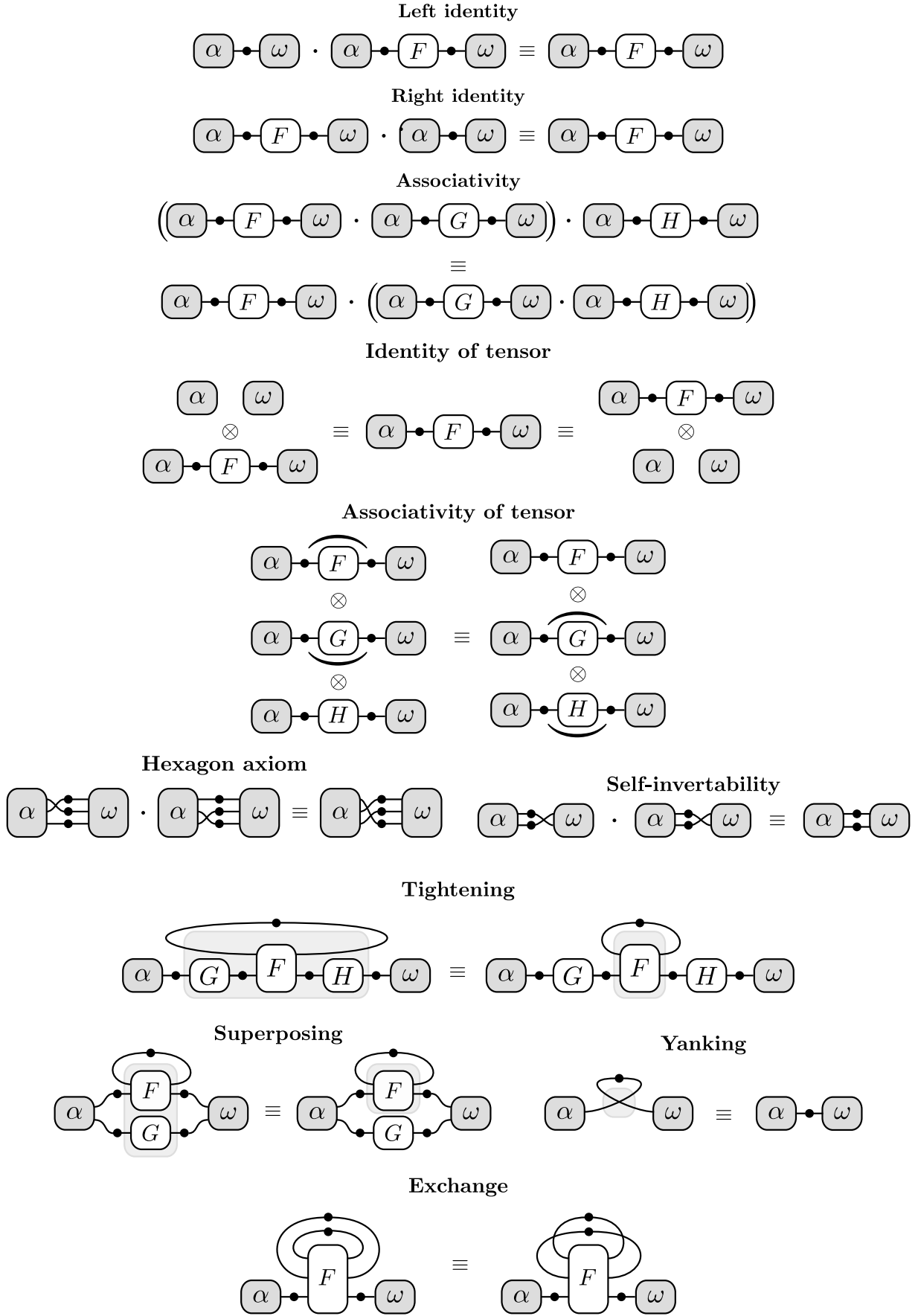}
  \caption{The axioms of STMCs, represented using labelled interfaced linear hypergraphs.}
  \label{fig:hypergraph-stmc}
\end{figure}

\begin{theorem}[Soundness]
  For any morphisms $f,g \in \termtype$, if $f = g$ under the equational theory of the category, then their interpretations as labelled interfaced linear hypergraphs are isomorphic $\hypfun{f} \equiv \hypfun{g}$.
\end{theorem}
\begin{proof}
  Composition produces well-formed interfaced linear hypergraphs (Proposition \ref{prop:well-formed-composition}) and satisfies the axioms of categories with the identity hypergraph $\morph{\id[n]}{n}{n}$ as the unit of composition for $n$. 
  Monoidal tensor produces well-formed interfaced linear hypergraphs (Proposition \ref{prop:well-formed-tensor}), is a bifunctor (Propositions \ref{prop:bifunctoriality-1} and \ref{prop:bifunctoriality-2}) and satisfies the axioms of (strict) monoidal categories with the empty hypergraph $\morph{0}{0}{0}$ as the monoidal unit. 
  The swap hypergraph is natural (Proposition \ref{prop:naturality-swap}) and satisfies the axioms of symmetric monoidal categories. 
  The trace operator produces well-formed interfaced linear hypergraphs hypergraphs (Proposition \ref{prop:well-formed-trace}) and satisfies the axioms of symmetric traced monoidal categories specified in Section \ref{sec:monoidal-categories}.
\end{proof}

\begin{remark}
  One may wonder if the axioms also hold arbitrary STMCs where the objects are not just natural numbers.
  The answer is yes -- the generalisation can be found in Section \ref{sec:generalisation}.
\end{remark}

\section{Completeness}\label{sec:completeness}

We are also able to recover categorical terms in an STMC from labelled interfaced linear hypergraphs. 
This is a two stage process: first we show that any well-formed labelled interfaced linear hypergraph has at least one corresponding categorical term (\emph{definability}); then we show that all of these terms are equal in the category (\emph{coherence}).

\subsection{Definability}

The strategy to retrieve a categorical term from a hypergraph is to exploit the formal graphical representation, in which all edges are `stacked'.
From this representation we can read off a tensor of generators, then connect wires of opposite polarities by linking them with trace and symmetries.
An example is shown in Figure \ref{fig:completeness}.

\begin{definition}[Definability]
  Labelled interfaced linear hypergraphs are definable if for every $F \in \lilhyp$, we can retrieve a well-formed categorical equation for which the hypergraph interpretation of that term is equivalent to the original graph, i.e. for a candidate $\morph{\term}{\lhypterm}{\termtype}$, then $\term[F] \equiv F$.
\end{definition}

\begin{figure}
  \centering
  \fontsize{6}{10}
  \raisebox{2em}{\includesvg[scale=0.6]{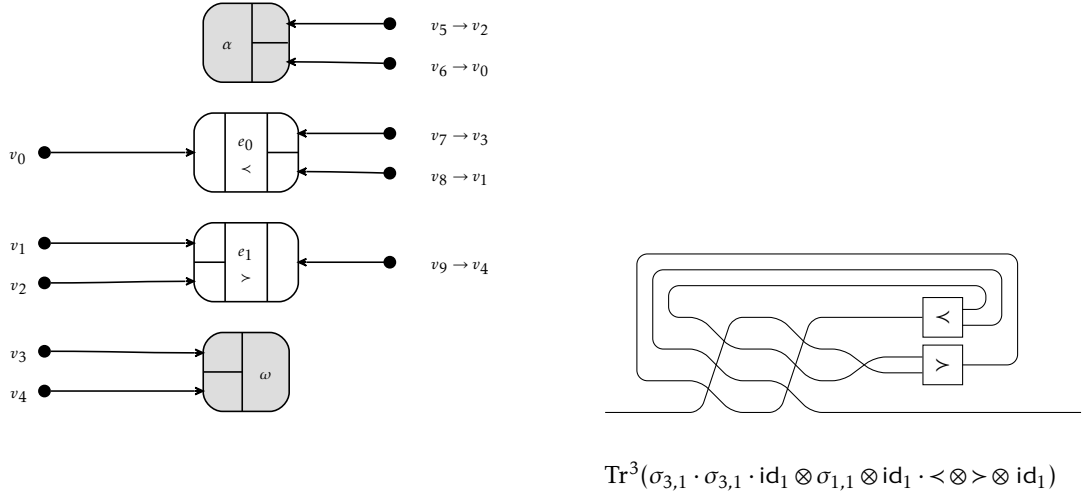}}
  \fontsize{12}{10}
  \qquad
  \raisebox{-0.1\height}{\begin{tikzpicture}[yscale=-1,x=1em,y=1em]

    \draw [rounded corners] (-2,3) -- (1,3) -- (2,0) -- (3.5,0) -- (4.5,1) -- (5.5,1) -- (6.5,1.75) -- (8,1.75);
    \draw [rounded corners] (9.25,-0.25) -- (10,-0.25) -- (10,-1) -- (0,-1) -- (0,0) -- (1,0) -- (2,1) -- (3.5,1) -- (4.5,2) -- (5.5,2) -- (6.5, 1.25) -- (8, 1.25);
    \draw [rounded corners] (9.25,0.25) -- (10.5,0.25) -- (10.5,-1.5) -- (-0.5,-1.5) -- (-0.5,1) -- (1,1) -- (2,2) -- (3.5,2) -- (4.5,3) -- (13,3); 
    \draw [rounded corners] (9.25,1.5) -- (11,1.5) -- (11, -2) -- (-1,-2) -- (-1,2) -- (1,2) -- (2,3) -- (3.5,3) -- (4.5,0) -- (8,0);

    \node [draw, minimum height = 1.25em, minimum width = 1.25em, anchor=west] at (8,0) {\scalebox{0.75}{$\fork$}};
    \node [draw, minimum height = 1.25em, minimum width = 1.25em, anchor=west] at (8,1.5) {\scalebox{0.75}{$\join$}};

    \node [anchor=center] at (5.1,5) {\scalebox{0.8}{Tr$^3(\swap{3}{1} \seq \swap{3}{1} \seq \id[1] \tensor \swap{1}{1} \tensor \id[1] \seq \fork \tensor \join \tensor \,\id[1])$}};

\end{tikzpicture}}
  \caption{A hypergraph and its corresponding categorical term.}
  \label{fig:completeness}
\end{figure}  

\noindent The first step is to fix a total order on the edges $e_1,\cdots, e_n$, including any identity edges.
We fix this order $\leq$ globally.
The $\stack$ operation creates a tensor of the corresponding generators in $\termtype$ for each edge in the hypergraph.
Identity edges are represented by identity morphisms $\id[1]$ in the stack.
\begin{gather*}
  \morph{\stack[-][\leq][\Sigma]}{\lhypterm}{\termtype} \\
  \stack[F][\leq][\Sigma] = \bigotimes_{e \in (E_F + E_F[\id], \leq)} \phi \text{ where }  \phi = \begin{cases} 
    \labels(e) &\text{if}\ e \in E_F \\
    \id[1] & \text{if}\ e \in E_F[\id]
  \end{cases}
\end{gather*}

\noindent Most of the outputs from our stack of generators will need to connect to the inputs of other generators in the stack, so we must trace them around.
Then the only remaining step is to then connect the traced wires to the corresponding inputs in the stack.
Here it will be useful to consider the all the target and source vertices as two totally ordered sets, respecting our new edge order $\leq$.
We write $\vs[\leq]$, $\vt[\leq]$ for the ordered set of all vertices which respects the original order on the vertices and the new order on the edges.
Since the interface is not contained within the order, we set input vertices to be the lowest elements of $\vt[\leq]$ and the output vertices to be the greatest elements of $\vs[\leq]$.
For example, if we have sets $\vs[e_1] = \{s_1,s_2\}$, $\vs[e_2] = \{s_3,s_4, s_5\}$ and $\vs[\interface] = \{s_6\}$, then if we define $\leq$ as $e_1 < e_2$, then $\vs[\leq] = \{s_1,s_2,s_3,s_4,s_5,s_6\}$. 
To simplify notation, we also introduce the notion of a \emph{connections permutation}.

\begin{definition}[Connections permutation]
    For an interfaced linear hypergraph $H$ equipped with edge order $\leq$, where $|\vs[\leq]| = x$, we call its connections permutation $\morph{p}{[x]}{[x]}$ the permutation such that every $i < x$, $\vconnsr_H(\proj{i}(\vt[\leq])) = \proj{p(i)}(\vs[\leq])$. 
\end{definition}

\begin{lemma}[Discrete composition]\label{lem:edgeless-composition}
    For any two discrete interfaced linear hypergraphs $\morph{F}{m}{n}$ and $\morph{G}{n}{k}$, with connections permutations $p$ and $q$ respectively, then the connections permutation of $F \seq G$ is $q \circ p$.
\end{lemma}
\begin{proof}
  By definition of connections permutations, for each target $t = \proj{i}(\vt[\interface][F])$, $\vconnsr_F(t) = \proj{p(i)}(\vs[\interface][F])$, and for each target $t = \proj{i}(\vt[\interface][G])$, $\vconnsr_G(t) = \proj{q(i)}(\vs[\interface][G])$.
  Since composition deletes the input vertices of $G$, we are only concerned with the input vertices of $F$.
  By the first connections permutation, in $F$ the $i$th input vertex originally connected to the $p(i)$th output vertex, so by definition of composition, in $F \seq G$ it will be connected to $\vconnsr_G(\proj{p(i)}(\vt[\interface][G]))$.
  By the second connections permutation this is equal to $\proj{q(p(i)}(\vs[\interface][G])$.
  Therefore the connections permutation of $F \seq G$ is $q \circ p$.
\end{proof}

\noindent We use this permutation to define a `shuffle' construct comprised of symmetries and identities, defined recursively over the set $\vs[\leq]$.
The target that connects to the lowest source is determined, and a symmetry pulling this wire up to the `top' is then defined: this wire is now in the correct position and is of no further concern to us. 
We recursively perform $\shuffle$ on the remaining source and target vertices until none remain, as demonstrated in Figure~\ref{fig:shuffle}.
Before proceeding, we show that the `input-output' connectivity of the shuffle construct reflects the connectivity of the original hypergraph.

\begin{figure}
  \centering
  \begin{tikzpicture}[yscale = -1, x = 0.5em, y = 1.25em]


    \draw[rounded corners] (1,1) -- (2,1);
    \draw[rounded corners] (1,2) -- (2,2);
    \draw[rounded corners] (1,3) -- (2,3);
    
    \node[minimum width=1em,minimum height=1em,anchor=east] at (1,1){$v_0$};
    \node[minimum width=1em,minimum height=1em,anchor=east] at (1,2){$v_1$};
    \node[minimum width=1em,minimum height=1em,anchor=east] at (1,3){$v_2$};

    \node[draw,minimum width=2em, minimum height = 4em, anchor=west] at (2,2){$\phi_3$};

    \draw[rounded corners] (6,1) -- (7,1);
    \draw[rounded corners] (6,2) -- (7,2);
    \draw[rounded corners] (6,3) -- (7,3);

    \node[minimum width=1em,minimum height=1em,anchor=west] at (7,1){$v_3$};
    \node[minimum width=1em,minimum height=1em,anchor=west] at (7,2){$v_4$};
    \node[minimum width=1em,minimum height=1em,anchor=west] at (7,3){$v_5$};

    \node at (12,2){$=$};


    \draw[rounded corners] (17,1) -- (18,1) -- (20,2) -- (21,2);
    \draw[rounded corners] (17,2) -- (18,2) -- (20,1) -- (26,1);
    \draw[rounded corners] (17,3) -- (18,3) -- (20,3) -- (21,3);

    \node[minimum width=1em,minimum height=1em,anchor=east] at (17,1){$v_0$};
    \node[draw = lightgray, minimum width=1em,minimum height=1em,anchor=east] at (17,2){$v_1$};
    \node[minimum width=1em,minimum height=1em,anchor=east] at (17,3){$v_2$};

    \node[draw,minimum width=2em, minimum height = 2.5em, anchor=west] at (21,2.5){$\phi_2$};

    \draw[rounded corners] (25,2) -- (26,2);
    \draw[rounded corners] (25,3) -- (26,3);

    \node[draw = lightgray, minimum width=1em,minimum height=1em,anchor=west] at (26,1){$v_3$};
    \node[minimum width=1em,minimum height=1em,anchor=west] at (26,2){$v_4$};
    \node[minimum width=1em,minimum height=1em,anchor=west] at (26,3){$v_5$};

    \node[anchor=west] at (15.5,-0.5) {$\vconnsl(v_3) = v_1$};


    \node at (31,2){$=$};

    \node[minimum width=1em,minimum height=1em,anchor=east] at (36,1){$v_0$};
    \node[text = gray, minimum width=1em,minimum height=1em,anchor=east] at (36,2){$v_1$};
    \node[draw = lightgray, minimum width=1em,minimum height=1em,anchor=east] at (36,3){$v_2$};

    \draw[rounded corners] (36,1) -- (37,1) -- (39,2) -- (40,2);
    \draw[rounded corners] (40,2) -- (41,2) -- (43,3) -- (44,3);
    \draw[rounded corners] (36,2) -- (37,2) -- (39,1) -- (49,1);
    \draw[rounded corners] (36,3) -- (41,3) -- (43,2) -- (49,2);

    \node[draw,minimum width=2em, minimum height = 1em, anchor=west] at (44,3){$\phi_1$};

    \draw[rounded corners] (48,3) -- (49,3);

    \node[text = gray, minimum width=1em,minimum height=1em,anchor=west] at (49,1){$v_3$};
    \node[draw = lightgray, minimum width=1em,minimum height=1em,anchor=west] at (49,2){$v_4$};
    \node[minimum width=1em,minimum height=1em,anchor=west] at (49,3){$v_5$};

    \node at (54,2){$=$};

    \node[anchor=west] at (36.5,-0.5) {$\vconnsl(v_4) = v_2$}; 


    \node[draw=lightgray, minimum width=1em,minimum height=1em,anchor=east] at (59,1){$v_0$};
    \node[text=gray,minimum width=1em,minimum height=1em,anchor=east] at (59,2){$v_1$};
    \node[text=gray,minimum width=1em,minimum height=1em,anchor=east] at (59,3){$v_2$};

    \draw[rounded corners] (59,1) -- (60,1) -- (62,2) -- (64,2) -- (66,3) -- (67,3);
    \draw[rounded corners] (59,2) -- (60,2) -- (62,1) -- (67,1);
    \draw[rounded corners] (59,3) -- (64,3) -- (66,2) -- (67,2);

    \node[text=gray,minimum width=1em,minimum height=1em,anchor=west] at (67,1){$v_3$};
    \node[text=gray,minimum width=1em,minimum height=1em,anchor=west] at (67,2){$v_4$};
    \node[draw = lightgray, minimum width=1em,minimum height=1em,anchor=west] at (67,3){$v_5$};

    \node[anchor=west] at (57, -0.5){$\vconnsl(v_5) = v_0$};
\end{tikzpicture}
  \caption{Performing the $\shuffle$ algorithm, where $\phi_n$ denotes the shuffle construct for $n$ vertices.}
  \label{fig:shuffle}    
\end{figure}
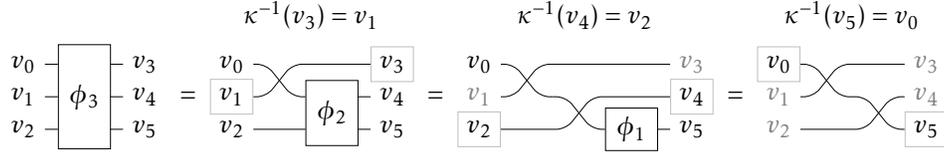

\begin{algorithm}[tb]
  \SetKwProg{Fn}{Function}{}{end}
  \SetKwData{Hyper}{$\lilhyper$}
  \SetKwData{Targets}{$T$}
  \SetKwData{Sources}{$S$}
  \SetKwData{Term}{$\text{f}$}
  \SetKwData{s}{$\mathsf{s}$} 
  \SetKwData{t}{$\mathsf{t}$} 
  \SetKwData{i}{$\mathsf{i}$} 
  \SetKwData{j}{$\mathsf{j}$} 
  \SetKwData{k}{$\mathsf{k}$}
  \SetKwFunction{Conns}{$\vconnsr$}
  \SetKwData{VT}{$\vt$}
  \SetKwData{VS}{$\vs$}
  \SetKwData{inputs}{$\mathsf{x}$}
  \SetKwData{vertices}{$\mathsf{v}$}
  \SetKwFunction{Shuffle}{$\shuffle_{\Sigma,\leq}$}
  \SetKwFunction{Shuffley}{$\shuffle'_{\Sigma,\leq}$}
  \SetAlgoLined
  \Fn{\Shuffley{\Sources,\Targets,\Conns}}{
      \If{$|\Sources| = 0$}{
          \KwRet $\id[0]$\;
      }
      $\s \leftarrow \proj{0}(\Sources)$; \, $\t \leftarrow \vconnsr^{-1}(\s)$\;
      $\i \leftarrow i \text{ where } \proj{i}(\Targets) = \t$; \, $\j \leftarrow |\Targets| - \i - 1$\;
      \Term $\leftarrow$ $\swap{\i}{1} \tensor \j$\; 
      \KwRet $\Term \seq \id[1] \tensor \Shuffle{\Sources - \s, \, \Targets - \t,\,\Conns}$\;
  }
  \Fn{\Shuffle{$F = \lilhyper$}}{
    \Shuffley{$\vs[\leq], \vt[\leq], \vconnsr$}\;
  }
  \label{alg:shuffle}
  \caption{Defining the shuffle construct.}
\end{algorithm}

\begin{lemma}[Correctness of shuffle]\label{lem:correctness-shuffle}
  For any interfaced linear hypergraph $F$ with connections permutation $p$ and some $\morph{\bar{v}}{0}{n} \in \termtype = v_0 \tensor v_1 \tensor \cdots \tensor v_{n-1}$, $\bar{v} \seq \shuffle[F][\leq][\Sigma] = v_{p^{-1}(0)} \tensor v_{p^{-1}(1)} \tensor \cdots \tensor v_{p^{-1}(n-1)}$.
\end{lemma}
\begin{proof}
  This is by induction on $n$. For $n < 2$ the statement is trivially correct. 
  For $n = 2$, there are two cases: $\{0 \mapsto 0, 1 \mapsto 1\}$ and $\{0 \mapsto 1, 1 \mapsto 0\}$. 
  By naturality of symmetry, $\bar{v} \seq \sigma_2$ is equal to $v_0 \tensor (v_1 \seq \phi_1) = v_0 \tensor v_1$ in the former and $v_1 \tensor (v_0 \seq \sigma_1) = v_1 \tensor v_0$, so for both cases the statement holds. 
  For $n > 2$, $\bar{v} \seq \phi_{n} = v_{x} \tensor (v_{0} \tensor \cdots. \tensor v_{x-1} \tensor v_{x+1} \tensor \cdots \tensor v_{n-1} \seq \sigma_{n-1})$ where $x = p^{-1}(0)$ by naturality of symmetry.
  Therefore the first element of the tensor is correct, and the remaining elements follow by inductive hypothesis.
\end{proof}

\noindent Since there is no input `box', we also need to precompose the shuffle construct with another symmetry to pull the input wires to the `top' of the term.
To retrieve a term from a hypergraph $H$ with edge order $\leq$, we simply trace the composition of the corresponding shuffle construct and edge stack.

\begin{definition}[Definability functor]\label{lem:definabilty-functor}
  We define the definability functor as the identity-on-objects traced monoidal functor $\morph{\term[-][\leq][\Sigma]}{\lhypterm}{\textbf{Term}_{\hypsig}}$ with its action defined for a given edge order $\leq$ and interfaced linear hypergraph $\morph{F}{m}{n}$ with $|\vs[\leq]|$ as
  \begin{center}
      $\term[F][\leq][\Sigma] = \trace{x-m}{\swap{x-m}{m} \seq \shuffle[F][\leq][\Sigma] \seq \stack[F][\leq][\Sigma] \tensor \id[n]}$
  \end{center}
\end{definition}

\noindent To conclude definability we must be able to return to the original hypergraph. 
The shuffle construct is our main obstacle to showing this, so we tackle it separately.

\begin{lemma}[Definability of shuffle]\label{lem:definability-shuffle}
  For any shuffle construct $\morph{\phi_n}{n}{n}$, interpretation $F = \hypfun{\phi_n}$, and permutation $\morph{p}{\set{n}}{\set{n}}$ such that for some $\morph{\bar{v}}{0}{n} = v_0 \tensor v_1 \tensor \cdots \tensor v_{n-1}$, $\bar{v} \seq \phi = v_{p(0)} \tensor v_{p(1)} \tensor \cdots \tensor v_{p(n-1)}$, then $\vconnsr(\proj{i}(\vt[\interface][F])) = \proj{p(i)}(\vs[\interface][F])$.
\end{lemma}
\begin{proof}
  We first use staging and composite symmetry to arrange the shuffle construct into `slices' containing exactly one symmetry $\swap{1}{1}$. We then perform induction of $n$.
  For $n < 2$ the statement holds trivially.
  For $n = 2$ we examine the two cases $\{0 \mapsto 0, 1 \mapsto 1\}$ and $\{0 \mapsto 1, 1 \mapsto 0\}$ as in correctness of shuffle.
  They correspond to the interfaced linear hypergraphs $\id[2]$ and $\swap{1}{1}$ respectively.
  For both cases the statement holds.
  For $n > 2$, we split the definition of $\sigma_n$ into two parts: $\sigma_n' = \swap{x}{1} \tensor \id[n - 1 - x])$ where $x = p^{-1}(0)$, and $\sigma_n'' = \id[1] \tensor \sigma_{n-1}$. 
  In $\hypfun{\sigma'_n}$, $\vconnsr(\proj{x}(\vt[\interface])) = \proj{0}(\vs[\interface])$ by definition of the swap hypergraph. 
  In $\hypfun{\sigma''_n}$, $\vconnsr(\proj{0}(\vt[\interface])) = \proj{0}(\vs[\interface])$ by definition of identity and monoidal tensor. 
  Therefore in $\hypfun{\sigma_n}$, $\vconnsr(\proj{i}(\vt)) = \proj{0}(\vs)$ by discrete composition. 
  So for the first vertex in $\vt$ the statement holds. 
  For the remaining vertices we apply the inductive hypothesis to $\phi_{n-1}$ and add one to the indices of each vertex, since we tensor the construct with an identity wire.
\end{proof}

\noindent Finally we can take on the entire term.

\begin{proposition}[Definability]\label{prop:definability}
  For any interfaced linear hypergraph $\morph{F}{m}{n}$ equipped with edge order $\leq$, where $G = \hypfun{\term[F][\leq][\hypsig]}$, then $F \equiv G$.
\end{proposition}
\begin{proof}
  We map the sources and targets of edges in $F$ to the corresponding vertices in $\hypfun{\term[F][\leq][\hypsig]}$.
  We do the same for the edges. 
  If there is an $\id[1]$ in the stack of generators, we perform an expansion to introduce an identity edge.
  The labelling condition is immediate, so we only need to consider the connections condition. 
  Using the edge order $\leq$, we consider the connections permutation of $F$ and $G$, which we name $p$ and $q$ respectively.
  If $F \equiv G$, then the two permutations must be the same.
  
  We examine the permutation $q$.
  There are two classes of vertices to consider: those that are inputs and those that are not.
  In the first case, the $i$th input (corresponding to $\proj{i}(\vt[\leq])$) will connect to the $p(i)$th source vertex by definability of shuffle, so $q = p$.
  In the second case, the $i$th target (which corresponds to $\proj{i + m}(\vt[\leq])$), will connect to the $p(i + m)$th source vertex by definition of trace, discrete composition and definability of shuffle.
  So $q = p$ in this case also.
  Therefore the connection permutations are equal, so $F \equiv G$.
\end{proof}

\subsection{Coherence}

We cannot immediately conclude completeness. 
There are multiple orders we can choose when stacking the edges, resulting in multiple different shuffle constructs and recovered terms. 
For \emph{coherence} these recovered terms must all be equal by the equations of the STMC.

\begin{definition}[Coherence]\label{def:coherence}
  Interfaced linear hypergraphs are coherent if, for any well-formed interfaced linear hypergraph $F$ and any two orders on its edges $\leq_1$, $\leq_2$, $\term[F][\leq_1][\hypsig] = \term[F][\leq_2][\hypsig]$ by the equations of the STMC.
\end{definition}

\noindent Fortunately, we only need to consider switching two consecutive edges while retaining the others, e.g. $x < f < g < y$ becomes $x < g < f < y$. We can then swap any two edges in the set by propagating these swaps throughout the entire set. To enable us to do this, we can combine a tensor of edge boxes into one single box:

\begin{lemma}[Combination]\label{lem:combination}
  For any $t = \bigtensor_{i = 0}^{m} a_i \tensor \bigtensor_{i=0}^{n} b_i \tensor \bigtensor_{i = 0}^{p}{c_i}$, we can rewrite it as $a~\tensor~\bigtensor_{i=0}^{n}{b_i}~\tensor~c$ where $\morph{a}{\Sigma_{i=0}^{m}\,\dom{a_i}}{\Sigma_{i=0}^{m}\,\cod{a_i}}$ and $\morph{c}{\hypsig_{i=0}^{p}\,\dom{c_i}}{\hypsig_{i=0}^{p}\,\cod{c_i}}$.
\end{lemma}

\noindent We need a lemma to show that we can transform the shuffle construct for a given order into one for a different order by adding appropriate symmetries on either side, reflecting that the edge boxes have now swapped over.

\begin{lemma}[Coherence of shuffle]\label{lem:coherence-shuffle}
  For any interfaced linear hypergraph $\morph{F}{m}{n}$, two orders $\leq_1,\leq_2$ on its edges in which only two consecutive elements $\morph{e_1}{n'}{n}$ and $\morph{e_2}{p'}{p}$ have been swapped, and shuffle constructs \[\morph{\shuffle[F][\leq_1][\Sigma]}{m + p + q + r + s}{p' + q' + r' + s' + n}\] \[\morph{\shuffle[F][\leq_2][\Sigma]}{m + p + r + q + s}{p' + r' + q' + s' + n}\]
  \noindent then
  \[\shuffle[F][\leq_1][\hypsig] = \id_m \tensor \id_p \tensor \swap{q}{r} \tensor \id_s \seq \shuffle[F][\leq_2][\hypsig] \seq \id_{p'} \tensor \swap{r'}{q'} \tensor \id_{s'} \tensor \id[n]\]

  \begin{center}
      \begin{tikzpicture}[yscale=-1,x=1em,y=1.25em]
    \draw [rounded corners] (0,0) -- (1,0);
    \draw [rounded corners] (0,1) -- (1,1);
    \draw [rounded corners] (0,2) -- (1,2);
    \draw [rounded corners] (0,3) -- (1,3);
    \draw [rounded corners] (0,4) -- (1,4);

    \node[draw, minimum height = 6em, minimum width = 5.5em, anchor = west] at (1, 2){$\shuffle_{\leq_1}$};

    \draw [rounded corners] (6.5,0) -- (7.5,0);
    \draw [rounded corners] (6.5,1) -- (7.5,1);
    \draw [rounded corners] (6.5,2) -- (7.5,2);
    \draw [rounded corners] (6.5,3) -- (7.5,3);
    \draw [rounded corners] (6.5,4) -- (7.5,4);

    \node at (9,2){$=$};

    \draw [rounded corners] (11,0) -- (14,0);
    \draw [rounded corners] (11,1) -- (14,1);
    \draw [rounded corners] (11,2) -- (12,2) -- (13,3) -- (14,3);
    \draw [rounded corners] (11,3) -- (12,3) -- (13,2) -- (14,2);
    \draw [rounded corners] (11,4) -- (14,4);

    \node[draw, minimum height = 6em, minimum width = 5.5em, anchor = west] at (14, 2){$\shuffle_{\leq_2}$};

    \draw [rounded corners] (19.5,0) -- (22.5,0);
    \draw [rounded corners] (19.5,1) -- (20.5,1) -- (21.5,2) -- (22.5,2);
    \draw [rounded corners] (19.5,2) -- (20.5,2) -- (21.5,1) -- (22.5,1);
    \draw [rounded corners] (19.5,3) -- (20.5,3) -- (21.5,3) -- (22.5,3);
    \draw [rounded corners] (19.5,4) -- (20.5,4) -- (21.5,4) -- (22.5,4);

\end{tikzpicture}
  \end{center}

\end{lemma}
\begin{proof}

  \noindent We show equality by asserting for every `input' to the term, it always leads to the same `output'. 
  In the diagram below, each box represents a `bundle' of wires. 
  By correctness of shuffle (Lemma \ref{lem:correctness-shuffle}), we know that the each shuffle construct respects the connections permutation of the $H$ with order $\leq_1$ and $\leq_2$ respectively. 
  Effectively, this means that it will shuffle the vertices into new bundles for each generator in the stack, only differing by the swapped edges (labelled $r'$ and $q'$ on the diagram below).
  We can therefore use naturality of symmetry to to show both expressions are equal.

  \begin{center}

      \textbf{LHS}

      
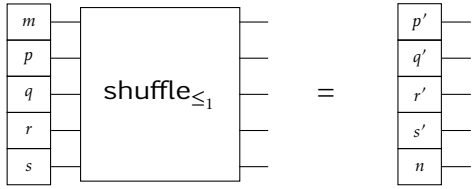
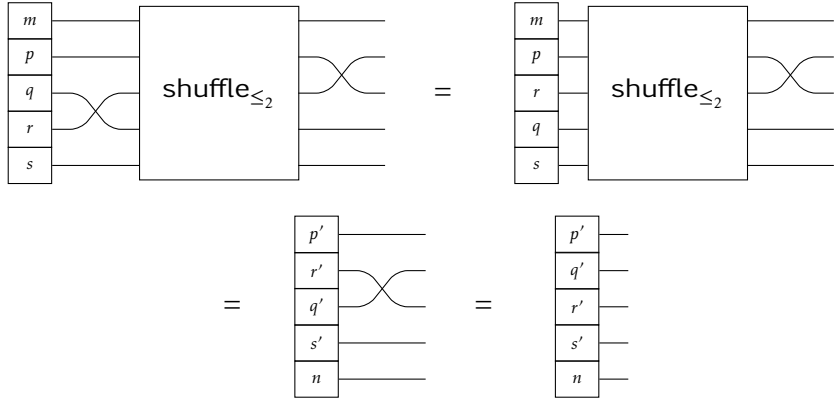
\vspace{1em}

      \begin{tikzpicture}[yscale=-1,x=1em,y=1.25em]
        
    \node[draw, minimum height = 1.25em, minimum width = 1.5em, anchor = east] at (0,0){\scalebox{0.6}{$m$}};
    \node[draw, minimum height = 1.25em, minimum width = 1.5em, anchor = east] at (0,1){\scalebox{0.6}{$p$}};
    \node[draw, minimum height = 1.25em, minimum width = 1.5em, anchor = east] at (0,2){\scalebox{0.6}{$q$}};
    \node[draw, minimum height = 1.25em, minimum width = 1.5em, anchor = east] at (0,3){\scalebox{0.6}{$r$}};
    \node[draw, minimum height = 1.25em, minimum width = 1.5em, anchor = east] at (0,4){\scalebox{0.6}{$s$}};

    \draw [rounded corners] (0,0) -- (1,0);
    \draw [rounded corners] (0,1) -- (1,1);
    \draw [rounded corners] (0,2) -- (1,2);
    \draw [rounded corners] (0,3) -- (1,3);
    \draw [rounded corners] (0,4) -- (1,4);

    \node[draw, minimum height = 6em, minimum width = 5.5em, anchor = west] at (1, 2){$\shuffle_{\leq_1}$};

    \draw [rounded corners] (6.5,0) -- (7.5,0);
    \draw [rounded corners] (6.5,1) -- (7.5,1);
    \draw [rounded corners] (6.5,2) -- (7.5,2);
    \draw [rounded corners] (6.5,3) -- (7.5,3);
    \draw [rounded corners] (6.5,4) -- (7.5,4);

    \node at (9.5,2){$=$};

    \node[draw, minimum height = 1.25em, minimum width = 1.5em, anchor = east] at (13.5,0){\scalebox{0.6}{$p'$}};
    \node[draw, minimum height = 1.25em, minimum width = 1.5em, anchor = east] at (13.5,1){\scalebox{0.6}{$q'$}};
    \node[draw, minimum height = 1.25em, minimum width = 1.5em, anchor = east] at (13.5,2){\scalebox{0.6}{$r'$}};
    \node[draw, minimum height = 1.25em, minimum width = 1.5em, anchor = east] at (13.5,3){\scalebox{0.6}{$s'$}};
    \node[draw, minimum height = 1.25em, minimum width = 1.5em, anchor = east] at (13.5,4){\scalebox{0.6}{$n$}};

    \draw [rounded corners] (13.5,0) -- (14.5,0);
    \draw [rounded corners] (13.5,1) -- (14.5,1);
    \draw [rounded corners] (13.5,2) -- (14.5,2);
    \draw [rounded corners] (13.5,3) -- (14.5,3);
    \draw [rounded corners] (13.5,4) -- (14.5,4);

\end{tikzpicture}

      \vspace{1em}

      \textbf{RHS}
      
      \vspace{1em}

      \begin{tikzpicture}[yscale=-1,x=1em,y=1.25em]

    \node[draw, minimum height = 1.25em, minimum width = 1.5em, anchor = east] at (0,0){\scalebox{0.6}{$m$}};
    \node[draw, minimum height = 1.25em, minimum width = 1.5em, anchor = east] at (0,1){\scalebox{0.6}{$p$}};
    \node[draw, minimum height = 1.25em, minimum width = 1.5em, anchor = east] at (0,2){\scalebox{0.6}{$q$}};
    \node[draw, minimum height = 1.25em, minimum width = 1.5em, anchor = east] at (0,3){\scalebox{0.6}{$r$}};
    \node[draw, minimum height = 1.25em, minimum width = 1.5em, anchor = east] at (0,4){\scalebox{0.6}{$s$}};

    \draw [rounded corners] (0,0) -- (3,0);
    \draw [rounded corners] (0,1) -- (3,1);
    \draw [rounded corners] (0,2) -- (1,2) -- (2,3) -- (3,3);
    \draw [rounded corners] (0,3) -- (1,3) -- (2,2) -- (3,2);
    \draw [rounded corners] (0,4) -- (3,4);

    \node[draw, minimum height = 6em, minimum width = 5.5em, anchor = west] at (3, 2){$\shuffle_{\leq_2}$};

    \draw [rounded corners] (8.5,0) -- (11.5,0);
    \draw [rounded corners] (8.5,1) -- (9.5,1) -- (10.5,2) -- (11.5,2);
    \draw [rounded corners] (8.5,2) -- (9.5,2) -- (10.5,1) -- (11.5,1);
    \draw [rounded corners] (8.5,3) -- (11.5,3);
    \draw [rounded corners] (8.5,4) -- (11.5,4);

    \node at (13.5,2){$=$};

    \node[draw, minimum height = 1.25em, minimum width = 1.5em, anchor = east] at (17.5,0){\scalebox{0.6}{$m$}};
    \node[draw, minimum height = 1.25em, minimum width = 1.5em, anchor = east] at (17.5,1){\scalebox{0.6}{$p$}};
    \node[draw, minimum height = 1.25em, minimum width = 1.5em, anchor = east] at (17.5,2){\scalebox{0.6}{$r$}};
    \node[draw, minimum height = 1.25em, minimum width = 1.5em, anchor = east] at (17.5,3){\scalebox{0.6}{$q$}};
    \node[draw, minimum height = 1.25em, minimum width = 1.5em, anchor = east] at (17.5,4){\scalebox{0.6}{$s$}};

    \draw [rounded corners] (17.5,0) -- (18.5,0);
    \draw [rounded corners] (17.5,1) -- (18.5,1);
    \draw [rounded corners] (17.5,2) -- (18.5,2);
    \draw [rounded corners] (17.5,3) -- (18.5,3);
    \draw [rounded corners] (17.5,4) -- (18.5,4);

    \node[draw, minimum height = 6em, minimum width = 5.5em, anchor = west] at (18.5, 2){$\shuffle_{\leq_2}$};

    \draw [rounded corners] (24,0) -- (27,0);
    \draw [rounded corners] (24,1) -- (25,1) -- (26,2) -- (27,2);
    \draw [rounded corners] (24,2) -- (25,2) -- (26,1) -- (27,1);
    \draw [rounded corners] (24,3) -- (27,3);
    \draw [rounded corners] (24,4) -- (27,4);

\end{tikzpicture}

      \vspace{1em}

      \begin{tikzpicture}[yscale=-1,x=1em,y=1.25em]

    \node[anchor = east] at (0,9){$=$};

    \node[draw, minimum height = 1.25em, minimum width = 1.5em, anchor = east] at (3,7){\scalebox{0.6}{$p'$}};
    \node[draw, minimum height = 1.25em, minimum width = 1.5em, anchor = east] at (3,8){\scalebox{0.6}{$r'$}};
    \node[draw, minimum height = 1.25em, minimum width = 1.5em, anchor = east] at (3,9){\scalebox{0.6}{$q'$}};
    \node[draw, minimum height = 1.25em, minimum width = 1.5em, anchor = east] at (3,10){\scalebox{0.6}{$s'$}};
    \node[draw, minimum height = 1.25em, minimum width = 1.5em, anchor = east] at (3,11){\scalebox{0.6}{$n$}};

    \draw [rounded corners] (3,7) -- (6,7);
    \draw [rounded corners] (3,8) -- (4,8) -- (5,9) -- (6,9);
    \draw [rounded corners] (3,9) -- (4,9) -- (5,8) -- (6,8);
    \draw [rounded corners] (3,10) -- (6,10);
    \draw [rounded corners] (3,11) -- (6,11);

    \node at (8,9) {$=$};

    \node[draw, minimum height = 1.25em, minimum width = 1.5em, anchor = east] at (12,7){\scalebox{0.6}{$p'$}};
    \node[draw, minimum height = 1.25em, minimum width = 1.5em, anchor = east] at (12,8){\scalebox{0.6}{$q'$}};
    \node[draw, minimum height = 1.25em, minimum width = 1.5em, anchor = east] at (12,9){\scalebox{0.6}{$r'$}};
    \node[draw, minimum height = 1.25em, minimum width = 1.5em, anchor = east] at (12,10){\scalebox{0.6}{$s'$}};
    \node[draw, minimum height = 1.25em, minimum width = 1.5em, anchor = east] at (12,11){\scalebox{0.6}{$n$}};

    \draw [rounded corners] (12,7) -- (13,7);
    \draw [rounded corners] (12,8) -- (13,8);
    \draw [rounded corners] (12,9) -- (13,9);
    \draw [rounded corners] (12,10) -- (13,10);
    \draw [rounded corners] (12,11) -- (13,11);

\end{tikzpicture}
      
  \end{center}

\end{proof}

\begin{lemma}[Coherence for two edges]\label{lem:completeness-two}
  For any interfaced linear hypergraph $\morph{F}{m}{n}$ and two orders $\leq_1,\leq_2$ on its edges which differ only by the swapping of two consecutive elements, $\term[H][\leq_1][\Sigma] = \term[H][\leq_2][\Sigma]$.
\end{lemma}
\begin{proof}

  \noindent Any term generated by $\term[F][\leq][\hypsig]$ can be rewritten in the form \[\trace{|\vt|-m}{\swap{|\vt|-m}{m} \seq \shuffle[F][\leq][\hypsig] \seq x \tensor f \tensor g \tensor y \tensor \id[n]}\] by using combination (Lemma \ref{lem:combination}). 
  So we have a term of the form:
  
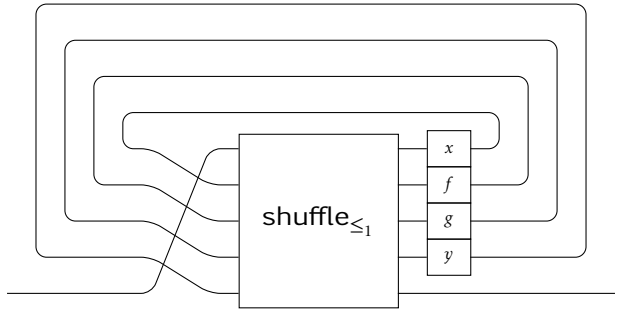
\begin{center}
      \begin{tikzpicture}[yscale=-1,x=1em,y=1.25em]

    \draw[rounded corners] (16, 9) -- (20, 9) -- (20, 2) -- (1, 2) -- (1, 9) -- (5, 9) -- (7, 10) -- (8, 10);
    \draw[rounded corners] (16, 8) -- (19, 8) -- (19, 3) -- (2, 3) -- (2, 8) -- (5, 8) -- (7, 9) -- (8, 9);
    \draw[rounded corners] (16, 7) -- (18, 7) -- (18, 4) -- (3, 4) -- (3, 7) -- (5, 7) -- (7, 8) -- (8, 8);
    \draw[rounded corners] (16, 6) -- (17, 6) -- (17, 5) -- (4, 5) -- (4, 6) -- (5, 6) -- (7, 7) -- (8, 7);

    \draw[rounded corners] (0, 10) -- (5, 10) -- (7,6) -- (8,6);

    \node[draw, minimum height = 6em, minimum width = 5.5em, anchor = west] at (8, 8){$\shuffle_{\leq_1}$};

    \draw[rounded corners] (13.5, 6) -- (14.5, 6);
    \draw[rounded corners] (13.5, 7) -- (14.5, 7);
    \draw[rounded corners] (13.5, 8) -- (14.5, 8);
    \draw[rounded corners] (13.5, 9) -- (14.5, 9);
    \draw[rounded corners] (13.5, 10) -- (21, 10);

    \node[draw, minimum height = 1.25em, minimum width = 1.5em, anchor = west] at (14.5,6){\scalebox{0.6}{$x$}};
    \node[draw, minimum height = 1.25em, minimum width = 1.5em, anchor = west] at (14.5,7){\scalebox{0.6}{$f$}};
    \node[draw, minimum height = 1.25em, minimum width = 1.5em, anchor = west] at (14.5,8){\scalebox{0.6}{$g$}};
    \node[draw, minimum height = 1.25em, minimum width = 1.5em, anchor = west] at (14.5,9){\scalebox{0.6}{$y$}};

\end{tikzpicture}
  \end{center}
  \noindent By exchange:
  
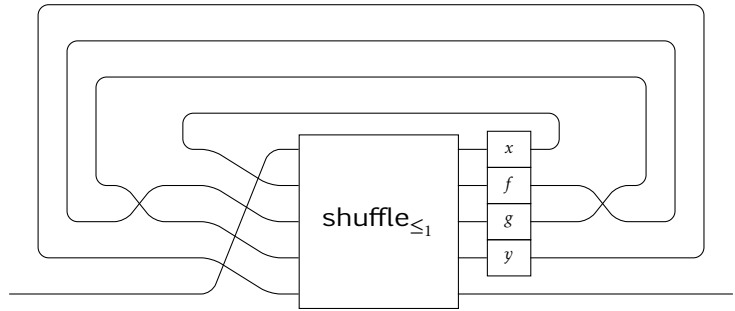
\begin{center}
      \begin{tikzpicture}[yscale=-1,x=1em,y=1.25em]

    \draw[rounded corners] (16, 9) -- (22, 9) -- (22, 2) -- (-1, 2) -- (-1, 9) -- (5, 9) -- (7, 10) -- (8, 10);
    \draw[rounded corners] (16, 8) -- (18, 8) -- (19, 7) -- (20, 7) -- (20, 4) -- (1, 4) -- (1, 7) -- (2, 7) -- (3, 8) -- (5, 8) -- (7, 9) -- (8, 9);
    \draw[rounded corners] (16, 7) -- (18, 7) -- (19, 8) -- (21, 8) -- (21, 3) -- (0, 3) -- (0, 8) -- (2, 8) -- (3, 7) -- (5, 7) -- (7, 8) -- (8, 8);
    \draw[rounded corners] (16, 6) -- (17, 6) -- (17, 5) -- (4, 5) -- (4, 6) -- (5, 6) -- (7, 7) -- (8, 7);

    \draw[rounded corners] (-2, 10) -- (5, 10) -- (7,6) -- (8,6);

    \node[draw, minimum height = 6em, minimum width = 5.5em, anchor = west] at (8, 8){$\shuffle_{\leq_1}$};

    \draw[rounded corners] (13.5, 6) -- (14.5, 6);
    \draw[rounded corners] (13.5, 7) -- (14.5, 7);
    \draw[rounded corners] (13.5, 8) -- (14.5, 8);
    \draw[rounded corners] (13.5, 9) -- (14.5, 9);
    \draw[rounded corners] (13.5, 10) -- (23, 10);

    \node[draw, minimum height = 1.25em, minimum width = 1.5em, anchor = west] at (14.5,6){\scalebox{0.6}{$x$}};
    \node[draw, minimum height = 1.25em, minimum width = 1.5em, anchor = west] at (14.5,7){\scalebox{0.6}{$f$}};
    \node[draw, minimum height = 1.25em, minimum width = 1.5em, anchor = west] at (14.5,8){\scalebox{0.6}{$g$}};
    \node[draw, minimum height = 1.25em, minimum width = 1.5em, anchor = west] at (14.5,9){\scalebox{0.6}{$y$}};

\end{tikzpicture}
  \end{center}
  \noindent By naturality of symmetry and functoriality:
  
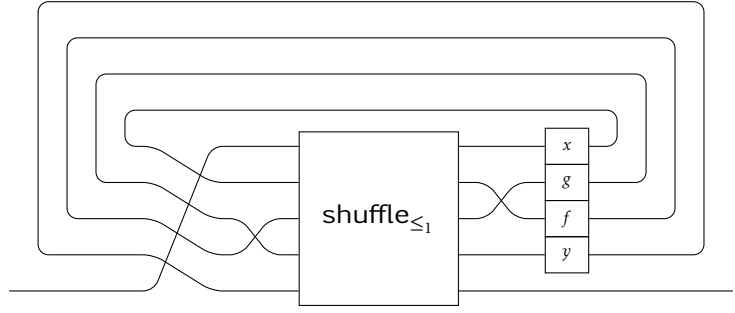
\begin{center}
      \begin{tikzpicture}[yscale=-1,x=1em,y=1.25em]
        
    \draw[rounded corners] (18, 9) -- (22, 9) -- (22, 2) -- (-1, 2) -- (-1, 9) -- (3, 9) -- (5, 10) -- (8, 10);
    \draw[rounded corners] (18, 8) -- (21, 8) -- (21, 3) -- (0, 3) -- (0, 8) -- (3, 8) -- (5, 9) -- (6, 9) -- (7, 8) -- (8, 8);
    \draw[rounded corners] (18, 7) -- (20, 7) -- (20, 4) -- (1, 4) -- (1, 7) -- (3, 7) -- (5, 8) -- (6, 8) -- (7, 9) -- (8, 9);
    \draw[rounded corners] (18, 6) -- (19, 6) -- (19, 5) -- (2, 5) -- (2, 6) -- (3, 6) -- (5, 7) -- (8, 7);

    \draw[rounded corners] (-2, 10) -- (3, 10) -- (5,6) -- (8,6);

    \node[draw, minimum height = 6em, minimum width = 5.5em, anchor = west] at (8, 8){$\shuffle_{\leq_1}$};

    \draw[rounded corners] (13.5, 6) -- (16.5, 6);
    \draw[rounded corners] (13.5, 7) -- (14.5, 7) -- (15.5, 8) -- (16.5, 8);
    \draw[rounded corners] (13.5, 8) -- (14.5, 8) -- (15.5, 7) -- (16.5, 7);
    \draw[rounded corners] (13.5, 9) -- (16.5, 9);
    \draw[rounded corners] (13.5, 10) -- (23, 10);

    \node[draw, minimum height = 1.25em, minimum width = 1.5em, anchor = west] at (16.5,6){\scalebox{0.6}{$x$}};
    \node[draw, minimum height = 1.25em, minimum width = 1.5em, anchor = west] at (16.5,7){\scalebox{0.6}{$g$}};
    \node[draw, minimum height = 1.25em, minimum width = 1.5em, anchor = west] at (16.5,8){\scalebox{0.6}{$f$}};
    \node[draw, minimum height = 1.25em, minimum width = 1.5em, anchor = west] at (16.5,9){\scalebox{0.6}{$y$}};

\end{tikzpicture}
  \end{center}
  \noindent By coherence of shuffle (Lemma \ref{lem:coherence-shuffle}):
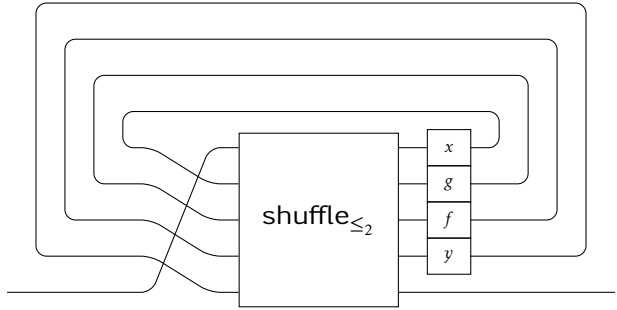
  \begin{center}
      \begin{tikzpicture}[yscale=-1,x=1em,y=1.25em]
        
    \draw[rounded corners] (16, 9) -- (20, 9) -- (20, 2) -- (1, 2) -- (1, 9) -- (5, 9) -- (7, 10) -- (8, 10);
    \draw[rounded corners] (16, 8) -- (19, 8) -- (19, 3) -- (2, 3) -- (2, 8) -- (5, 8) -- (7, 9) -- (8, 9);
    \draw[rounded corners] (16, 7) -- (18, 7) -- (18, 4) -- (3, 4) -- (3, 7) -- (5, 7) -- (7, 8) -- (8, 8);
    \draw[rounded corners] (16, 6) -- (17, 6) -- (17, 5) -- (4, 5) -- (4, 6) -- (5, 6) -- (7, 7) -- (8, 7);

    \draw[rounded corners] (0, 10) -- (5, 10) -- (7,6) -- (8,6);

    \node[draw, minimum height = 6em, minimum width = 5.5em, anchor = west] at (8, 8){$\shuffle_{\leq_2}$};

    \draw[rounded corners] (13.5, 6) -- (14.5, 6);
    \draw[rounded corners] (13.5, 7) -- (14.5, 7);
    \draw[rounded corners] (13.5, 8) -- (14.5, 8);
    \draw[rounded corners] (13.5, 9) -- (14.5, 9);
    \draw[rounded corners] (13.5, 10) -- (21, 10);

    \node[draw, minimum height = 1.25em, minimum width = 1.5em, anchor = west] at (14.5,6){\scalebox{0.6}{$x$}};
    \node[draw, minimum height = 1.25em, minimum width = 1.5em, anchor = west] at (14.5,7){\scalebox{0.6}{$g$}};
    \node[draw, minimum height = 1.25em, minimum width = 1.5em, anchor = west] at (14.5,8){\scalebox{0.6}{$f$}};
    \node[draw, minimum height = 1.25em, minimum width = 1.5em, anchor = west] at (14.5,9){\scalebox{0.6}{$y$}};

\end{tikzpicture}
  \end{center}
\end{proof}

\noindent We can then extend this lemma to obtain our final coherence result.

\begin{proposition}[Coherence]\label{prop:coherence}
    For all orderings of edges $\leq_x$ on an interfaced linear hypergraph $F$, \[\term[F][\leq_1][\hypsig] = \term[F][\leq_2][\Sigma] = \cdots = \term[F][\leq_x][\Sigma]\]
\end{proposition}
\begin{proof}
    By repeatedly applying Lemma \ref{lem:completeness-two} until the desired order is obtained.
\end{proof}

\noindent Since the edge order chosen is irrelevant, we are justified in dropping all subscripts from $\term$ and not referring to `the chosen' order. 
We can now conclude completeness.

\begin{theorem}[Completeness I]
  For any interfaced linear hypergraph $F \in \lilhyp$ there exists a unique morphism $f \in \termtype$, up to the equations of the STMC, such that $\hypfun{f} = H$. 
\end{theorem}
\begin{proof}
  By definability (Proposition~\ref{prop:definability}) and coherence (Proposition~\ref{prop:coherence}).
\end{proof}

\noindent We can also operate in the opposite direction and return to the original term after translating it into a hypergraph.
We first state a property of the definability functor.

\begin{lemma}[Compositionality of definability]\label{lem:comp-def}
    For any $F,G \in \lilhyp$ and $x \in \nat$, \[
      \term[F\seq G] = \term[F] \seq \term[G] \qquad 
      \term[F \tensor G] = \term[F] \tensor \term[G] \qquad
      \term[\trace{x}{F}] = \trace{x}{\term[F]}
    \]   
\end{lemma}
\begin{proof}
    By sliding and yanking.
\end{proof}

\begin{theorem}[Completeness II]\label{thm:termhypterm}
    For any morphism $f \in \termtype$, $\term[\hypfun{f}] = f$.
\end{theorem}
\begin{proof}
    As $f$ is freely generated, $f = \trace{x}{f_0 \seq f_1 \cdots f_{n-1}}$ by staging and global trace, where each $f_i = \id[p] \tensor k \tensor \id[q]$.
    Therefore $\hypfun{f} = \trace{x}{\hypfun{f_0} \seq \hypfun{f_1} \cdots \hypfun{f_{n-1}}}$ by definition of the definability functor, and $\term[\hypfun{f}] = \trace{x}{\term[f_0] \seq \term[f_1] \cdots \term[n-1]}$ by Lemma~\ref{lem:comp-def}.
    By sliding and yanking, each $\term[\hypfun{f_i}] = f_i$, and by composition, $\term[\hypfun{f}] = f $.
\end{proof}

\section{Graph rewriting}\label{sec:graph-rewriting}

For standard STMCs, reasoning diagrammatically using isomorphism of diagrams works well, as the axioms are absorbed into the graphical notation. 
However, we often wish to add extra structure to our categories, with associated axioms.
To solve this problem in the language of terms we use \emph{term rewriting}.

\begin{definition}[Subterm]\label{def:subterm}
    For any morphisms $f,g \in \ntermtype$, we say that $g$ is a subterm of $f$ if there exists $\hat{f_1},\hat{f_2} \in \ntermtype$ and $n,x \in \nat$ such that $f = \trace{x}{\hat{f_1} \seq \id[n] \tensor g \seq \hat{f_2}}$, where $\hat{f_1},\hat{f_2}$ do not contain any traces.
    
    \begin{center}
        \begin{tikzpicture}[yscale=-1,x=1em,y=1.25em]

    \draw (0,3) -- (1,3);
    \node [draw, minimum width=2.5em, minimum height=2.5em, anchor=west] at (1,3) {$f$};
    \draw (3.5,3) -- (4.5,3);

    \node at (6,3) {$=$};

    \draw (7.5,3.5) -- (8.5,3.5);
    \node [draw, minimum width=2.5em, minimum height=2.5em, anchor=west] at (8.5,3) {$\hat{f_1}$};
    \draw (11,2.5) -- (14.5,2.5);
    \draw (11,3.5) -- (12,3.5);
    \node [draw, minimum width=1.5em, minimum height=1.5em, anchor=west] at (12,3.5) {$g$};
    \draw (13.5,3.5) -- (14.5,3.5);
    \node [draw, minimum width=2.5em, minimum height=2.5em, anchor=west] at (14.5,3) {$\hat{f_2}$};
    \draw (17,3.5) -- (18,3.5);
    \draw [rounded corners] (17,2.5) -- (18,2.5) -- (18,1.5) -- (7.5,1.5) -- (7.5,2.5) -- (8.5,2.5);

\end{tikzpicture}
    \end{center}
\end{definition}

\begin{definition}[Term rewriting]\label{def:term-rewriting}
  A rewrite rule in $\ntermtype$ is a pair $\rrule{l}{r}$ where $\morph{l,r}{X}{Y}$ are terms in $\ntermtype$ with the same domain and codomain. 
  We write rules as $\morph{\rrule{l}{r}}{X}{Y}$. 
  A rewriting system $\mce$ is a set of rewrite rules. 
  We can perform a rewrite step in $\mce$ for two terms $g$ and $h$ (written $g \rewrite_{\mce} h$) if there exists rewriting rule $\rrule{l}{r} \in \mce$ such that $l$ is a subterm of $g$, i.e. we can write $f$ in the form $\trace{x}{\hat{f_1} \seq \id \tensor l \seq \hat{f_2}}$ for some trace-free terms $\hat{f_1}$ and $\hat{f_2}$.

  \begin{center}
      \begin{tikzpicture}[yscale=-1,x=1em,y=1.25em]

    \draw (0,3) -- (1,3);
    \node [draw, minimum width=2.5em, minimum height=2.5em, anchor=west] at (1,3) {$g$};
    \draw (3.5,3) -- (4.5,3);

    \node at (6,3) {$=$};

    \draw (7.5,3.5) -- (8.5,3.5);
    \node [draw, minimum width=2.5em, minimum height=2.5em, anchor=west] at (8.5,3) {$\hat{f_1}$};
    \draw (11,2.5) -- (14.5,2.5);
    \draw (11,3.5) -- (12,3.5);
    \node [draw, minimum width=1.5em, minimum height=1.5em, anchor=west] at (12,3.5) {$l$};
    \draw (13.5,3.5) -- (14.5,3.5);
    \node [draw, minimum width=2.5em, minimum height=2.5em, anchor=west] at (14.5,3) {$\hat{f_2}$};
    \draw (17,3.5) -- (18,3.5);
    \draw [rounded corners] (17,2.5) -- (18,2.5) -- (18,1.5) -- (7.5,1.5) -- (7.5,2.5) -- (8.5,2.5);

    \node at (21,3) {$\rewrite_{\mcr}$};

    \draw (23.5,3.5) -- (24.5,3.5);
    \node [draw, minimum width=2.5em, minimum height=2.5em, anchor=west] at (24.5,3) {$\hat{f_1}$};
    \draw (27,2.5) -- (30.5,2.5);
    \draw (27,3.5) -- (28,3.5);
    \node [draw, minimum width=1.5em, minimum height=1.5em, anchor=west] at (28,3.5) {$r$};
    \draw (29.5,3.5) -- (30.5,3.5);
    \node [draw, minimum width=2.5em, minimum height=2.5em, anchor=west] at (30.5,3) {$\hat{f_2}$};
    \draw (33,3.5) -- (34,3.5);
    \draw [rounded corners] (33,2.5) -- (34,2.5) -- (34,1.5) -- (23.5, 1.5) -- (23.5,2.5) -- (24.5,2.5);

    \node at (35.5,3) {$=$};

    \draw (37,3) -- (38,3);
    \node [draw, minimum width=2.5em, minimum height=2.5em, anchor=west] at (38,3) {$h$};
    \draw (40.5,3) -- (41.5,3);

\end{tikzpicture}
  \end{center}
\end{definition}

\noindent As is often the case in the algebraic realm, it can be difficult to identify the occurrence of $l$ in $f$ due to functoriality. 
However, unlike with the axioms of STMCs, these extra axioms are not an intrinsic part of our diagrams. 
We can no longer rely purely on hypergraph isomorphisms. 
Fortunately, we can generalise term rewriting to \emph{graph rewriting}. 
First we must formalise the notion of a subgraph of an interfaced linear hypergraph.

\begin{definition}[Subgraph]\label{def:subgraph}
  For any linear hypergraphs $F,G \in \lilhyp$, we say that $G$ is a subgraph of $F$ if there exists an embedding monomorphism $G \to F$.
\end{definition}

\noindent Intuitively, to transform some subgraph $L$ into a larger graph $G$, we apply a sequence of operations.

\begin{lemma}\label{lem:operations-mono}
  For any minimal $F,G \in \lilhyp$ and $x \in \nat$, there exist monomorphisms $F \to F \seq G$, $F \to G \seq F$, $F \to F \tensor G$, $F \to G \tensor F$ and $F \to \trace{x}{F}$.
\end{lemma}
\begin{proof}
  This is immediate, as we simply embed the operand into the result.
\end{proof}

\begin{remark}
  Note that the use of homeomorphism in the definition of trace guarantees that we can define a monomorphism $F \to \trace{x}{F}$.
  Had we used a naive definition that simply coalesced the input and output vertices, the vertex map would not be injective.
\end{remark}

\begin{lemma}\label{lem:mono-as-ops}
  For any monomorphism $\morph{m}{F}{G} \in \lilhyp$, there exist interfaced linear hypergraphs $F_1,F_2$ and $n,x \in \nat$ such that $m = \trace{x}{-} \circ (F_2~\seq~-) \circ (-~\seq~F_1) \circ (\id[n]~\tensor~-)$.
\end{lemma}
\begin{proof}
  Since composition, tensor and trace can all be represented as monomorphisms (Lemma~\ref{lem:operations-mono}), we apply the sequence of operations that transforms $F$ into $G$.
\end{proof}

\begin{lemma}\label{lem:subterms}
  For any morphisms $f,g \in \ntermtype$, $g$ is a subterm of $f$ if and only if $\hypfun{g}$ is a subgraph of $\hypfun{f}$.
\end{lemma}
\begin{proof}
  For ($\Rightarrow$) we assume that $g$ is a subterm of $f$. Therefore we know that we can write $f$ in the form of Definition \ref{def:subterm}. By definition of $\hypfun{-}$, $\hypfun{f} = \trace{x}{\hypfun{\hat{f_1}} \seq \hypfun{\id[n]} \tensor \hypfun{g} \seq \hypfun{\hat{f_2}}}$. As we can express all operations as monomorphisms (Lemma \ref{lem:operations-mono}), there exists a monomorphism $\hypfun{g} \to \hypfun{f}$, namely $\trace{x}{-} \circ (\hypfun{\hat{f_1}} \seq - ) \circ ( - \seq \hypfun{\hat{f_2}}) \circ (\hypfun{\id[n]} \tensor -)$.

  For ($\Leftarrow$) we assume that there exists an embedding monomorphism $\morph{m}{\hypfun{g}}{\hypfun{f}}$. Therefore by Lemma \ref{lem:mono-as-ops}, there exists hypergraphs $F_1,F_2$ and $n,x \in \nat$ such that $m = \trace{x}{-} \circ (F_2 \seq -) \circ (- \seq F_1) \circ (\id[n] \tensor -)$. This means that $\hypfun{f} = \trace{x}{F_2 \seq \id[n] \tensor \hypfun{g} \seq F_1}$. We apply the definability functor to both sides:
  \begin{align*}
        \term[\hypfun{f}] = f &= 	\term[\trace{x}{F_2 \seq \id[n] \tensor \hypfun{g} \seq F_1}] & \text{Completeness II} \\
                           &= \trace{x}{\term[F_2] \seq \id[n] \tensor 	\term[\hypfun{g}] \seq 	\term[F_1]} & \text{Compositionality of definability} \\
                           &= \trace{x}{\term[F_2] \seq \id[n] \tensor g \seq \term[F_1]} & \text{Completeness II}
  \end{align*}

  \noindent Therefore there exist $f_1 = \term[F_1]$ and $f_2 = \term[F_2]$ such that $f = \trace{x}{f_1 \seq \id[n] \tensor g \seq f_2}$, so $g$ is a subterm of $f$.
\end{proof}

\subsection{DPO rewriting}

A popular approach to graph rewriting is known as double pushout (DPO) rewriting~\cite{ehrig1973graph}, and we use an extension of the traditional definition that introduces an `interface'~\cite{bonchi2017confluence}.
We start by recalling the definition for ordinary hypergraphs in $\labhyp$.

\begin{definition}[DPO]\label{def:dpo}
    A DPO rewrite rule $\rrule{L}{R} \in \labhyp$ is a span $L \leftarrow K \rightarrow R$, where $K$ is discrete. A DPO system $\mcr$ is a set of DPO rules. We write $G \grewrite_\mcr H$ if there exists span $L \leftarrow K \rightarrow R$ in $\mcr$, discrete hypergraph $J$, and cospan $K \rightarrow C \leftarrow J$, such that the diagram below commutes, and the two squares are pushouts.
\end{definition}

\begin{wrapfigure}[]{r}{0.29\textwidth}
  \centering
  \vspace{1em}

  \begin{tikzcd}
    L \arrow{d} & K \arrow{l}\arrow{r}\arrow{d} & R \arrow{d}\\
    G \arrow[ur, phantom, "\usebox\pushoutl" , very near start, color=black] & C \arrow{l} \arrow{r} & H \arrow[ul, phantom, "\usebox\pushoutr" , very near start, color=black] \\
    & J \arrow{ul}\arrow{ur}\arrow{u} &
\end{tikzcd}
\end{wrapfigure}

\noindent To perform rewrite rule $\rrule{L}{R}$ in some larger graph $G$, a \emph{matching} monomorphism $L \rightarrow G$ must first be identified. 
Then its \emph{pushout complement} $K \rightarrow C \rightarrow G$ can be computed: $C$ is effectively `$G$ with $L$ removed'. 
Finally we compute the pushout $C \rightarrow H \leftarrow R$, yielding the graph $H$: the graph $G$ with $L$ replaced by $R$.

To perform rewrite rule $\rrule{L}{R}$ in some larger graph $G$, a \emph{matching} $\morph{p}{L}{G}$ must first be identified. 
Its \emph{pushout complement} $K \rightarrow C \rightarrow G$ and the pushout $C \rightarrow H \leftarrow R$ can then be computed, yielding us the graph $H$: the graph $G$ with $L$ replaced by $R$.
An example of the procedure can be seen in Figure~\ref{fig:dpo-hyp}.

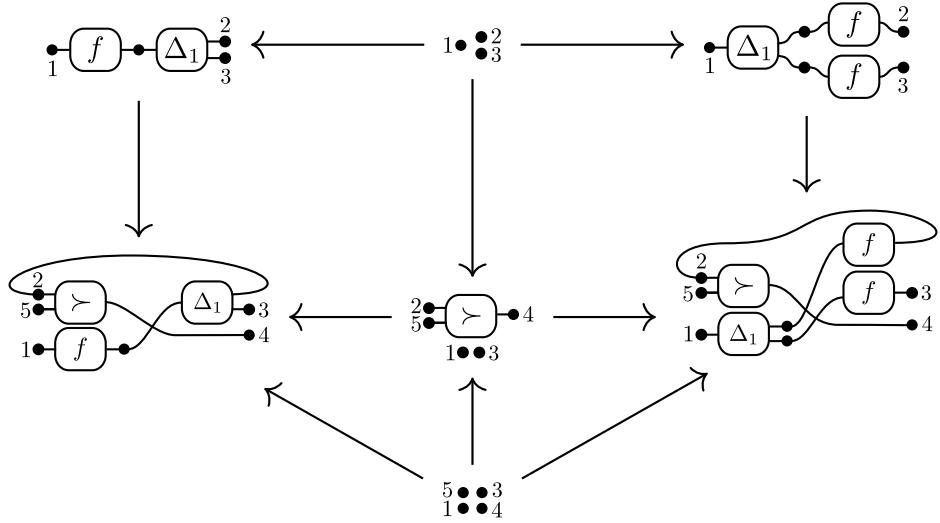
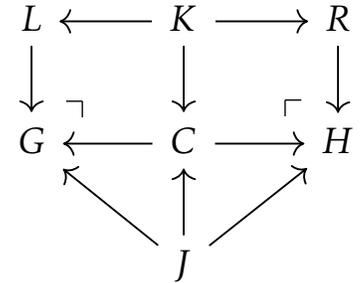
\begin{figure}
  \centering
      
\begin{tikzcd}[]
    \raisebox{-0.65em}{\includehypergraph{dpo/example-l}} \arrow{d} & \raisebox{-0.15em}{\includehypergraph{dpo/example-k}} \arrow{l} \arrow{r} \arrow{d} & \raisebox{-1em}{\includehypergraph{dpo/example-r}} \arrow{d} \\
    \raisebox{-1em}{\includehypergraph{dpo/example-g}} & \raisebox{-0.75em}{\includehypergraph{dpo/example-c}} \arrow{l} \arrow{r} & \raisebox{-0.65em}{\includehypergraph{dpo/example-h}}\\
    & \includehypergraph{dpo/example-j} \arrow{ur}\arrow{u}\arrow{ul} &
\end{tikzcd}
  \caption{An example of DPO rewriting in $\labhyp$, using the axiom of naturality in a Cartesian category.}
  \label{fig:dpo-hyp}
\end{figure}

\subsection{Adhesive categories}

Not all structures are compatible with DPO rewriting. 
For example, the pushout complements may not be unique. 
This is essential, as it implies that for a given matching monomorphism, there is a unique rewrite of the graph.
A commonly used framework that ensures the DPO procedure is always well-defined is that of \emph{adhesive categories}, introduced by Lack and Sobociński~\cite{lack2004adhesive}. 
The key property of these categories is that pushout complements are always unique for rewrite rules where each leg of the span is a monomorphism, if such a complement exists.
Additionally, they also enjoy a local Church-Rosser theorem and a concurrency theorem.
We have already met an adhesive category:

\begin{proposition}\label{prop:shyp-adhesive}
  $\labhyp$ is adhesive.
\end{proposition}
\begin{proof}
  $\labhyp$ is a slice category of a presheaf category, so is adhesive~\cite{lack2004adhesive}.
\end{proof}

\noindent We need to build on this to reach our \emph{interfaced linear} hypergraphs.
Unfortunately our category of interfaced linear hypergraphs $\lilhyp$ is not adhesive: pushout complements are not unique.
This is because we do not require that the interface orders are preserved by homomorphism, as illustrated below.

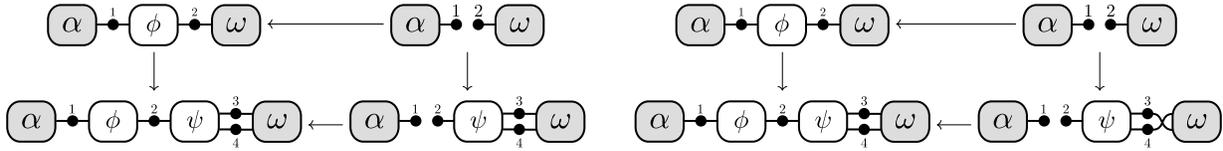
\begin{center}

\begin{tikzcd}[row sep = large]
   \raisebox{-1.3em}{\includesvg{dpo/poc-not-unique/l}} \arrow{d} & \raisebox{-1.3em}{\includesvg{dpo/poc-not-unique/k}} \arrow{l}\arrow{d} \\
   \raisebox{-1.3em}{\includesvg{dpo/poc-not-unique/g}} & \raisebox{-1.3em}{\includesvg{dpo/poc-not-unique/c-1}} \arrow{l}
\end{tikzcd}

  \quad

\begin{tikzcd}[row sep = large]
   \raisebox{-1.3em}{\includesvg{dpo/poc-not-unique/l}} \arrow{d} & \raisebox{-1.3em}{\includesvg{dpo/poc-not-unique/k}} \arrow{l}\arrow{d} \\
   \raisebox{-1.3em}{\includesvg{dpo/poc-not-unique/g}} & \raisebox{-1.3em}{\includesvg{dpo/poc-not-unique/c-2}} \arrow{l}
\end{tikzcd}

\end{center}

\noindent The left case would be the `natural' choice: we are only concerned with the section of graph being rewritten, so everything else should be left untouched.
A possible solution would be to enforce that interfaces \emph{are} preserved by homomorphism.
This raises its own problems, as usually this interface flexibility is advantageous: for example, it allows us to model the operations of our hypergraphs (Lemma~\ref{lem:operations-mono}).
A simpler option is to translate simply translate our interfaced linear hypergraphs into $\labhyp$ and perform rewriting there.
Since we are using DPO \emph{with interfaces}, we know that the interfaces are preserved throughout the rewriting procedure, and thus we can recover it at the end.

The first thing we must do is to remove the interfaces of our interfaced linear hypergraphs, so that we have a regular (uninterfaced) linear hypergraph.
We call this procedure \emph{trimming}.

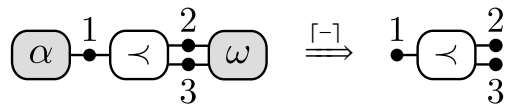
\begin{center}
  \includesvg[scale=0.6]{dpo/trim-before} \quad \raisebox{1.5em}{$\xRightarrow{\htrim{-}}$} \quad \includesvg[scale=0.6]{dpo/trim-after}
\end{center}

\begin{definition}[Trimming]\label{def:trim}
  For any interfaced linear hypergraph $F$, we can trim its interfaces with the functor $\morph{\htrim{-}}{\lilhyp}{\labhyp}$, defined as $\htrim{F} = (V, E_F, \lambda e.\vs[e][F],\lambda e. \vconnsr_F^\mon \circ \vt[e][F])$, where $V = \bigcup_{e \in E} \vs[e][F]$.
\end{definition}

\noindent We also write $\htrim{v}$ for the image of a vertex $v$ and $\htrim{V}$ for the image of a set of vertices $V$ under this functor.

\begin{lemma}\label{lem:linear-trim}
  For any $F \in \lilhyp$, $\htrim{F}$ is linear.
\end{lemma}
\begin{proof}
  Since each $\vs[e][F]$ is disjoint, each vertex is only in the sources of one edge.
  Since each $\vt[e][F]$ is disjoint and $\vconnsr_F$ is bijective, each vertex is only in the targets of one edge.
  Therefore $\htrim{F}$ is linear.
\end{proof}

\noindent Of course, we cannot simply `forget' the interfaces -- we will need them to recover the orders after we have performed rewriting.
We can keep track of the interfaces using the $J$ graph from the DPO diagram in Definition~\ref{def:dpo}.

\begin{center}
  \includesvg[scale=0.6]{dpo/trim-before} \quad \raisebox{1.5em}{$\xRightarrow{\io{-}}$} \quad \includesvg[scale=0.6]{dpo/io}
\end{center}

\begin{definition}[Interface]
  For any $\morph{G}{m}{n} \in \lilhyp$, we write $\io{G} \in \labhyp$ for the discrete hypergraph with $m + n$ vertices.
\end{definition}

\begin{definition}[Interfacing]\label{def:interfaces}
  For any $J,H \in \labhyp$, where $J$ is discrete, we say that $J$ interfaces $H$ if there exists 
  \begin{itemize}[noitemsep]
    \item a partition $V_J = V_\einput + V_\eoutput$ equipped with total orders on the two subsets
    \item a morphism $\morph{m}{J}{H}$ that is injective when restricted to the two subsets
  \end{itemize}
  such that
  \begin{itemize}[noitemsep]
     \item for any $v \in V_\einput$, $\inputs{m(v)} = 0$
     \item for any $v \in V_\eoutput$, $\outputs{m(v)} = 0$
  \end{itemize}
  Furthermore, for any $v \in V_H$ not in the image of $m$, $\inputs{v},\outputs{v} > 0$. 
\end{definition}

\noindent We are now ready to formulate the notion of a rewrite rule in $\lilhyp$.

\begin{definition}[Rewrite rule]\label{def:rewrite-rule}
  For a pair of interfaced linear hypergraphs $\morph{L,R}{m}{n} \in \lilhyp$, we represent their corresponding rewrite rule in $\labhyp$ as $\htrim{L} \leftarrow \io{L} \rightarrow \htrim{R}$, with the monomorphisms $\morph{p}{\io{L}}{L}$ and $\morph{q}{\io{L}}{R}$ defined such that for any vertex $v \in \io{L}$, there exists $x \in \nat$ such that $p(v) = \htrim{\proj{x}(\vt[\interface][L])}$ and $q(v) = \htrim{\proj{x}(\vs[\interface][L])}$.  
  We write this rewrite rule as $\htrim{\rrule{L}{R}}$.
\end{definition}

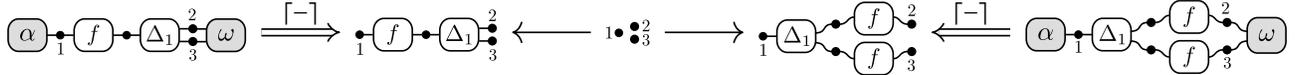
\begin{center}

\begin{tikzcd}[row sep = large]
    \raisebox{-0.6em}{\includehypergraph{dpo/example-l-linear}} \arrow[Rightarrow]{r}{\htrim{-}} & 
    \raisebox{-0.6em}{\includehypergraph{dpo/example-l}} & 
    \raisebox{-0.15em}{\includehypergraph{dpo/example-k}} \arrow{l} \arrow{r} & 
    \raisebox{-1.05em}{\includehypergraph{dpo/example-r}}& 
    \raisebox{-1em}{\includehypergraph{dpo/example-r-linear}} \arrow[Rightarrow,swap]{l}{\htrim{-}}
\end{tikzcd}

\end{center}

\noindent For a rewriting system $\mce$, we write $\hypfun{\mce}$ for the interpretation of its rewrite rules as spans of hypergraphs as illustrated in \ref{def:rewrite-rule}.

\begin{lemma}
  For any interfaced linear hypergraph $G$, $\io{G}$ interfaces $\htrim{G}$ with the partitioning $V_\einput = \htrim{\vt[\interface][G]}$, $V_\eoutput = \htrim{\vs[\interface][G]}$.
  For a rewrite span $\htrim{\rrule{L}{R}}$, $\io{L}$ interfaces $\htrim{L}$ and $\htrim{R}$ with the partitioning $V_\einput = \htrim{\vt[\interface][L]}$, $V_\eoutput = \htrim{\vs[\interface][L]}$.
\end{lemma}
\begin{proof}
  For $G$ and $L$ this is immediate.
  For $R$, this follows as the left and right hand side of a rewrite rule must have the same domain and codomain.
\end{proof}

\noindent Once we have translated back into regular hypergraphs, we can perform rewriting as described above.
Once we have acquired a rewritten hypergraph, we must use the interface $J$ to reconstruct an interfaced linear hypergraph.

\begin{proposition}\label{prop:linear}
  For a rewrite span $\htrim{\rrule{L}{R}} \in \htrim{\mce}$ and linear hypergraph $G$, if there exists a matching $L \to G$, and $G \grewrite_{\htrim{\mcr}} H$, then $H$ is linear.
\end{proposition}
\begin{proof}
  Since $\htrim{G}$ is linear, then $C$ must also be linear by homomorphism, and $R$ is linear by Lemma $\ref{lem:linear-trim}$.
  By the DPO diagram, we know that $H$ is the pushout of $C \xleftarrow{p} K \xrightarrow{q} R$, where $K = V_\einput + V_\eoutput$.
  For any $v \in V_\einput$, $\inputs{q(v)} = 0$ as $K$ interfaces $R$ and $\outputs{p(v)} = 0$ as $K$ interfaces $L$ and $C$ is the complement $G - L$.
  Therefore in the pushout $R \xrightarrow{r} H \xleftarrow{s}$ for any $v \in V_\einput$, $r(p(v)) = s(q(v))$ will be the source of an edge from $R$ and the target of an edge from $C$, and only one each since $R$ and $C$ are linear.
  A similar argument follows for each $v \in V_\eoutput$.
  Therefore $H$ is linear.
\end{proof}

\begin{lemma}\label{lem:interface-g-h}
  For a DPO diagram as illustrated in Definition~\ref{def:dpo} and rewrite rule $\htrim{\rrule{L}{R}}$, if $J$ interfaces $G$ then it also interfaces $H$.
\end{lemma}
\begin{proof}
  As $J$ interfaces $G$ and there exists a morphism $C \to G$, $J$ must also interface $C$ by homomorphism condition: a morphism must preserve sources and targets so if a vertex is not a source or target in $G$, it cannot be a source or target in $C$.
  The reasoning is then the same as Proposition~\ref{prop:linear}: the pushout hypergraph $H$ only `fills the hole', so it will not affect the interfaces.
  Therefore $J$ interfaces $H$.
\end{proof}

\noindent To retrieve an interfaced linear hypergraph from a regular one with interfacing morphism, we use the \emph{reinterface} functor $\morph{\hinter{-}{=}}{\labhyp \times \labhyp}{\lilhyp}$.

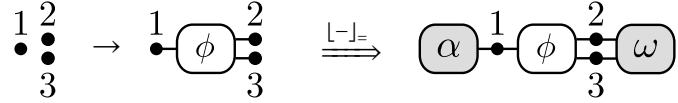
\begin{center}
  \raisebox{0em}{\includesvg[scale=0.6]{dpo/io}}
  \quad\raisebox{1.5em}{$\rightarrow$}\quad
  \includesvg[scale=0.6]{dpo/inter-before}\quad
  \quad\raisebox{1.3em}{$\xRightarrow{\hinter{-}{=}}$} \quad \includesvg[scale=0.6]{dpo/inter-after}
\end{center}

\begin{definition}[Reinterfacing]
  For any $H,J \in \labhyp$ and morphism $\morph{p}{J}{H}$ such that $J$ interfaces $H$ with partition $V_\einput + V_\eoutput$, we can reinterface $H$ with respect to $J$ with the functor $\morph{\hinter{H}{J}}{\labhyp}{\lilhyp}$, defined as $\hinter{H}{J} = (E_H, S', T', \lambda s_v. t_v, \labels_H\}$ where 
  \begin{gather*}
      \vs[e \in E_H]' = \{s_{v} \text{ fresh in } \atoms \,|\, v \in \esources_H(e)\} \qquad 
      \vs[\interface]' = \{s_v \text{ fresh in } \atoms \,|\, v \in p^\mon(V_\eoutput)\} \\
      \vt[e \in E_H]' = \{t_{v} \text{ fresh in } \atoms \,|\, v \in \etargets_H(e)\} \qquad
      \vt[\interface]' = \{t_v \text{ fresh in } \atoms \,|\, v \in p^\mon(V_\einput)\}
  \end{gather*}
\end{definition}
\begin{definition}[Rewriting]\label{def:rewriting}
  For two hypergraphs $L,R \in \lilhyp$ and rewrite rule $\rrule{L}{R} \in \mcr$, we define the corresponding rewrite rule $\htrim{\rrule{L}{R}} \in \labhyp$ as in Definition~\ref{def:rewrite-rule}.
  For a morphism $L \to G \in \lilhyp$ such that $\htrim{L} \to \htrim{G}$ is a matching, we perform the rewrite $\htrim{G} \grewrite_{\htrim{\mcr}} H \in \labhyp$.
  The resulting hypergraph in $\lilhyp$ is $\hinter{H}{\io{G}}$.
\end{definition}

\noindent We call a monomorphism $L \to G \in \lilhyp$ an $\interface$-matching if the corresponding morphism $\htrim{L} \to \htrim{G}$ is a matching, i.e. $\io{L} \rightarrow \htrim{L} \to \htrim{G}$ has a pushout complement.

\begin{theorem}
  For a rewrite rule $\rrule{L}{R} \in \lilhyp$ and $\interface$-matching $L \to G$, the procedure in Definition \ref{def:rewriting} yields a unique interfaced linear hypergraph.
\end{theorem}
\begin{proof}
  The DPO procedure yields us a unique hypergraph that is linear by Proposition~\ref{prop:linear}. As $\io{G}$ interfaces $\htrim{G}$, it also interfaces $H$ by Lemma~\ref{lem:interface-g-h}, so we can obtain a unique interfaced linear hypergraph $\hinter{H}{\io{G}}$.
\end{proof}

\noindent In general, although we know that pushout complements are unique in adhesive categories, we do not actually have a guarantee that they exist for a given monomorphism $L \to G$.
Usually some variety of the \emph{no-dangling-hyperedge} condition is used to identify the monomorphisms that allow for the definition of pushout complements.
However, in our case we actually can guarantee they exist!

\begin{definition}[No-dangling-hyperedge condition]
  Morphisms $K \xrightarrow{p} L \xrightarrow{q} G$ in $\labhyp$ satisfy the no-dangling-hyperedge condition if, for each hyperedge $h$ not in the image of $q$, every source and target of $h$ is either (i) not in the image of $q$ or (ii) in the image of $q \circ p$.
\end{definition}

\begin{lemma}
  For any $L,G \in \labhyp$, the morphisms $\io{L} \xrightarrow{p} \htrim{L} \xrightarrow{q} \htrim{G}$ always satisfy the no-dangling-hyperedge condition.
\end{lemma}
\begin{proof}
  Assume there is an edge $e \in \htrim{G}$ that is not in the image of $q$, with a source $v_G = m(v_L)$.
  Since $\htrim{L}$ is linear (Lemma~\ref{lem:linear-trim}), $\outputs{v_L} = 0$ by preservation of sources.
  By definition of $\io{L}$, any vertex with out-degree $0$ must be in the image of $p$, so $v_G$ is in the image of $q \circ p$.
  The same follows for targets.
\end{proof}

\begin{theorem}[\texorpdfstring{$\interface$-matchings}{Interface-matchings}]\label{thm:matchings}
  All monomorphisms in $\lilhyp$ are $\interface$-matchings.
\end{theorem}
\begin{proof}
  To form a pushout complement $C$ we remove elements of $\htrim{L}$ from $\htrim{G}$ in the obvious way.
  This is a well-formed linear hypergraph since $\io{L} \to \htrim{L} \to \htrim{G}$ always satisfies the no-dangling-hyperedge condition.
\end{proof}

\noindent Now all that remains is to show that our notion of rewriting in the graph language is equivalent to that in the term language.

\begin{theorem}
  For a set of equations $\mce$ in $\ntermtype$, $g \rewrite_{\mce} h$ if and only if $\hypfun{g} \grewrite_{\hypfun{\mce}} \hypfun{h}$.
\end{theorem}
\begin{proof}
  For $(\Rightarrow)$, assume that $g \rewrite_{\mce} h$. By Definition \ref{def:term-rewriting} this means that there exists $\rrule{l}{r} \in \mcr$ such that
  \begin{center}
      \begin{tikzpicture}[yscale=-1,x=1em,y=1.25em]

    \draw (0,3) -- (1,3);
    \node [draw, minimum width=2.5em, minimum height=2.5em, anchor=west] at (1,3) {$g$};
    \draw (3.5,3) -- (4.5,3);

    \node at (6,3) {$=$};

    \draw (7.5,3.5) -- (8.5,3.5);
    \node [draw, minimum width=2.5em, minimum height=2.5em, anchor=west] at (8.5,3) {$\hat{f_1}$};
    \draw (11,2.5) -- (14.5,2.5);
    \draw (11,3.5) -- (12,3.5);
    \node [draw, minimum width=1.5em, minimum height=1.5em, anchor=west] at (12,3.5) {$l$};
    \draw (13.5,3.5) -- (14.5,3.5);
    \node [draw, minimum width=2.5em, minimum height=2.5em, anchor=west] at (14.5,3) {$\hat{f_2}$};
    \draw (17,3.5) -- (18,3.5);
    \draw [rounded corners] (17,2.5) -- (18,2.5) -- (18,1.5) -- (7.5,1.5) -- (7.5,2.5) -- (8.5,2.5);

    \node at (21,3) {$\rewrite_{\mcr}$};

    \draw (23.5,3.5) -- (24.5,3.5);
    \node [draw, minimum width=2.5em, minimum height=2.5em, anchor=west] at (24.5,3) {$\hat{f_1}$};
    \draw (27,2.5) -- (30.5,2.5);
    \draw (27,3.5) -- (28,3.5);
    \node [draw, minimum width=1.5em, minimum height=1.5em, anchor=west] at (28,3.5) {$r$};
    \draw (29.5,3.5) -- (30.5,3.5);
    \node [draw, minimum width=2.5em, minimum height=2.5em, anchor=west] at (30.5,3) {$\hat{f_2}$};
    \draw (33,3.5) -- (34,3.5);
    \draw [rounded corners] (33,2.5) -- (34,2.5) -- (34,1.5) -- (23.5, 1.5) -- (23.5,2.5) -- (24.5,2.5);

    \node at (35.5,3) {$=$};

    \draw (37,3) -- (38,3);
    \node [draw, minimum width=2.5em, minimum height=2.5em, anchor=west] at (38,3) {$h$};
    \draw (40.5,3) -- (41.5,3);

\end{tikzpicture}
  \end{center}
  \noindent For the if direction, we use $\hypfun{-}$ to translate the rewrite rule $\rrule{l}{r}$ into a rewrite rule in the language of interfaced linear hypergraphs as described in Definition~\ref{def:rewrite-rule}.
  Therefore there exists a span $\htrim{\hypfun{l}} \leftarrow \io{\hypfun{l}} \to \htrim{\hypfun{r}}$.
  Since $l$ is a subterm of $f$, $L$ corresponds to a subgraph of $\hypfun{g}$ by Lemma \ref{lem:subterms}, so there exists a monomorphism $\hypfun{l} \to \hypfun{g}$. 
  Any monomorphism $\hypfun{l} \to \hypfun{g}$ is an $\interface$-matching by Theorem~\ref{thm:matchings} so the pushout complement of $\io{\hypfun{l}} \rightarrow \htrim{\hypfun{l}} \rightarrow \htrim{\hypfun{g}}$ exists.
  We can use the same procedure for $\htrim{\hypfun{r}}$ and $\htrim{\hypfun{h}}$ and assert that the pushout complement of $\io{\hypfun{r}} \to \htrim{\hypfun{r}} \to \htrim{\hypfun{h}}$ also exists. 
  We know that the two pushout complements are $C_L = \htrim{\hypfun{g}} - \htrim{\hypfun{l}}$ and $C_R = \htrim{\hypfun{h}} - \htrim{\hypfun{r}}$ respectively.
  Since the only difference between $\htrim{\hypfun{g}}$ and $\htrim{\hypfun{h}}$ is $\htrim{\hypfun{l}}$ and $\htrim{\hypfun{r}}$ by definition of rewrite rules, $C_L = C_R$. 
  We name this pushout complement $C$, and note that this implies that $\htrim{\hypfun{l}} \to \htrim{\hypfun{g}} \leftarrow C$ and $\htrim{\hypfun{r}} \to \htrim{\hypfun{h}} \leftarrow C$ are pushouts.
  By Lemma~\ref{lem:interface-g-h}, $\io{\hypfun{g}}$ interfaces $\htrim{\hypfun{g}}$ and $\io{\hypfun{h}}$ interfaces $\htrim{\hypfun{h}}$, and as $\hypfun{g}$ and $\hypfun{h}$ have the same type, $\io{\hypfun{g}} \equiv \io{\hypfun{h}}$: we name this hypergraph $J$.
  We can conclude that there exist morphisms $J \to \htrim{\hypfun{g}}$, $J \to C$ and $J \to \htrim{\hypfun{h}}$.
  Therefore the DPO diagram specified in Definition~\ref{def:dpo} exists, so $G \grewrite_{\hypfun{\mce}} H$.

  For the only if direction, we assume that $\hypfun{g} \grewrite_{\hypfun{\mce}} \hypfun{h}$ for some rewrite rule $\rrule{\hypfun{l}}{\hypfun{r}}$. 
  This means there exists a DPO diagram with matchings $\htrim{\hypfun{l}} \to \htrim{\hypfun{g}}$ and $\htrim{\hypfun{r}} \to \htrim{\hypfun{h}}$. 
  Since every monomorphism in $\lilhyp$ corresponds to a matching in $\labhyp$ by Theorem~\ref{thm:matchings}, this means there must exist monomorphisms $\hypfun{l} \to \hypfun{g}$ and $\hypfun{r} \to \hypfun{h}$. 
  Therefore by Lemma \ref{lem:subterms}, $l$ is a subterm of $g$ and $r$ is a subterm of $h$, so $g \rewrite_{\mce} h$.
\end{proof}

\noindent The results of this section gives us a graph rewriting system for rules that are spans of monomorphisms, i.e. those in which no vertices of the interface $K$ are collapsed into one. 
This is since the pushout complement is only uniquely defined for $K \xrightarrow{p} \htrim{L} \rightarrow \htrim{G}$ if $p$ is mono.
However, there are cases where rewrite rules that are \emph{not} spans of monomorphisms might be desirable, such as a rule introducing an edge to an empty wire.

\begin{center}

\begin{tikzcd}[row sep = large, column sep = large]
   \raisebox{-0.2em}{\includesvg{dpo/not-ll/l}} & \raisebox{-0.2em}{\includesvg{dpo/not-ll/k}} \arrow{l}\arrow{r} & \raisebox{-1.4em}{\includesvg{dpo/not-ll/r}} \\
\end{tikzcd}

\end{center}

\noindent Since the left leg of the span is not a monomorphism, the pushout complement is not unique and therefore multiple derivations can be performed for one matching.
Even more crucially, one of the derivations would be degenerate if translated back into an interfaced linear hypergraph: there is a situation where a source connects to a source!

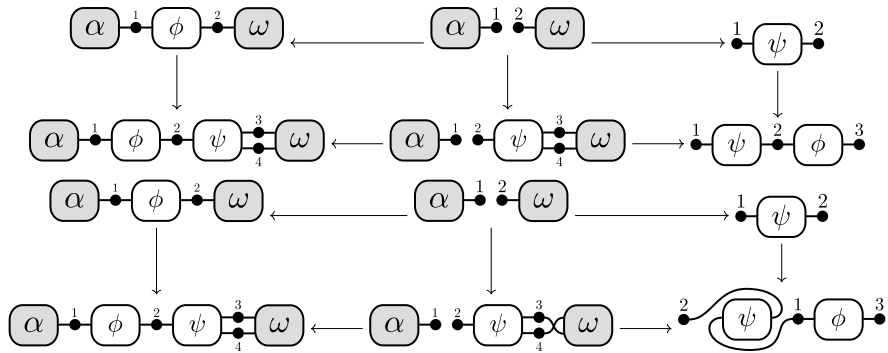
\begin{center}

\begin{tikzcd}[row sep = large, column sep = large]
   \raisebox{-0.2em}{\includesvg{dpo/not-ll/l}} \arrow{d} & \raisebox{-0.2em}{\includesvg{dpo/not-ll/k}} \arrow{l}\arrow{d}\arrow{r} & \raisebox{-1.3em}{\includesvg{dpo/not-ll/r}} \arrow{d} \\
   \raisebox{-1.3em}{\includesvg{dpo/not-ll/g}} & \raisebox{-1.3em}{\includesvg{dpo/not-ll/c-1}} \arrow{l}\arrow{r} & \raisebox{-1.3em}{\includesvg{dpo/not-ll/h-1}}
\end{tikzcd}

  \quad

\begin{tikzcd}[row sep = large, column sep = large]
   \raisebox{-0.2em}{\includesvg{dpo/not-ll/l}} \arrow{d} & \raisebox{-0.2em}{\includesvg{dpo/not-ll/k}} \arrow{l}\arrow{d}\arrow{r} & \raisebox{-1.3em}{\includesvg{dpo/not-ll/r}} \arrow{d} \\
   \raisebox{-1.3em}{\includesvg{dpo/not-ll/g}} & \raisebox{-1.3em}{\includesvg{dpo/not-ll/c-2}} \arrow{l}\arrow{r} & \raisebox{-1.3em}{\includesvg{dpo/not-ll/h-2}}
\end{tikzcd}

\end{center}

\noindent In~\cite{bonchi2016rewriting} this phenomenon is permitted thanks to a compact closed structure in which wires can be directed arbitrarily. 
However in our traced context we have a strict notion of source and target which cannot be altered.
To solve this issue, we can use homeomorphism to perform an expansion on the left hand side of the rule. 
This yields us a span of monomorphisms as desired and guarantees us a unique pushout complement.
Once we have completed our rewriting, we can perform any smoothings if necessary in order to retrieve a minimal hypergraph.

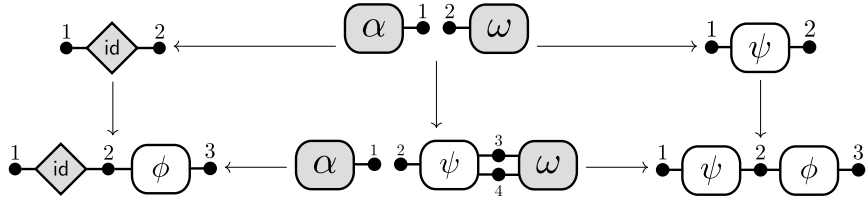
\begin{center}

\begin{tikzcd}[row sep = large, column sep = large]
   \raisebox{-1.4em}{\includesvg{dpo/not-ll/l-s}} \arrow{d} & \raisebox{-0.2em}{\includesvg{dpo/not-ll/k}} \arrow{l}\arrow{d}\arrow{r} & \raisebox{-1.3em}{\includesvg{dpo/not-ll/r}} \arrow{d} \\
   \raisebox{-1.4em}{\includesvg{dpo/not-ll/g-s}} & \raisebox{-1.4em}{\includesvg{dpo/not-ll/c-1}} \arrow{l}\arrow{r} & \raisebox{-1.4em}{\includesvg{dpo/not-ll/h-1}}
\end{tikzcd}

\end{center}

\section{Case study: digital circuits}\label{sec:case-study}

The initial motivation for developing a graph rewriting system was for use as an operational semantics for digital circuits. 
In this section we will detail how we can apply our framework for this purpose.

As detailed in \S\ref{sec:cartesian}, when a STMC is Cartesian, it admits a \emph{fixpoint operator}, which we can use to represent feedback.
This is leads immediately to models of digital circuits.
However, as we must now consider terms up to the axioms of Cartesian categories in addition to the axioms of STMCs, equality is no longer captured by graph isomorphism and we must use graph rewriting.
In the hypergraph interpretation the diagonal morphism and the unique morphism into the terminal object are represented as edges, and the Cartesian axioms are expressed as graph rewrite rules.

\begin{example}\label{ex:dpo}

We wish to interpret the Cartesian axioms in $\ntermtype$ as graph rewrite rules in $\labhyp$. 
We start our example by interpreting naturality of $\ccopy{1}$ as a span of monomorphisms $\Rightarrow_{\ccopy{1}} : L \xleftarrow{m} C \xrightarrow{n} R$ in $\labhyp$:
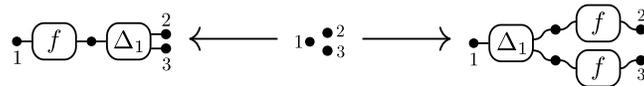
\begin{center}

\begin{tikzcd}[row sep = large]
    \raisebox{-0.6em}{\includehypergraph{axioms/naturality-copy-lhs-trimmed}} & \raisebox{-0.25em}{\includehypergraph{axioms/naturality-copy-interface}} \arrow{l}\arrow{r} & \raisebox{-1.05em}{\includehypergraph{axioms/naturality-copy-rhs-trimmed}}
\end{tikzcd}

\end{center}
Terms such as $\trace{1}{\join \tensor\, f \seq \swap{1}{1} \seq \ccopy{1} \tensor 1}$ (with $\join$ as defined in \S~\ref{sec:hypergraphs}) illustrate why the graph rewriting approach is useful. 
The symmetry obscures the existence of the left hand side of $\Rightarrow_\ccopy{n}$, so rewriting \textit{qua} term requires the expenditure of additional computation to identify the redex. 
In contrast, in the hypergraph representation the rule is more easily identifiable as seen in \figurename~\ref{fig:dpo-example}.

\begin{figure}[tb]
    \centering
        
\begin{tikzcd}[row sep = large]
    \raisebox{-0.3em}{\includehypergraph{dpo/example-l-linear}} \arrow[Rightarrow]{r}{\htrim{-}} \arrow{d} & 
    \raisebox{-0.65em}{\includehypergraph{dpo/example-l}} \arrow{d} & 
    \raisebox{-0.15em}{\includehypergraph{dpo/example-k}} \arrow{l} \arrow{r} \arrow{d} & 
    \raisebox{-1em}{\includehypergraph{dpo/example-r}} \arrow{d} & 
    \raisebox{-1em}{\includehypergraph{dpo/example-r-linear}} \arrow[Rightarrow,swap]{l}{\htrim{-}} \\
    \raisebox{-1em}{\includehypergraph{dpo/example-g-int}} \arrow[Rightarrow,swap]{r}{\htrim{-}} \arrow[Rightarrow,bend right]{drr}{\io{-}} &
    \raisebox{-1em}{\includehypergraph{dpo/example-g}} & 
    \raisebox{-0.75em}{\includehypergraph{dpo/example-c}} \arrow{l} \arrow{r} & 
    \raisebox{-0.65em}{\includehypergraph{dpo/example-h}} \arrow[Rightarrow,swap]{r}{\hinter{-}{{}}} & 
    \raisebox{-0.8em}{\includehypergraph{dpo/example-h-int}} & \\
    & & \includehypergraph{dpo/example-j} \arrow{ur}\arrow{u}\arrow{ul} &
\end{tikzcd}
    \caption{An example of a DPO diagram using the Cartesian copy axiom.}
    \label{fig:dpo-example}
\end{figure}
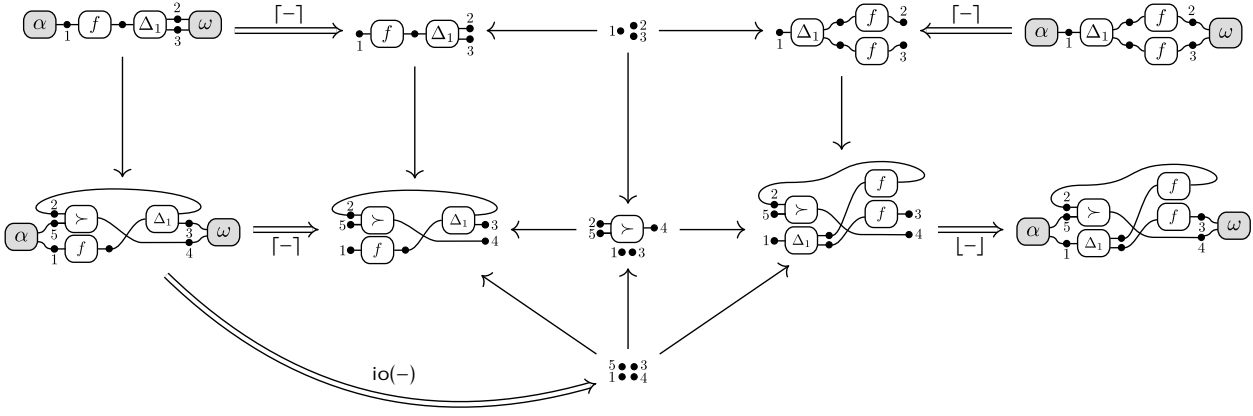
\end{example}

\noindent We present digital circuits as morphisms in a free traced PROP, where the objects correspond to the number of wires in a bus~\cite{ghica2016categorical}. 
Morphisms are freely generated over a \emph{circuit signature} containing values $\morph{v}{0}{1}$, forming a lattice, gates $\morph{k}{m}{1}$, and `special morphisms' for forking ($\morph{\fork}{1}{2}$), joining ($\morph{\join}{2}{1}$) and stubbing ($\morph{\stub}{1}{0}$) wires.
We write $\mathbf{v}=v_0\tensor\cdots\tensor v_m$. 
Gates and special morphisms have associated axioms:
\begin{align*}
    &v \,\seq \fork\, = v \tensor v & 0 \to 2 \\
    &v \tensor v' \seq \join\, = v \lub v' & 0 \to 1 \\
    &v \seq \stub = 0 & 0 \to 0 \\
    &\textbf{v} \seq k = v \text{ s.t. if } v_i \in \mathbf{v} \sqsupseteq v_i' \in \mathbf{v'} \text{ then } \mathbf{v} \seq k \sqsupseteq \mathbf{v'} \seq k
     & 0 \to 1
\end{align*}
Additionally, we can use the special morphisms to define Cartesian copy and delete maps, along with a dual notion of maps $\cmerge{n}$ to join buses of arbitrary size. 
The outputs of circuits depend wholly on their inputs, so a circuit will always give the same output for some given inputs. 
It follows that any circuits with zero outputs are considered to be equal.

Delay is represented as a morphism $\morph{\delay_t}{1}{1}$, parameterised over a \textit{time monoid} $t \in \mathbf{T}$. 
Since this adds a temporal aspect to our circuits, we switch from values to \emph{waveforms}, sequences of values over time. 
The axioms of  delay are:
\begin{align*}
&\delay \seq k = k \seq \delay & m \to 1 \\
&\delay^2 \tensor v \tensor v' \seq \cmerge{2} \seq k = ((\delay^2 \seq k) \tensor (v \tensor v' \seq k) \seq \cmerge{1} & 2 \to 1 \\ 
&\bot \seq \delay = \bot & 0 \to 0 \\
&\delay \seq \stub = \stub & 0 \to 0
\end{align*}
Of most interest is the second axiom (\emph{Streaming}), an interaction between gates and time which corresponds to stream manipulation.
Two copies of the gate $k$ are used, one to handle the `head' of the waveform and the other to handle the `tail'.
Feedback is represented using the trace, so the fixpoint operator can be used to `unfold' the circuit.

We have already seen in Example~\ref{ex:dpo} that term rewriting can be computationally awkward.
Another problem in the reduction of circuits is that some of the local traces can be unproductive due to infinite unfolding.
The key result in~\cite{ghica2016categorical} is that a circuit can be converted to a graph and brought to a normal form in which the reduction can be made efficient.
A guarantee can be given for a circuit to be either productive or, if unproductive, the lack of productivity can be efficiently detected.

\section{Generalisation}\label{sec:generalisation}

So far, we have considered only terms in PROPs, where objects are natural numbers and tensor product is addition.
This corresponds especially well with the notion of a linear hypergraph with $m$ inputs and $n$ outputs corresponding to a morphism $m \to n$. 
However, we may also wish to represent categories with arbitrary objects, as we do with string diagrams. 
It is not difficult to generalise our hypergraphs to correspond to terms in \emph{any} STMC with set of objects $\vhypsig$ and morphisms freely generated over signature $\ehypsig$.
We fix this STMC here as $\gtermtype$.

Rather than just thinking of the type of linear hypergraphs as $m \to n$, we could instead decompose them into the form $1 \tensor 1 \tensor \cdots \tensor 1 \to 1 \tensor 1 \tensor \cdots \tensor 1$, since tensor product is addition.
In fact, this corresponds even more closely to our graphical notation, as each of the `wires' in the graph represents the natural number $1$. 
However, when we write the types this way, we see that there is no reason to restrict ourselves to just numbers. The wires can instead correspond to arbitrary objects, as in string diagrams.
We can communicate this idea in our hypergraphs by labelling vertices as well as edges.
These labels correspond to the objects in our category. 

We extend hypergraph signatures with this vertex label set $\vhypsig$ to create a \emph{generalised hypergraph signature}.
Since the inputs and outputs of boxes are tensors of objects rather than natural numbers, edge labels now have associated domains and codomains of elements of the free monoid under $\tensor$ on $\vhypsig$ rather than arities and coarities of natural numbers.
We introduce the vertex labelling operations $\morph{\vtlabels}{\vt}{\vhypsig^\mon}$ and $\morph{\vslabels}{\vs}{\vhypsig^\mon}$ and the following conditions:

\begin{itemize}
    \item For all $e \in \edges$, $\vslabels(\vs[e]) = \dom{\labels(e)}$
    \item For all $e \in \edges$, $\vtlabels(\vt[e]) = \cod{\labels(e)}$
    \item For all $t\in \vt$ and $s \in \vs$, $(\vtlabels \circ \vconnsr)(t) = \vslabels(s)$
\end{itemize}

\begin{example}[Generalised interfaced linear hypergraph]
    Below is an example of a generalised interfaced linear hypergraph of type $A \to D$, for a generalised signature \[(\ehypsig,\vhypsig) = (\{\morph{\,{\fork}{}}{C}{B \tensor D}, \morph{\,\join{}}{B \tensor A}{C}\},\{A, B, C\}).\] 
    The corresponding term is $\trace{C}{{\fork} \tensor \id[A] \seq \id[B] \tensor \swap{D}{A} \seq \join \tensor \,\id[D]}$.
    
    \begin{center}
        \includesvg[scale=0.6]{example-generalised}
    \end{center}
\end{example}

\noindent Generalised interfaced linear hypergraph homomorphisms are as with interfaced linear hypergraphs but with the addition of vertex label conditions.

\begin{center}

\begin{tikzcd}[row sep = large]
    \vs_F \arrow{r}{\vslabels_F} \arrow{d}{\vmaps{h}} & {\vhypsig} \arrow{d}{\id} \\
    \vs_G \arrow{r}{\vslabels_G} & {\vhypsig}
\end{tikzcd}

\begin{tikzcd}[row sep = large]
    \vt_F \arrow{r}{\vtlabels_F} \arrow{d}{\vmapt{h}} & {\vhypsig} \arrow{d}{\id} \\
    \vt_G \arrow{r}{\vtlabels_G} & {\vhypsig} 
\end{tikzcd}

\end{center}

\noindent As before, generalised interfaced linear hypergraphs over the signature $(\ehypsig,\vhypsig)$ are the objects in the category $\glhyp^\interface$, with generalised interfaced linear hypergraph homomorphisms as the morphisms between them.
Operations on generalised hypergraphs are defined similarly to before but with the addition of vertex labels, which are affected in the same way as edge labels. 
Therefore, we can also form a category where the generalised linear hypergraphs are morphisms:

\begin{definition}[Generalised category of hypergraphs]
    Generalised interfaced linear hypergraphs over the signature $(\ehypsig,\vhypsig)$ form a symmetric traced monoidal category $\ghypterm$ with objects the elements of ${\vhypsig}^\mon$, elements of $\ghypterm(M,N)$ as the generalised hypergraphs of type $M \to N$, unit of composition as the identity hypergraph, monoidal tensor as $\tensor$ as the empty hypergraph, symmetry on $M$ and $N$ as $\swap{M}{N}$, and trace on $X$ as \[\morph{\trace{X}{-}}{\ghypterm(X~\tensor~M,X~\tensor~N)}{\ghypterm(M,N)}.\]
    
\end{definition}

\noindent We define the operations \begin{align*}
    \morph{\hypfun{-}_{\ehypsig,\vhypsig}}{\gtermtype}{\ghypterm} \\ \morph{\term_{\ehypsig,\vhypsig,\leq}}{\ghypterm}{\gtermtype}
\end{align*} in a similar manner to before. Soundness and completeness translate smoothly into the generalised case.

\begin{theorem}[Generalised soundness]
    For any morphisms $f,g \in \gtermtype$ if $f=g$ under the equational theory of the category then their interpretations as morphisms are isomorphic, $\hypfun{f}_{\ehypsig,\vhypsig}~\equiv~\hypfun{g}_{\ehypsig,\vhypsig}$.
\end{theorem}

\begin{theorem}[Generalised completeness]
    For any generalised interfaced linear hypergraph $F \in \ghypterm$ there exists a unique morphism $f \in \gtermtype$, up to the equations of the STMC, such that $\hypfun{f}_{\ehypsig,\vhypsig} = F$. 
    Furthermore, for any morphism $f \in \gtermtype$, $\term[\hypfun(f)] = f$.
\end{theorem}

\section{Related and further work}\label{sec:related-further-work}

\subsection{Related work}

Diagrammatic languages for traced categories are certainly not new: their formal presentation as \emph{string diagrams} has existed since the 90s~\cite{joyal1991geometry,joyal1996traced}.
A soundness and completeness theorem for this language, while common knowledge for many years but often omitted or only proved for certain signatures \cite{selinger2010survey}, was finally formally proved in \cite{kissinger2014abstract}.
Combinatorial languages predate even this, having existed since at least the 80s in the guise of \emph{flowchart schemes} \cite{stefuanescu1986feedback,cazanescu1990algebraic,cazanescu1994feedback}. 
This type of diagrams have also been used to show the completeness of finite dimensional vector spaces~\cite{hasegawa2008finite}, and when equipped with a dagger, Hilbert spaces~\cite{selinger2012finite}.
However, it was not shown whether these diagrams were also suitable for graph rewriting in the presence of Cartesian structure, which is essential for our goal of obtaining an operational semantics of digital circuits.

Graphical languages for traced categories have seen many applications, such as to illustrate cyclic lambda calculi~\cite{hasegawa1997recursion}, or to reason graphically about programs~\cite{schweimeier1999categorical}.
But we are not just concerned with diagrammatic languages as a standalone concept: we are interested in the context of performing \emph{graph rewriting} with them in order to reason with additional axioms.
This has been studied in the context symmetric traced categories before, most notably with \emph{open graphs} \cite{dixon2013open}, which were developed into \emph{framed point graphs}~\cite{kissinger2012pictures}, and \emph{hypergraphs}~\cite{bonchi2016rewriting,zanasi2017rewriting,bonchi2018rewriting,bonchi2020string}. 
As we have already established, these approaches were unsuitable in the context of digital circuits, due to the use of a compact closed or Frobenius structure.
However, our work uses some of the building blocks from these frameworks, such as the use of homeomorphism and rewriting using hypergraphs.

Unlike the explicit interfaces in our framework, the hypergraphs in the existing framework are defined via a cospan structure.
Cospans are used in categorical representations of open networks \cite{fong2015decorated,baez2020structured}, with each leg of the cospan representing inputs and outputs respectively.
This is natural in the presence of a Frobenius monoid which allows vertices to be identified arbitrarily.
A similar recipe could have been followed here by requiring the morphisms of the cospan to be injective, which would have simplified some of the proofs.
However, we considered a more elementary presentation in which one can arrive at the desired technical result without relying on avoidable mathematical concepts, which may make it more accessible to a wider audience.
Regardless, it should be easy to see how to encode an interfaced linear hypergraph $\morph{H}{m}{n}$ as a cospan of regular hypergraphs $m \to \htrim{H} \leftarrow n$.

\subsection{Further work}
In this paper we have revisited and solidified the mathematical foundation upon which the graphical language for digital circuits of~\cite{ghica2016categorical} is based.
This graphical language opens up numerous avenues, such as additional opportunities for partial evaluation or allowing us to reason with circuits not handled well by traditional methods, such as cyclic combinational circuits \cite{malik1994analysis}.

A potential next step is to refine the existing categorical semantics of digital circuits and identify any missing axioms.
Indeed, for us to have a complete diagrammatic semantics the underlying categorical framework must of course also be complete! 
Instantiating the categorical semantics to a concrete category~\cite{ghica2017diagrammatic} pointed towards such additional axioms, most notably regarding \emph{unproductive circuits}, as mentioned in the case study.
By identifying these axioms we can present our diagrammatic semantics as a complete package for reasoning about digital circuits.

Another natural future avenue is that of \emph{automating} the graph rewrites.
While it may be simple to identify potential redexes by eye in small systems, in practice there may be many and the derivation procedure would be long and tedious.
Automating rewrites presents some additional issues, such as the task of choosing between different rewrites.
We can take the \emph{global trace-delay form} of \cite{ghica2017diagrammatic} and verify in our formal framework that we do in fact have confluence for our operational semantics.

Our firm, mathematically sound graphical foundation also opens up the opportunity for the development of circuit design tools, following in the footsteps of tools like such as Quantomatic~\cite{kissinger2015quantomatic}, Homotopy.io~\cite{homotopyio} or Cartographer~\cite{sobocinski2019cartographer}, with the aim of bringing the reduction-based operational semantics into a field dominated by simulation. Our approach serves to contrast and complement, not replace, the existing paradigm, and offer an alternative insight into the design and evaluation of digital circuits.

\bibliography{refs}

\appendix

\section{Axioms of STMCs}\label{sec:appendix}

\subsection{Left identity of composition}

For any interfaced linear hypergraph $\morph{F}{m}{n}$, $\id[m] \seq F \equiv F$.

\begin{center}
  \includesvg[scale=0.6]{axioms/id-l}
\end{center}

\subsubsection*{Definitions}
\begin{gather*}
  H = \id[m] \seq F \\
  E_H = E_F \qquad \vs[\interface][H] = \vs[\interface][F] \qquad \vs[e \in E_F][H] = \vs[e][F] \qquad \vt[\interface][H] = [m] \qquad \vt[e \in E_F][H] = \vt[e][F] \qquad \labels_H = \labels_F \\
  \vconnsr_H(t) = \begin{cases}
    \vconnsr_F(\proj{i}(\vt[\interface][F])) & \text{if}\ v = \proj{i}([m]) \\
    \vconnsr_F(t) & \text{otherwise}
  \end{cases}
\end{gather*}

\subsubsection*{Equivalence maps}
\begin{gather*}
  \vmaps{h}(v) = v
  \qquad
  \vmapt{h}(v) = \begin{cases}
    \proj{i}(\vt[\interface][F]) & \text{if}\ v = \proj{i}([m]) \\
    v & \text{otherwise}
  \end{cases}
  \qquad
  \emap{h}(e) = e
\end{gather*}

\subsection{Right identity of composition}

For any interfaced linear hypergraph $\morph{F}{m}{n}$, $F \seq \id[m] \equiv F$.

\begin{center}
  \includesvg[scale=0.6]{axioms/id-r}
\end{center}

\subsubsection*{Definitions}
\begin{gather*}
  H = F \seq \id[n] \\
  E_H = E_F \qquad \vs[\interface][H] = [m] \qquad \vs[e \in E_F][H] = \vs[e][F] \qquad \vt[\interface][H] = \vt[\interface][F] \qquad \vt[e \in E_F][H] = \vt[e][F] \qquad \labels_H = \labels_F \\
  \vconnsr_H(v) = \begin{cases}
    \proj{i}([n]) & \text{if}\ \vconnsr_F(v) = \proj{i}(\vs[\interface][F]) \\
    \vconnsr_F & \text{otherwise}
  \end{cases}
\end{gather*}

\subsubsection*{Equivalence maps}
\begin{gather*}
  \vmaps{h}(v) = v 
  \qquad
  \vmaps{h}(v) = \begin{cases}
    \proj{i}([n]) & \text{if}\ v = \proj{i}(\vs[\interface]) \\
    v & \text{otherwise}
  \end{cases}
  \qquad
  \emap{h}(e) = e
\end{gather*}

\subsection{Associativity of composition}

For any interfaced linear hypergraphs $\morph{F}{m}{n}$, $\morph{G}{n}{p}$ and $\morph{H}{p}{q}$, $(F \seq G) \seq H \equiv F \seq (G \seq H)$.

\begin{center}
  \includesvg[scale=0.6]{axioms/assoc}
\end{center}

\subsubsection*{Definitions}
\begin{gather*}
  K = (F \seq G) \seq H \\ 
  E_K = E_F + E_G + E_H \qquad \labels_K = \labels_F + \labels_G + \labels_H \\
  \vs_K = \vs_F - \vs[\interface][F] + \vs_G - \vs[\interface][G] + \vs_H \qquad \vt_K = \vt_F + \vt_G - \vt[\interface][G] + \vt_H - \vt[\interface][H] \\
  \vconnsr_K(v) = \begin{cases}
    \vconnsr_H(\proj{j}(\vs[\interface][H]) & \text{if}\ \vconnsr_F(v) = \proj{i}(\vs[\interface][F]) \wedge \vconnsr_G(\proj{i}(\vt[\interface][G])) = \proj{j}(\vs[\interface][G]) \\
    \vconnsr_H(\proj{i}(\vs[\interface][H])) & \text{if}\ \vconnsr_G(v) = \proj{i}(\vs[\interface][G]) \\
    \vconnsr_G(\proj{i}(\vs[\interface][G])) & \text{if}\ \vconnsr_F(v) = \proj{i}(\vs[\interface][F]) \\
    \vconnsr_F + \vconnsr_G + \vconnsr_H & \text{otherwise}  
  \end{cases}
\end{gather*}

\noindent The definition of $L = F \seq (G \seq H)$ is identical.

\subsubsection*{Equivalence maps}
\begin{gather*}
  \vmaps{h} = \id \qquad \vmapt{h} = \id \qquad \emap{h} = \id
\end{gather*}

\subsection{Left identity of tensor}

For any interfaced linear hypergraph $\morph{F}{m}{n}$, $\id[0] \tensor F \equiv F$.

\begin{center}
  \includesvg[scale=0.6]{axioms/id-tensor-l}
\end{center}

\subsubsection*{Definitions}
\begin{gather*}
  H = \id[0] \tensor F \\
  E_H = E_F \qquad \vs_H = F \qquad \vt_H = F \qquad \vconnsr_H = \vconnsr_F \qquad \labels_H = \labels_F
\end{gather*}

\subsubsection*{Equivalence maps}
\begin{gather*}
  \vmaps{h} = \id \qquad \vmapt{h} = \id \qquad \emap{h} = \id
\end{gather*}

\subsection{Right identity of tensor}

For any interfaced linear hypergraph $\morph{F}{m}{n}$, $F \tensor \id[0] \equiv F$.

\begin{center}
  \includesvg[scale=0.6]{axioms/id-tensor-r}
\end{center}

\subsubsection*{Definitions}
\begin{gather*}
  H = F \tensor \id[0] \\
  E_H = E_F \qquad \vs_H = F \qquad \vt_H = F \qquad \vconnsr_H = \vconnsr_F \qquad \labels_H = \labels_F
\end{gather*}

\subsubsection*{Equivalence maps}
\begin{gather*}
  \vmaps{h} = \id \qquad \vmapt{h} = \id \qquad \emap{h} = \id
\end{gather*}

\subsection{Associativity of tensor}

For any interfaced linear hypergraphs $\morph{F}{m}{n}$, $\morph{G}{p}{q}$ and $\morph{H}{r}{s}$, $F \tensor (G \tensor H) \equiv (F \tensor G) \tensor H$.

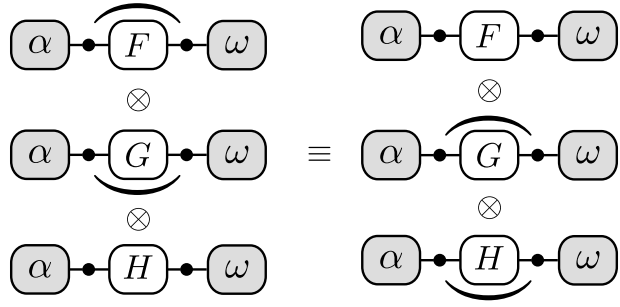
\begin{center}
  \includesvg[scale=0.6]{axioms/assoc-tensor}
\end{center}

\subsubsection*{Definitions}
\begin{gather*}
  K = F \tensor (G \tensor H) \\
  E_K = E_F + E_G + E_H \qquad \labels_H = \labels_F + \labels_G + \labels_H \\ 
  \vs[e \in E_F][K] = \vs[e][F] \qquad \vs[e \in E_G][K] = \vs[e][G] \qquad \vs[e \in E_H][K] = \vs[e][H] \qquad \vs[\interface][K] = \vs[\interface][F] + \vs[\interface][G] + \vs[\interface][H] \\
  \vt[e \in E_F][K] = \vt[e][F] \qquad \vt[e \in E_G][K] = \vt[e][G] \qquad \vt[e \in E_H][K] = \vt[e][H] \qquad \vt[\interface][K] = \vt[\interface][F] + \vt[\interface][G] + \vt[\interface][H]
\end{gather*}

\noindent The definition of $L = (F \tensor G) \tensor H$ is identical.

\subsubsection*{Equivalence maps}
\begin{gather*}
  \vmaps{h} = \id \qquad \vmapt{h} = \id \qquad \emap{h} = \id
\end{gather*}

\subsection{Bifunctoriality of tensor I (Proposition~\ref{prop:bifunctoriality-1})}

For any $m,n \in \nat$, $\id[m] \tensor \id[n] \equiv \id_{m + n}$

\begin{center}
  \includesvg[scale=0.6]{axioms/bifunctoriality-1}
\end{center}

\subsubsection*{Definitions}
\begin{gather*}
  H = \id[m] \tensor \id[n] \\
  E_H = \emptyset \qquad \labels_H = \emptyset \qquad \vs[\interface][H] = [m] + [n] \qquad \vt[\interface][H] = [m] + [n] \qquad \vconnsr(\proj{i}(\vt[\interface][H])) = \proj{i}(\vs[\interface][H]) \\
  L = \id_{m + n} \\
  E_H = \emptyset \qquad \labels_H = \emptyset \qquad \vs[\interface][H] = [m + n] \qquad \vt[\interface][H] = [m + n] \qquad \vconnsr(\proj{i}(\vt[\interface][H])) = \proj{i}(\vs[\interface][H]) 
\end{gather*}

\subsection*{Equivalence maps}
\begin{gather*}
  \vmaps{h}(s) = \begin{cases}
    \proj{i}([m + n]) & \text{if}\ s = \proj{i}([m]) \\
    \proj{i + m}([m + n]) & \text{if}\ s = \proj{i}{[n]}
  \end{cases}
  \qquad
  \vmapt{h}(t) = \begin{cases}
    \proj{i}([m + n]) & \text{if}\ t = \proj{i}([m]) \\
    \proj{i + m}([m + n]) & \text{if}\ t = \proj{i}{[n]}
  \end{cases}
  \qquad
  \emap{h} = \emptyset
\end{gather*}

\subsection{Bifunctoriality of tensor II (Proposition~\ref{prop:bifunctoriality-2})}

For any interfaced linear hypergraphs $\morph{F}{m}{n}$, $\morph{G}{r}{s}$, $\morph{H}{n}{p}$ and $\morph{K}{s}{t}$, $F \tensor G \seq H \tensor K \equiv (F \seq H) \tensor (G \seq K)$.

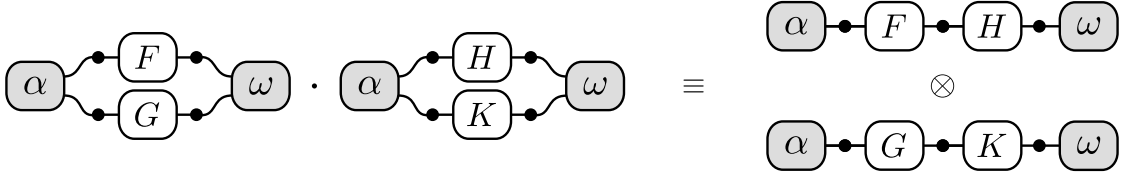
\begin{center}
  \includesvg[scale=0.6]{axioms/bifunctoriality-2}
\end{center}

\subsubsection*{Definitions}
\begin{gather*}
  L = F \tensor G \seq H \tensor K \\
  E_L = E_F + E_G + E_H + E_K \qquad 
  \labels_L = \labels_F + \labels_G + \labels_H + \labels_L \\
  \vs[\interface][L] = \vs[\interface][H] + \vs[\interface][L] \quad\,
  \vs[e \in E_F][L] = \vs[e][F] \quad\,
  \vs[e \in E_G][L] = \vs[e][G] \quad\,
  \vs[e \in E_H][L] = \vs[e][H] \quad\,
  \vs[e \in E_K][L] = \vs[e][K] \\
  \vt[\interface][L] = \vt[\interface][F] + \vt[\interface][G] \quad\,
  \vt[e \in E_F][L] = \vt[e][F] \quad\,
  \vt[e \in E_G][L] = \vt[e][G] \quad\,
  \vt[e \in E_H][L] = \vt[e][H] \quad\,
  \vt[e \in E_K][L] = \vt[e][K] \\
  \vconnsr(t) = \begin{cases}
    \vconnsr(\proj{i}(\vt[\interface][H])) & \text{if}\ \vconnsr_F(t) = \proj{i}(\vs[\interface][F]) \\
    \vconnsr(\proj{i}(\vt[\interface][K])) & \text{if}\ \vconnsr_G(t) = \proj{i}(\vs[\interface][G]) \\
    (\vconnsr_F + \vconnsr_G + \vconnsr_H + \vconnsr_K)(t) & \text{otherwise}
  \end{cases}
\end{gather*}

\noindent The definition of $A = (F \seq H) \tensor (G \tensor K)$ is identical.

\subsection*{Equivalence maps}
\begin{gather*}
  \vmaps{h} = \id \qquad \vmapt{h} = \id \qquad \emap{h} = \id
\end{gather*}

\subsection{Naturality of swap (Proposition~\ref{prop:naturality-swap})}

For any interfaced linear hypergraphs $\morph{F}{m}{n}$ and $\morph{G}{p}{q}$, $F \tensor G \seq \swap{n}{q} \equiv \swap{m}{p} \seq G \tensor F$.

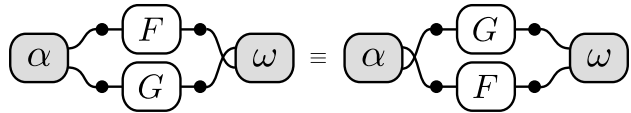
\begin{center}
  \includesvg[scale=0.6]{axioms/naturality-swap}
\end{center}

\subsubsection*{Definition}
We use the `alternate' representation of the swap hypergraph defined in Lemma~\ref{lem:alternate-swap}.

\begin{gather*}
  H = F \tensor G \seq \swap{n}{q} \\
  E_H = E_F + E_G \qquad
  \vs[\interface][H] = [n] + [q] \qquad
  \vs[e \in E_F][H] = \vs[e][F] \qquad
  \vs[e \in E_G][H] = \vs[e][G] \\
  \labels_H = \labels_F + \labels_G \qquad \vt[\interface][H] = \vt[\interface][F] \qquad
  \vt[e \in E_F][H] = \vt[e][F] \qquad
  \vt[e \in E_G][H] = \vt[e][G] \\
  \vconnsr_H(t) = \begin{cases}
    \proj{i}([n]) & \text{if}\ \vconnsr_F(t) = \proj{i}(\vs[\interface][F]) \\
    \proj{i}([q]) & \text{if}\ \vconnsr_G(t) = \proj{i}(\vs[\interface][G]) \\
    (\vconnsr_F + \vconnsr_G)(t) & \text{otherwise}
  \end{cases} \\ \\
  K = \swap{m}{p} \seq F \tensor G \\
  E_H = E_F + E_G \qquad
  \vs[\interface][H] = \vs[\interface][F] \qquad
  \vs[e \in E_F][H] = \vs[e][F] \qquad
  \vs[e \in E_G][H] = \vs[e][G] \\
  \labels_H = \labels_F + \labels_G \qquad 
  \vt[\interface][H] = [m] + [p] \qquad
  \vt[e \in E_F][H] = \vt[e][F] \qquad
  \vt[e \in E_G][H] = \vt[e][G] \\
  \vconnsr_H(t) = \begin{cases}
    \vconnsr_F(\proj{i}(\vt[\interface][F]) & \text{if}\ t = \proj{i}([m]) \\
    \vconnsr_G(\proj{i}(\vt[\interface][G]) & \text{if}\ t = \proj{i}([p]) \\
    (\vconnsr_F + \vconnsr_G)(t) & \text{otherwise}
  \end{cases} \\
\end{gather*}

\subsection*{Equivalence maps}
\begin{gather*}
  \vmaps{h}(s) = \begin{cases}
    \proj{i}(\vs[\interface][F]) & \text{if}\ s = \proj{i}([m]) \\
    \proj{i}(\vs[\interface][G]) & \text{if}\ s = \proj{i}([p]) \\
    s & \text{otherwise}
  \end{cases}
  \qquad
  \vmapt{h}(t) = \begin{cases}
    \proj{i}([m]) & \text{if}\ s = \proj{i}(\vt[\interface][F]) \\
    \proj{i}([p]) & \text{if}\ s = \proj{i}(\vt[\interface][G]) \\
    t & \text{otherwise}
  \end{cases}
  \qquad
  \emap{h} = \id
\end{gather*}

\subsection{Hexagon axiom}

For any $m,n,p \in \nat$, $\swap{m}{n} \tensor \id[p] \seq \id[n] \tensor \swap{m}{p} \equiv \swap{m}{n+p}$.

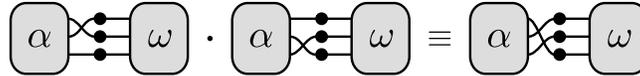
\begin{center}
  \includesvg[scale=0.6]{axioms/hexagon}
\end{center}

\subsubsection*{Definitions}
\begin{gather*}
  H = \swap{m}{n} \tensor \id[p] \seq \id[n] \tensor \swap{m}{p} \\
  E_H = \emptyset \qquad \labels_H = \emptyset \qquad
  \vs[\interface][H] = [n]_s + [p]_s + [m]_s \qquad \vt[\interface][H] = [m]_t + [n]_t + [p]_t \\
  \vconnsr_H(\proj{i}([m]_t) = \proj{i}([m]_s) \qquad \vconnsr_H(\proj{i}([n]_t) = \proj{i}([n]_s) \qquad \vconnsr_H(\proj{i}([p]_t) = \proj{i}([p]_s) \\\\
  K = \swap{m}{n+p} \\
  E_K = \emptyset \qquad \labels_K = \emptyset \qquad
  \vs[\interface][K] = [n + p]_s + [m]_s \qquad \vt[\interface][K] = [m]_t + [n + p]_t \\
  \vconnsr_K(\proj{i}([m]_t)) = \proj{i}([m]_s) \qquad \vconnsr_K(\proj{i}([n+p]_t)) = \proj{i}([n+p]_s)
\end{gather*}

\subsubsection*{Equivalence maps}
\begin{gather*}
  \vmaps{h}(v) = \begin{cases}
    \proj{i}([n + p]_s) & \text{if}\ v = \proj{i}([n]_s) \\
    \proj{i + n}([n + p]_s) & \text{if}\ v = \proj{i}([p]_s) \\
    \proj{i}([m]_s) & \text{if}\ v = \proj{i}([m]_s)
  \end{cases} \qquad
  \vmapt{h}(v) = \begin{cases}
    \proj{i}([n + p]_t) & \text{if}\ v = \proj{i}([n]_t) \\
    \proj{i + n}([n + p]_t) & \text{if}\ v = \proj{i}([p]_t) \\
    \proj{i}([m]_t) & \text{if}\ v = \proj{i}([m]_t)
  \end{cases} \qquad
  \emap{h} = \emptyset
\end{gather*}

\subsection{Self-invertability}

For any $m,n \in \nat$, $\swap{m}{n} \seq \swap{n}{m} \equiv \id_{m + n}$.

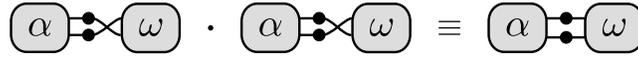
\begin{center}
  \includesvg[scale=0.6]{axioms/self-inverse}
\end{center}

\subsubsection*{Definitions}
\begin{gather*}
  H = \swap{m}{n} \seq \swap{n}{m} \\
  E_H = \emptyset \qquad \labels_H = \emptyset \qquad \vs[\interface][H] = [m]_s + [n]_s \qquad \vt[\interface][H] = [m]_t + [n]_t \\
  \vconnsr_H(\proj{i}([m]_t)) = \proj{i}([m]_s) \qquad \vconnsr_H(\proj{i}([n]_t)) = \proj{i}([n]_s)
\end{gather*}

\noindent The definition of $\id_{m + n}$ is identical.

\subsubsection*{Equivalence maps}
\begin{gather*}
  \vmaps{h} = \id \qquad \vmapt{h} = \id \qquad \emap = \emptyset
\end{gather*}

\subsection{Tightening}

For any interfaced linear hypergraphs $\morph{F}{x + m}{x + n}$, $\morph{G}{p}{m}$ and $\morph{H}{n}{q}$, \[\trace{x}{\id[x] \tensor G \seq F \seq \id[x] \tensor H} \equiv G \seq \trace{x}{F} \seq H.\]

\begin{center}
  \includesvg[scale=0.6]{axioms/tightening}
\end{center}

\noindent This proof is by induction on $x$.

\subsubsection*{Zero case: $x = 0$}
\[\trace{0}{\id[0] \tensor G \seq F \seq \id[0] \tensor H} \equiv G \seq \trace{0}{F} \seq H\]
\begin{align*}
  \trace{0}{\id[0] \tensor G \seq F \seq \id[0] \tensor H} & \equiv \id[0] \tensor G \seq F \seq \id[0] \tensor H & \text{definition of trace} \\
  & \equiv G \seq F \seq \id[0] \tensor H & \text{left identity of tensor} \\
  & \equiv G \seq F \seq H & \text{left identity of tensor} \\
  & \equiv G \seq \trace{0}{F} \seq H & \text{definition of trace}
\end{align*}

\subsubsection*{Base case: $x = 1$}
\[\trace{1}{\id[1] \tensor G \seq F \seq \id[1] \tensor H} \equiv G \seq \trace{1}{F} \seq H\]
\subsubsection*{Definitions}
\begin{gather*}
  K = \trace{1}{\id[1] \tensor G \seq F \seq \id[1] \tensor H} \\
  E_K = E_F + E_G + E_H \qquad E[{\id}] = \{e_{\id}\} \qquad \labels_K = \labels_F + \labels_G + \labels_H \\
  \vs[\interface][K] = \vs[\interface][G] \qquad \vs[e_{\id}][K] = [1]_s \qquad \vs[e \in E_F][K] = \vs[e][F] \qquad \vs[e \in E_G][K] = \vs[e][G] \qquad \vs[e \in E_H][K] = \vs[e][H] \\
  \vt[\interface][K] = \vt[\interface][H] \qquad \vt[e_{\id}][K] = [1]_t \qquad \vt[e \in E_F][K] = \vt[e][F] \qquad \vt[e \in E_G][K] = \vt[e][G] \qquad \vt[e \in E_H][K] = \vt[e][H] \\
  \vconnsr_K(t) = \begin{cases}
    \vconnsr_H(\proj{i-1}(\vt[\interface][H])) & \text{if}\ t = \proj{0}([1]_t) \wedge \vconnsr_F(\proj{0}(\vt[\interface][F])) = \proj{i}(\vs[\interface][F]) \\
    \vconnsr_F(\proj{0}(\vt[\interface][F])) & \text{if}\ t = \proj{0}([1]_t) \\ 
    \proj{0}([1]_s) & \text{if}\ \vconnsr_F(t) = \proj{0}(\vs[\interface][F]) \\
    \vconnsr_H(\proj{i-1}(\vt[\interface][H])) & \text{if}\ \vconnsr_F(t) = \proj{i}(\vs[\interface][F]) \\
    \proj{0}([1]_s) & \text{if}\ \vconnsr_G(t) = \proj{i}(\vs[\interface][G]) \wedge \vconnsr_F(\proj{i}(\vt[\interface][F])) = \proj{0}(\vs[\interface][F]) \\
    \vconnsr_F(\proj{i+1}(\vt[\interface][F])) & \text{if}\ \vconnsr_G(t) = \proj{i}(\vs[\interface][G]) \\
    (\vconnsr_F + \vconnsr_G + \vconnsr_H)(v) & \text{otherwise}
  \end{cases} \\ \\ 
  L = G \seq \trace{1}{F} \seq H \\
  E_L = E_F + E_G + E_H \qquad E[{\id}] = \{e_{\id}\} \qquad \labels_L = \labels_F + \labels_G + \labels_H \\
  \vs[\interface][L] = \vs[\interface][G] \qquad \vs[e_{\id}][L] = \{\proj{0}(\vs[\interface][F])\} \qquad \vs[e \in E_F][L] = \vs[e][F] \qquad \vs[e \in E_G][L] = \vs[e][G] \qquad \vs[e \in E_H][L] = \vs[e][H] \\
  \vt[\interface][L] = \vt[\interface][H] \qquad \vt[e_{\id}][K] = \{\proj{0}(\vt[\interface][F])\} \qquad \vt[e \in E_F][L] = \vt[e][F] \qquad \vt[e \in E_G][L] = \vt[e][G] \qquad \vt[e \in E_H][L] = \vt[e][H] \\
  \vconnsr_L(t) = \begin{cases}
    \vconnsr_H(\proj{i-1}(\vt[\interface][H])) & \text{if}\ \vconnsr_F(t) = \proj{i}(\vs[\interface][F]) \\
    \vconnsr_F(\proj{i+1}(\vt[\interface][F])) & \text{if}\ \vconnsr_G(t) = \proj{i}(\vs[\interface][G]) \\
    (\vconnsr_F + \vconnsr_G + \vconnsr_H)(v) & \text{otherwise}
  \end{cases}
\end{gather*}

\subsubsection*{Equivalence maps}
\begin{gather*}
  \vmaps{h}(v) = \begin{cases}
    \proj{0}(\vs[\interface][F]) & \text{if}\ v = \proj{0}([1]_s) \\
    v & \text{otherwise} 
  \end{cases} \qquad \vmapt{h} = \begin{cases}
    \proj{0}(\vt[\interface][F]) & \text{if}\ v = \proj{0}([1]_t) \\
    v & \text{otherwise} 
  \end{cases}  \qquad \emap{h} = \id
\end{gather*}

\subsubsection*{Inductive case: $x = k + 1$ for $k > 1$}
\[\trace{k+1}{\id_{k+1} \tensor G \seq F \seq \id_{k+1} \tensor H} \equiv G \seq \trace{k+1}{F} \seq H\]
\begin{align*}
  \trace{k + 1}{\id[k + 1] \tensor G \seq F \seq \id[k + 1] \tensor H} & \equiv \trace{1}{\trace{k}{\id[k + 1] \tensor G \seq F \seq \id[k + 1] \tensor H}} & \text{definition of trace}\\
  & \equiv \trace{1}{\trace{k}{k \tensor \id[1] \tensor G \seq F \seq k \tensor \id[1] \tensor H}} & \text{bif of tensor I} \\
  & \equiv \trace{1}{\id[1] \tensor G \seq \trace{k}{F} \seq \id[1] \tensor H} & \text{IH (1)} \\
  & \equiv G \seq \trace{1}{\trace{k}{F}} \seq H & \text{base case (2)} \\
  & \equiv G \seq \trace{k + 1}{F} \seq H & \text{definition of trace}
\end{align*}

\noindent At (1), we use our inductive hypothesis with $m = 1 + m$ and $n = 1 + n$, since it applies for any $m, n \in \nat$. 
At (2), we can use our base case since $\trace{k}{F}$ has type $1 + m \to 1 + n$.

\subsection{Superposing}

For any interfaced linear hypergraph $\morph{F}{x + m}{x + n}$ and $n \in \nat$, \[\trace{x}{F \tensor \id[n]} \equiv \trace{x}{F} \tensor \id[n].\]

\begin{center}
  \includesvg[scale=0.6]{axioms/superposing}
\end{center}

\noindent This proof is by induction on $x$.

\subsubsection*{Zero case: $x = 0$}
\[\trace{0}{F \tensor \id[n]} \equiv \trace{0}{F} \tensor \id[n]\]
This follows immediately by definition of trace.

\subsubsection*{Base case: $x = 1$}
\[\trace{1}{F \tensor \id[n]} \equiv \trace{1}{F} \tensor \id[n]\]
\subsubsection*{Definitions}
\begin{gather*}
  H = \trace{x}{F \tensor G} \\
  E_H = E_F + E_G \qquad E_H[\id] = \{e_{\id}\} \qquad \labels_H = \labels_F + \labels_G \\
  \vs[\interface][H] = \vs[\interface][F] - \proj{0}(\vs[\interface][F]) + [n]_s \qquad \vs[e_{\id}] = \proj{0}(\vs[\interface][F]) \qquad \vs[e \in E_F][H] = \vs[e][F] \\
  \vt[\interface][H] = \vt[\interface][F] - \proj{0}(\vt[\interface][F]) + [n]_t \qquad \vt[e_{\id}] = \proj{0}(\vt[\interface][F]) \qquad \vt[e \in E_F][H] = \vt[e][F] \\
  \vconnsr_H(t) = \begin{cases}
    \proj{i}([n]_s) & \text{if}\ t = \proj{i}([n]_t) \\
    \vconnsr_F(t) & \text{otherwise}
  \end{cases}  
\end{gather*}

\noindent The definition of $L = \trace{x}{F} \tensor G$ is identical.

\subsubsection*{Equivalence maps}
\begin{gather*}
  \vmaps{h} = \id \qquad \vmapt{h} = \id \qquad \emap{h} = \id
\end{gather*}

\subsubsection*{Inductive case: $x = k + 1$ for $k > 1$}
\[\trace{k+1}{F \tensor G} \equiv \trace{k+1}{F} \tensor G\]
\begin{align*}
  \trace{k+1}{F \tensor G} & \equiv \trace{1}{\trace{k}{F \tensor G}} & \text{definition of trace} \\
                           & \equiv \trace{1}{\trace{k}{F} \tensor G} & \text{IH (1)} \\
                           & \equiv \trace{1}{\trace{k}{F}} \tensor G & \text{base case (2)} \\
                           & \equiv \trace{k+1}{F} \tensor G & \text{definition of trace}
\end{align*}

\noindent
At (1), $\morph{\trace{1}{F}}{k + m}{k + n}$ so we can apply our inductive hypothesis. 
At (2), we observe that as $\morph{F}{k + 1 + m}{k + 1 + n}$ then $\trace{k}{F}{1 + m}{1 + n}$, so we can use our base case. 
Therefore $\trace{k + 1}{F \tensor G} = \trace{k + 1}{F} \tensor G$.
Therefore for any $x \in \nat$, $\trace{x}{F \tensor G} = \trace{x}{F} \tensor G$.

\subsection{Yanking}

For any $x \in \nat$, $\trace{x}{\swap{x}{x}} \equiv \id[x]$.
\begin{center}
  \includesvg[scale=0.6]{axioms/yanking}
\end{center}
This proof is by induction on $x$.

\subsubsection*{Zero case: $x = 0$}
\[\trace{0}{\swap{0}{0}} \equiv \id[0]\]
This follows immediately by definition of trace and symmetry.

\subsubsection*{Base case: $x = 1$}
\[\trace{1}{\swap{1}{1}} \equiv \id[1]\]
\subsubsection*{Definitions}
\begin{gather*}
  H = \trace{1}{\swap{1}{1}} \\
  E_H = \emptyset \qquad E_H[\id] = \{e_{\mf{id}}\} \qquad \labels_H = \emptyset \\ 
  \vs[\interface][H] = [1]_{s1} \qquad \vs[e_\mf{id}][H] = [1]_{s2} \qquad \vt[\interface][H] = [1]_{t1} \qquad \vt[e_\mf{id}][H] = [1]_{t2} \\ 
  \vconnsr_H(\proj{0}([1]_{t1}) = \proj{0}([1]_{s2}) \qquad \vconnsr_H(\proj{0}([1]_{t2}) = \proj{0}([1]_{s1})
\end{gather*}
By performing a smoothing we obtain the following:
\begin{gather*}
  L = \mf{smooth}(\trace{1}{\swap{1}{1}}) \\
  E_L = \emptyset \qquad E_L[\id] = \emptyset \qquad \labels_L = \emptyset \\
  \vs[\interface][L] = [1]_{s1} \qquad \vt[\interface][L] = [1]_{t1} \qquad
  \vconnsr_L(\proj{0}([1]_{t1}) = \proj{0}([1]_{s1})
\end{gather*}

\subsubsection*{Equivalence maps}
\begin{gather*}
  \vmaps{h} = \id \qquad \vmapt{h} = \id \qquad \emap{h} = \id
\end{gather*}

\subsubsection*{Inductive case: $x = k+1$ for $k > 1$}
\[\trace{k+1}{\swap{k+1}{k+1}} \equiv \id[{k+1}]\]
The inductive case begins by manipulating $\trace{k+1}{\swap{k+1}{k+1}}$ into a form with traces of only one wire.
\begin{align*}
  \trace{k + 1}{\swap{k + 1}{k + 1}} & \equiv \trace{k + 1}{k \tensor \swap{1}{n} \tensor 1 \seq \swap{k}{k} \tensor \swap{1}{1} \seq k \tensor \swap{k}{1} \tensor 1} & \text{inductive swap} \\
  & \equiv \trace{1}{\trace{k}{k \tensor \swap{1}{k} \tensor 1 \seq \swap{k}{k} \tensor \swap{1}{1} \seq k \tensor \swap{k}{1} \tensor 1}} & \text{definition of trace} \\
  & \equiv \trace{1}{\swap{1}{k} \tensor 1 \seq \trace{k}{\swap{k}{k} \tensor \swap{1}{1}} \seq \swap{k}{1} \tensor 1} & \text{tightening} \\
  & \equiv \trace{1}{\swap{1}{k} \tensor 1 \seq \trace{k}{\swap{k}{k}} \tensor \swap{1}{1} \seq \swap{k}{1} \tensor 1} & \text{superposing} \\
  & \equiv \trace{1}{\swap{1}{k} \tensor 1 \seq k \tensor \swap{1}{1} \seq \swap{k}{1} \tensor 1} & \text{IH} \\
\end{align*}
\subsubsection*{Definitions}
\begin{gather*}
  H = \trace{1}{\swap{1}{k} \tensor 1 \seq k \tensor \swap{1}{1} \seq \swap{k}{1} \tensor 1} \\
  E_H = \emptyset \qquad E_H[\id] = \{e_{\id}\} \labels = \emptyset \\
  \vs[\interface][H] = [k]_s + [1]_{s1} \qquad \vs[e_{\id}][H] = [1]_{s2} \qquad \vt[\interface][H] = [k]_t + [1]_{t1} \qquad \vt[e_{\id}][H] = [1]_{t2} \\
  \vconnsr(t) = \begin{cases}
    \proj{0}([1]_{s2}) & \text{if}\ t = \proj{0}([1]_{t1}) \\
    \proj{0}([1]_{s1}) & \text{if}\ t = \proj{0}([1]_{t2}) \\
    \proj{i}([k]_s) & \text{if}\ t = \proj{i}([k]_t)
  \end{cases} 
\end{gather*}
By performing a smoothing we obtain the following:
\begin{gather*}
  L = \mf{smooth}(\trace{1}{\swap{1}{k} \tensor 1 \seq k \tensor \swap{1}{1} \seq \swap{k}{1} \tensor 1}) \\
  E_L = \emptyset \qquad E_L[\id] = \emptyset \labels = \emptyset \\
  \vs[\interface][L] = [k]_s + [1]_{s1} \qquad \vt[\interface][L] = [k]_t + [1]_{t1} \qquad
  \vconnsr(t) = \begin{cases}
    \proj{0}([1]_{s1}) & \text{if}\ t = \proj{0}([1]_{s1}) \\
    \proj{i}([k]_s) & \text{if}\ t = \proj{i}([k]_t)
  \end{cases} 
\end{gather*}
\subsubsection*{Equivalence maps}
\begin{gather*}
  \vmaps{h} = \id \qquad \vmapt{h} = \id \qquad \emap{h} = \id
\end{gather*}

\subsection{Exchange}

For any interfaced linear hypergraph $\morph{F}{x + y + m}{x + y + n}$,
\[\trace{y}{\trace{x}{F}} \equiv \trace{x}{\trace{y}{\swap{y}{x} \tensor \id[m] \seq F \seq \swap{x}{y} \tensor \id[n]}}.\]

\begin{center}
  \includesvg[scale=0.6]{axioms/exchange}
\end{center}

\noindent This proof is by induction on $x$ and $y$.

\subsubsection*{Zero case I: $x = 0$, $y = k$}
\[\trace{k}{\trace{0}{F}} \equiv \trace{0}{\trace{k}{\swap{k}{0} \tensor \id[m] \seq F \seq \swap{0}{k} \tensor \id[n]}}\]
\begin{align*}
  \trace{k}{\trace{0}{F}} & \equiv \trace{k}{F} & \text{definition of trace}\\
                          & \equiv \trace{0}{\trace{k}{F}} & \text{definition of trace} \\
                          & \equiv \trace{0}{\trace{k}{\id_{k + m} \seq F \seq \id_{k + n}}} & \text{left/right identity} \\
                          & \equiv \trace{0}{\trace{k}{\id[k] \tensor \id[m] \seq F \seq  \id[k] \tensor \id[n]}} & \text{bifunctoriality I} \\
                          & \equiv \trace{0}{\trace{k}{\swap{k}{0} \tensor \id[m] \seq F \seq \swap{0}{k} \tensor \id[n]}} & \text{definition of swap}                          
\end{align*}

\subsubsection*{Zero case II: $x = k$, $y = 0$}
\[\trace{0}{\trace{k}{F}} \equiv \trace{k}{\trace{0}{\swap{0}{k} \tensor \id[m] \seq F \seq \swap{k}{0} \tensor \id[n]}}\]
\begin{align*}
  \trace{0}{\trace{k}{F}} & \equiv \trace{0}{\trace{k}{\id_{k + m} \seq F \seq \id_{k + n}}} & \text{left/right identity} \\
                          & \equiv \trace{0}{\trace{k}{\id[k] \tensor \id[m] \seq F \seq \id[k] \tensor \id[n]}} & \text{bifunctoriality I} \\
                          & \equiv \trace{0}{\trace{k}{\swap{0}{k} \tensor \id[m] \seq F \seq \swap{k}{0} \tensor \id[n]}} & \text{definition of swap} \\
                          & \equiv \trace{k}{\id[k] \tensor \id[m] \seq F \seq \id[k] \tensor \id[n]} & \text{definition of trace} \\
                          & \equiv \trace{k}{\trace{0}{\id[k] \tensor \id[m] \seq F \seq \id[k] \tensor \id[n]}} & \text{definition of trace}
\end{align*}

\subsubsection*{Base case: $x = 1$, $y = 1$}
\[\trace{1}{\trace{1}{F}} \equiv \trace{1}{\trace{1}{\swap{1}{1} \tensor \id[m] \seq F \seq \swap{1}{1} \tensor \id[n]}}\]
\subsubsection*{Definitions}

\begin{gather*}
  H = \trace{1}{\trace{1}{F}} \\
  E_H = E_F \qquad E_H[\id] = \{e_{\id 0}, e_{\id 1}\} \qquad \labels_H = \labels_F \qquad \vconnsr_H = \vconnsr_F\\
  \vs[\interface][H] = \vs[\interface][F] - (\proj{0}(\vs[\interface][F]) + \proj{1}(\vs[\interface][F])) \quad \vs[e \in E_F][H] = \vs[e][F] \quad \vs[e_{\id 0}] = \{\proj{0}(\vs[\interface][F])\} \quad \vs[e_{\id 1}] = \{\proj{1}(\vs[\interface][F])\} \\
  \vt[\interface][H] = \vt[\interface][F] - (\proj{0}(\vt[\interface][F]) + \proj{1}(\vt[\interface][F])) \quad \vt[e \in E_F][H] = \vt[e][F] \quad \vt[e_{\id 0}] = \{\proj{0}(\vt[\interface][F])\} \quad \vt[e_{\id 1}] = \{\proj{1}(\vt[\interface][F])\} \\\\
  L = \trace{1}{\trace{1}{\swap{1}{1} \tensor \id[m] \seq F \seq \swap{1}{1} \tensor \id[n]}} \\
  E_L = E_F \qquad E_H[\id] = \{e_{\id 0}, e_{\id 1}\} \qquad \labels_H = \labels_F \\
  \vs[\interface][H] = [n] \quad \vs[e \in E_F][H] = \vs[e][F] \quad \vs[e_{\id 0}] = [1]_{s0} \quad \vs[e_{\id 1}] = [1]_{s1} \\
  \vt[\interface][H] = [m] \quad \vt[e \in E_F][H] = \vt[e][F] \quad \vt[e_{\id 0}] = [1]_{t0} \quad \vs[e_{\id 1}] = [1]_{t1} \\
  \vconnsr_H(t) = \begin{cases}
    \vconnsr_F(\proj{1}(\vt[\interface][F])) & \text{if}\ t = [1]_{t0} \\
    \vconnsr_F(\proj{0}(\vt[\interface][F])) & \text{if}\ t = [1]_{t1} \\
    \vconnsr_F(\proj{i + 2}(\vt[\interface][F])) & \text{if}\ t = \proj{i}([m]) \\
    \proj{0}([1]_{s1}) & \text{if}\ \vconnsr_F(t) = \proj{0}(\vs[\interface][F]) \\
    \proj{0}([1]_{s0}) & \text{if}\ \vconnsr_F(t) = \proj{1}(\vs[\interface][F]) \\
    \proj{i-2}([n]) & \text{if}\ \vconnsr_F(t) = \proj{i}(\vs[\interface][F]) \\
    \vconnsr_F(t) & \text{otherwise}
  \end{cases}
\end{gather*}

\subsubsection*{Equivalence maps}
\begin{gather*}
  \vmaps{h}(v) = \begin{cases}
    \proj{0}([1]_{s0}) & \text{if}\ v = \proj{0}(\vs[\interface][F]) \\
    \proj{0}([1]_{s1}) & \text{if}\ v = \proj{1}(\vs[\interface][F]) \\
    \proj{i-2}([n]) & \text{if}\ v = \proj{i}(\vs[\interface][F]) \\
    v & \text{otherwise}
  \end{cases}
  \qquad
  \vmapt{h}(v) = \begin{cases}
    \proj{0}([1]_{t0}) & \text{if}\ v = \proj{0}(\vt[\interface][F]) \\
    \proj{0}([1]_{t1}) & \text{if}\ v = \proj{1}(\vt[\interface][F]) \\
    \proj{i-2}([n]) & \text{if}\ v = \proj{i}(\vt[\interface][F]) \\
    v & \text{otherwise}
  \end{cases} 
  \qquad
  \emap{h} = \id
\end{gather*}

\subsubsection*{Inductive case I: $x = k+1$, $y=1$}
\[\trace{1}{\trace{k+1}{F}} \equiv \trace{k+1}{\trace{1}{\swap{1}{k+1} \tensor \id[m] \seq \swap{k+1}{1} \tensor \id[m]}}\]
\begin{align*}
  \trace{1}{\trace{k+1}{F}} & \equiv \trace{1}{\trace{1}{\trace{k}{F}}} \\
                            & \qquad\qquad \text{definition of trace} \\
                            & \equiv \trace{1}{\trace{1}{\swap{1}{1} \tensor \id[m] \seq \trace{k}{F} \seq \swap{1}{1} \tensor \id[n]}} \\ 
                            & \qquad\qquad \text{base case} \\
                            & \equiv \trace{1}{\trace{1}{\trace{k}{\id[k] \tensor \swap{1}{1} \tensor \id[m] \seq F \seq \id[k] \tensor \swap{1}{1} \tensor \id[n]}}} \\ 
                            & \qquad\qquad \text{tightening} \\
                            & \equiv \trace{1}{\trace{k}{\trace{1}{\swap{1}{k} \tensor \id[1] \tensor \id[m] \seq \id[k] \tensor \swap{1}{1} \tensor \id[m] \seq F \seq \id[k] \tensor \swap{1}{1} \tensor \id[n] \seq \swap{k}{1} \tensor \id[1] \tensor \id[n]}}} \\
                            & \qquad\qquad \text{IH} \\
                            & \equiv \trace{k+1}{\trace{1}{\swap{1}{k} \tensor \id[1] \tensor \id[m] \seq \id[k] \tensor \swap{1}{1} \tensor \id[m] \seq F \seq \id[k] \tensor \swap{1}{1} \tensor \id[n] \seq \swap{k}{1} \tensor \id[1] \tensor \id[n]}} \\ 
                            & \qquad\qquad \text{definition of trace}\\
                            & \equiv \trace{k+1}{\trace{1}{(\swap{1}{k} \tensor \id[1] \seq \id[k] \tensor \swap{1}{1}) \tensor \id[m] \seq F \seq (\id[k] \tensor \swap{1}{1} \seq \swap{k}{1} \tensor \id[1])\tensor \id[n]}} \\ 
                            & \qquad\qquad \text{bifunctoriality II} \\
                            & \equiv \trace{k+1}{\trace{1}{\swap{1}{k+1} \tensor \id[m] \seq F \seq \swap{k+1}{1} \tensor \id[n]}} \\ 
                            & \qquad\qquad \text{definition of swap}
\end{align*}

\subsubsection*{Inductive case II: $x = 1$, $y = k'+1$}
\[\trace{k'+1}{\trace{1}{F}} \equiv \trace{1}{\trace{k'+1}{\swap{k'+1}{1} \tensor \id[m] \seq F \seq \swap{1}{k'+1} \tensor \id[n]}}\]
\begin{align*}
  \trace{k'+1}{\trace{1}{F}} & \equiv \trace{1}{\trace{k'}{\trace{1}{F}}} \\
                            & \qquad\qquad \text{definition of trace} \\
                            & \equiv \trace{1}{\trace{1}{\trace{k'}{\swap{k'}{1} \tensor \id[1] \tensor \id[m] \seq F \seq \swap{1}{k'} \tensor \id[1] \tensor \id[n]}}} \\
                            & \qquad\qquad \text{IH} \\
                            & \equiv \trace{1}{\trace{1}{\swap{1}{1} \tensor \id[m] \seq \trace{k'}{\swap{k'}{1} \tensor \id[1] \tensor \id[m] \seq F \seq \swap{1}{k'} \tensor \id[1] \tensor \id[n]} \seq \swap{1}{1} \tensor \id[n]}} \\
                            & \qquad\qquad \text{Base case} \\
                            & \equiv \trace{1}{\trace{1}{\trace{k'}{\id[k'] \tensor \swap{1}{1} \tensor \id[m] \seq \swap{k'}{1} \tensor \id[1] \tensor \id[m] \seq F \seq \swap{1}{k'} \tensor \id[1] \tensor \id[n] \seq \id[k'] \tensor \swap{1}{1} \tensor \id[n]}}} \\ 
                            & \qquad\qquad \text{tightening} \\ 
                            & \equiv \trace{1}{\trace{1}{\trace{k'}{(\id[k'] \tensor \swap{1}{1} \seq \swap{k'}{1} \tensor \id[1]) \tensor \id[m] \seq F \seq (\swap{1}{k'} \tensor \id[1] \seq \id[k'] \tensor \swap{1}{1}) \tensor \id[n]}}} \\ 
                            & \qquad\qquad \text{bifunctoriality II} \\
                            & \equiv \trace{1}{\trace{1}{\trace{k'}{\swap{k'+1}{1} \tensor \id[m] \seq F \seq \swap{1}{k'+1} \tensor \id[n]}}} \\ 
                            & \qquad\qquad \text{definition of swap} 
\end{align*}

\subsubsection*{Inductive case III: $x = k + 1$, $y = k' + 1$}
\[\trace{k'+1}{\trace{k+1}{F}} \equiv \trace{k+1}{\trace{k'+1}{\swap{k'+1}{k+1} \tensor \id[m] \seq F \seq \swap{k+1}{k'+1} \tensor \id[n]}}\]
\begin{align*}
  & \trace{k'+1}{\trace{k+1}{F}} \\
  & \equiv \trace{k'+1}{\trace{1}{\trace{k}{F}}} \\
                               & \qquad\qquad \text{definition of trace} \\
                               & \equiv \trace{1}{\trace{k'+1}{\swap{k'+1}{1} \tensor \id[m] \seq \trace{k}{F} \seq \swap{1}{k'+1} \tensor \id[n]}} \\
                               & \qquad\qquad \text{inductive case II} \\
                               & \equiv \trace{1}{\trace{k'+1}{\trace{k}{\id[k] \tensor \swap{k'+1}{1} \tensor \id[m] \seq F \seq \id[k] \tensor \swap{1}{k'+1} \tensor \id[n]}}} \\
                               & \qquad\qquad \text{tightening} \\
                               & \equiv \trace{1}{\trace{1}{\trace{k'}{\trace{k}{\id[k] \tensor \swap{k'+1}{1} \tensor \id[m] \seq F \seq \id[k] \tensor \swap{1}{k'+1} \tensor \id[n]}}}} \\
                               & \qquad\qquad \text{definition of trace} \\
                               & \equiv \trace{1}{\trace{1}{\trace{k}{\trace{k'}{\swap{k'}{k} \tensor \id_{1+1+m} \seq \id[k] \tensor \swap{k'+1}{1} \tensor \id[m] \seq F \seq \id[k] \tensor \swap{1}{k'+1} \tensor \id[n] \seq \swap{k}{k'} \tensor \id_{1+1+n}}}}} \\
                               & \qquad\qquad \text{IH} \\
                               & \equiv \trace{1}{\trace{k+1}{\trace{k'}{\swap{k'}{k} \tensor \id_{1+1+m} \seq \id[k] \tensor \swap{k'+1}{1} \tensor \id[m] \seq F \seq \id[k] \tensor \swap{1}{k'+1} \tensor \id[n] \seq \swap{k}{k'} \tensor \id_{1+1+n}}}} \\
                               & \qquad\qquad \text{definition of trace} \\
                               & \equiv \trace{1}{\trace{k+1}{\trace{k'}{(\swap{k'}{k} \tensor \id_{1+1} \seq \id[k] \tensor \swap{k'+1}{1}) \tensor \id_{m} \seq F \seq (\id[k] \tensor \swap{1}{k'+1} \seq \swap{k}{k'} \tensor \id_{1+1}) \tensor \id_{n}}}} \\
                               & \qquad\qquad \text{bifunctoriality I/II} \\
                               & \equiv \trace{k+1}{\trace{1}{\swap{1}{k+1} \tensor \id[m] \seq \trace{k'}{(\swap{k'}{k} \tensor \id_{1+1} \seq \id[k] \tensor \swap{k'+1}{1}) \tensor \id_{m} \seq F \seq (\id[k] \tensor \swap{1}{k'+1} \seq \swap{k}{k'} \tensor \id_{1+1}) \tensor \id_{n}} \seq \swap{k+1}{1} \tensor \id[m]}} \\
                               & \qquad\qquad \text{inductive case I} \\
                               & \equiv \text{Tr}^{k+1}(\text{Tr}^{1}(\text{Tr}^{k'}(\id_{k'} \tensor \swap{1}{k+1} \tensor \id[m] \seq (\swap{k'}{k} \tensor \id_{1+1} \seq \id[k] \tensor \swap{k'+1}{1}) \tensor \id_{m} \seq F \,\seq \\ &\qquad\qquad \text{tightening} \qquad\qquad\qquad\qquad\qquad\qquad\qquad (\id[k] \tensor \swap{1}{k'+1} \seq \swap{k}{k'} \tensor \id_{1+1}) \tensor \id_{n} \seq \id_{k'} \tensor \swap{k+1}{1} \tensor \id[m]))) \\
                               & \equiv \text{Tr}^{k+1}(\text{Tr}^{k' + 1}(\id_{k'} \tensor \swap{1}{k+1} \tensor \id[m] \seq (\swap{k'}{k} \tensor \id_{1+1} \seq \id[k] \tensor \swap{k'+1}{1}) \tensor \id_{m} \seq F \,\seq \\ &\qquad\qquad \text{definition of trace} \qquad\qquad\qquad\qquad\qquad\quad (\id[k] \tensor \swap{1}{k'+1} \seq \swap{k}{k'} \tensor \id_{1+1}) \tensor \id_{n} \seq \id_{k'} \tensor \swap{k+1}{1} \tensor \id[n])) \\
                               & \equiv \text{Tr}^{k+1}(\text{Tr}^{k' + 1}((\id_{k'} \tensor \swap{1}{k+1} \seq \swap{k'}{k} \tensor \id_{1+1} \seq \id[k] \tensor \swap{k'+1}{1}) \tensor \id_{m} \seq F \seq (\id[k] \tensor \swap{1}{k'+1} \seq \swap{k}{k'} \tensor \id_{1+1} \seq \id_{k'} \tensor \swap{k+1}{1}) \tensor \id[n])) \\ &\qquad\qquad \text{bifunctoriality II} \\
                               & \equiv \text{Tr}^{k+1}(\text{Tr}^{k' + 1}(\swap{k'+1}{k+1} \tensor \id_{m} \seq F \seq \swap{k+1}{k'+1} \tensor \id[n])) \\ &\qquad\qquad \text{definition of swap} 
\end{align*}
\end{document}